\documentclass[12pt]{amsart}
\usepackage{verbatim, latexsym, amssymb, amsmath,color}
\usepackage{enumitem}
\usepackage{epsfig}

\usepackage{geometry}
\usepackage{hyperref}
\usepackage{setspace}

\newtheorem{thm}{Theorem}[section]
\newtheorem{prop}[thm]{Proposition}
\newtheorem{lem}[thm]{Lemma}
\newtheorem{cor}[thm]{Corollary}
\theoremstyle{definition}
\newtheorem{defn}{Definition}[section]
\newtheorem{exe}{Example}[section]

\theoremstyle{remark}
\newtheorem{rem}{Remark}[section]

\begin{document}
\title[Invariant probability measures from pseudoholomorphic curves II]{Invariant probability measures from pseudoholomorphic curves II: Pseudoholomorphic curve constructions}

\author{Rohil Prasad}
\address{Department of Mathematics\\Princeton
University\\Princeton, NJ 08544}
\thanks{This material is based upon work supported by the National Science Foundation Graduate Research Fellowship Program under Grant No. DGE-1656466.}
\email{rrprasad@math.princeton.edu}

\begin{abstract}
In the previous work, we introduced a method for constructing invariant probability measures of a large class of non-singular volume-preserving flows on closed, oriented odd-dimensional smooth manifolds with pseudoholomorphic curve techniques from symplectic geometry. The technique requires existence of certain pseudoholomorphic curves satisfying some weak assumptions. In this work, we appeal to Gromov-Witten theory and Seiberg-Witten theory to construct large classes of examples where these pseudoholomorphic curves exist. Our argument uses neck stretching along with new analytical tools from Fish-Hofer's work on feral pseudoholomorphic curves. 
\end{abstract}

\maketitle

\tableofcontents

\section{Introduction}

This paper is a companion paper to the work \cite{prequel} by the author. 

The work \cite{prequel} presents a method for constructing invariant probability measures of the autonomous flow associated to a \emph{framed Hamiltonian structure} $\eta = (\lambda, \omega)$ on a closed, oriented smooth manifold $M$ of dimension $2n+1$ for $n \geq 1$. 

\begin{defn} \label{defn:framedHamiltonianStructures}
A \textbf{framed Hamiltonian structure} on $M$ is the datum of a $1$-form $\lambda$ and a two-form $\omega$ such that:
\begin{enumerate}
    \item $\lambda \wedge \omega^n > 0$,
    \item $d\omega = 0$.
\end{enumerate}

We call a framed Hamiltonian structure \emph{exact} if $\omega$ is exact. The datum of $M$ along with a choice of framed Hamiltonian structure $\eta = (\lambda, \omega)$ is called a \textbf{framed Hamiltonian manifold}.
\end{defn}

A framed Hamiltonian manifold $M$ with framed Hamiltonian structure $\eta = (\lambda, \omega)$ is denoted by the tuple $(M, \eta)$. The autonomous flow in question is the flow of the \emph{Hamiltonian vector field} $X$ associated to $\eta$. 

\begin{defn} \label{defn:hamiltonianVectorField}
The \textbf{Hamiltonian vector field} associated to a framed Hamiltonian structure $(\lambda, \omega)$ is the unique vector field $X$ satisfying
$$\lambda(X) \equiv 1$$
and
$$\omega(X, -) \equiv 0.$$
\end{defn}

It is a consequence of Cartan's formula for the Lie derivative and the definitions above that the Hamiltonian vector field $X$ associated to a framed Hamiltonian manifold $(M^{2n+1}, \eta = (\lambda, \omega))$ preserves the volume form
$$\text{dvol}_\eta = \lambda \wedge \omega^n.$$

Hamiltonian vector fields associated to framed Hamiltonian structures model volume-preserving vector fields arising in a variety of dynamical situations. 

Any non-singular volume-preserving vector field on a Riemannian $3$-manifold arises as the Hamiltonian vector field of a certain framed Hamiltonian structure. 

The Reeb vector field of a contact form $\lambda$ on an oriented smooth manifold $M$ of dimension $2n+1$ is merely the Hamiltonian vector field associated to the framed Hamiltonian structure $(\lambda, d\lambda)$. 

Hamiltonian vector fields associated to framed Hamiltonian structures also model the Hamiltonian flow on closed, regular energy levels of autonomous Hamiltonians in symplectic manifolds. This is the example that we are most directly concerned with in this work, and accordingly we will write it down in detail. 

\begin{exe} \label{exe:hypersurfaces}
Let $(W,\Omega)$ be a symplectic manifold of dimension $2n$. Let $M$ be a closed hypersurface in $W$, equal to $H^{-1}(0)$ for some smooth Hamiltonian function
$$H: W \to \mathbb{R}.$$

We assume that $0$ is a regular value of $H$, so $M$ is a smooth manifold. The Hamiltonian vector field $X_H$ associated to $H$ is the unique vector field on $W$ such that
$$\Omega(X_H, -) = -dH(-).$$

It is a quick exercise to show, using the Cartan formula for the Lie derivative, that the function $H$ is invariant under the flow of $X_H$, and that $X_H$ is tangent to $M$ at every point of $M$. 

Pick an almost-complex structure $J$ on $W$ \emph{compatible} with the symplectic form $\Omega$, meaning that 
$$\Omega(V, JV) > 0$$
for any nonzero tangent vector $V$ and
$$\Omega(J-, J-) = \Omega(-, -).$$

Let $\lambda$ be the one-form on $M$ with kernel $\xi = TM \cap J(TM)$. Set $\omega = \Theta|_M$.

Then $\eta = (\lambda, \omega)$ is a framed Hamiltonian structure on $M$. The associated Hamiltonian vector field $X$ can be shown to coincide with the restriction of $X_H$ to $M$. 
\end{exe}

Fix a framed Hamiltonian manifold $(M, \eta = (\lambda, \omega))$ of dimension $2n+1$ and let $X$ denote its associated Hamiltonian vector field. The main theorem (Theorem $1.4$) of \cite{prequel} uses \emph{pseudoholomorphic curves} inside the cylinder $\mathbb{R} \times M$, equipped with a special kind of almost-complex structure, to construct $X$-invariant probability measures. 

We also prove in \cite{prequel} that, under some mild assumptions on the framed Hamiltonian structure, that these probability measures are ``interesting''. That is, they are not equal to the normalized volume density
$$(\int_M \lambda \wedge \omega^n)^{-1}\lambda \wedge \omega^n$$
and may satisfy additional properties depending on the situation. These results, contained in Theorems $1.5$, $1.6$ and $1.8$ of \cite{prequel}, make use of the geometric and topological properties of pseudoholomorphic curves. 

In order to apply these theorems, we must show that the required pseudoholomorphic curves in $\mathbb{R} \times M$ exist. The main objective of this paper is to address this in the case where $M = H^{-1}(0)$ arises as a closed, regular energy level of a smooth Hamiltonian $H$ on a closed symplectic manifold $(W, \Omega)$, as in Example \ref{exe:hypersurfaces}. 

In this setting, we are able to leverage the \emph{Gromov-Witten theory} (and in four dimensions, \emph{Seiberg-Witten theory}) of the ambient manifold $W$ to construct pseudoholomorphic curves in $\mathbb{R} \times M$. The Gromov-Witten invariants are a rich package of invariants of closed symplectic manifolds, given by counting pseudoholomorphic curves subject to various conditions on the topology of the curves and their incidence with various cycles in $W$. The Seiberg-Witten invariants are a set of invariants of smooth, $4$-dimensional manifolds, which are essentially equivalent to the Gromov-Witten invariants in the symplectic case due to pioneering work of Taubes \cite{TaubesSWGR}. See \cite{IonelParker97} for additional discussion on this equivalence. 

Non-vanishing of certain Gromov-Witten invariants will ensure that pseudoholomorphic curves exist in $W$ and cross the hypersurface $M$. In the case where $W$ is four-dimensional, and under some topological restrictions on $M$ and $W$ identical to those considered in \cite{ChenWeinstein}, Taubes' non-vanishing theorem \cite{Taubes94} for Seiberg-Witten invariants allows us to conclude the same statement, with minor differences since the curve may not be connected. 

A degeneration procedure known as ``neck stretching'' around the hypersurface $M$ the allows us to conclude existence of pseudoholomorphic curves in $\mathbb{R} \times M$ from the existence of pseudoholomorphic curves in $W$. There are numerous details to discuss and complications to overcome. In particular, the fact that $M$ is not ``contact type'' or ``stable Hamiltonian'' renders much of the standard pseudoholomorphic curve and neck stretching analysis in the symplectic geometry literature unusable, and we must rely on new analytical estimates developed in the pioneering work of Fish-Hofer \cite{FishHoferFeral} on ``feral pseudoholomorphic curves''. There are also additional complications in our setting that do not arise in \cite{FishHoferFeral}.

Before stating our main results and subsequently elaborating on our main strategy, we must first review some preliminary definitions and notations regarding pseudoholomorphic curves. Section \ref{subsec:jCurves} lists the essential definitions and notations for pseudoholomorphic curves in general almost-complex manifolds. Section \ref{subsec:adaptedCylinders} specializes the discussion to pseudoholomorphic curves in $\mathbb{R} \times M$, where $(M, \eta = (\lambda, \omega))$ is a framed Hamiltonian manifold. It defines \emph{$\omega$-finite pseudoholomorphic curves} in $\mathbb{R} \times M$, which are the pseudoholomorphic curves required by \cite{prequel} to construct interesting invariant measures of the Hamiltonian vector field $X$ associated to $(M, \eta)$. It also defines Fish-Hofer's \emph{feral pseudoholomorphic curves}, which are special cases of $\omega$-finite pseudoholomorphic curves that we construct in this paper. 

We also note that our exposition here has significant overlaps with \cite{prequel}. Our goal is to make this work self-contained without reference to \cite{prequel}, which makes this overlap necessary.

\subsection{Pseudoholomorphic curves in almost-complex manifolds} \label{subsec:jCurves}

\begin{defn}\label{defn:markedNodalSurface}
A \textbf{nodal} Riemann surface is the datum 
$$(C, j, D)$$
which consists of the following data:
\begin{itemize}
    \item A (possibly disconnected) smooth surface $C$.
    \item An almost-complex structure $j$ on $C$.
    \item A set of pairs of \emph{nodal points} $D = \{(\overline{z}_1, \underline{z}_1), (\overline{z}_2, \underline{z}_2), \ldots \}$ such that the full collection of points in $D$ is discrete, closed, and no point of $D$ lies in the boundary $\partial C$.
\end{itemize}

A \textbf{marked nodal} Riemann surface is the datum 
$$(C, j, D, \mu)$$
where $(C, j, D)$ is a nodal Riemann surface and $\mu$ is a closed, discrete subset of $C \setminus (D \cup \partial C)$, called the set of \emph{marked points}. 
\end{defn}

\begin{rem}
We will often abuse notation and refer to the set $D$ of \emph{pairs} of nodal points as merely a \emph{set} of nodal points.  
\end{rem}

\begin{defn}\label{defn:stableSurface}
A marked nodal Riemann surface $(C, j, D, \mu)$ is \textbf{stable} if every connected component of $C \setminus (D \cup \mu)$ has negative Euler characteristic. 
\end{defn}

Next, we define the \emph{arithmetic genus} of a nodal Riemann surface. Doing so first requires defining the \emph{genus} of a smooth, two-dimensional manifold. 

Recall for any connected \emph{compact} surface $C$, the genus of $C$, denoted by $\text{Genus}(C)$, is the genus of the closed surface obtained from $C$ by capping off all of the boundary components of $C$ with disks. For any disconnected surface $C$ with compact connected components, $\text{Genus}(C)$ takes values in the extended positive integers $\mathbb{Z}_{\geq 0} \cup \{\infty\}$. We do not restrict $C$ to have finitely many connected components. If infinitely many connected components of $C$ have positive genus, we set $\text{Genus}(C) = \infty$. If not, $\text{Genus}(C)$ is equal to the sum of the genera of its connected components. 

In the case where $C$ is not compact, the genus is defined using a compact exhaustion. 

\begin{defn} \label{defn:curveGenus}
The \textbf{genus} $\text{Genus}(C)$ of a connected, non-compact Riemann surface $C$ is defined as follows. Select a sequence of compact sub-surfaces $\{C_k\}_{k \geq 0}$ such that
$$C_k \subset \text{int}(C_{k+1})$$
for every $k$ and 
$$C = \cup_{k \geq 0} C_k.$$

Then set
$$\text{Genus}(C) = \lim_k \text{Genus}(C_k) \in \mathbb{Z}_{\geq 0} \cup \{\infty\}.$$

If $C$ is disconnected and non-compact, and any of its connected components $C'$ have $\text{Genus}(C') = \infty$, we set $\text{Genus}(C) = \infty$. Otherwise, we define $\text{Genus}(C)$ to be equal to $\infty$ if infinitely many connected components have positive genus, and equal to the sum of the genera of its connected components if only finitely many connected components have positive genus. 
\end{defn}

It is a straightforward exercise to show that the definition of $\text{Genus}(C)$ is independent of the choice of compact exhaustion $\{C_k\}$. Now, we define arithmetic genus. 

\begin{defn} \label{defn:curveArithmeticGenus}
Let $(C, j, D)$ be a nodal Riemann surface. Let $C'$ be the surface constructed from $C$ by performing a connect sum operation at each pair of nodal points. Then the \textbf{arithmetic genus} is defined by 
$$\text{Genus}_{\text{arith}}(C, D) = \text{Genus}(C')$$
where $\text{Genus}(C')$ is defined as in Definition \ref{defn:curveGenus}.
\end{defn}

\begin{rem}
Let $(C, j, D)$ be a nodal Riemann surface, and suppose it has finitely many connected components $\{C_i\}_{i=1}^k$. Then from Definition \ref{defn:curveArithmeticGenus} it is immediate that
$$\text{Genus}_{\text{arith}}(C, D) \geq \sum_{i=1}^k \text{Genus}(C_i),$$
which we will use later without comment. 
\end{rem}

Now we can define pseudoholomorphic curves.

\begin{defn} \label{defn:nodalCurve}
A \textbf{(marked nodal) pseudoholomorphic curve} or \textbf{$J$-holomorphic curve} is the datum
$$\mathbf{u} = (u, C, j, W, J, D, \mu)$$
which consists of the following data:
\begin{itemize}
    \item A marked nodal Riemann surface $(C, j, D, \mu)$.
    \item An almost-complex manifold $(W, J)$.
    \item A smooth map
    $$u: C \to W$$
    satisfying the non-linear Cauchy-Riemann equation
    $$du + J \circ du \circ j = 0$$
    and the property that $u(\overline{z}) = u(\underline{z})$ for every pair $(\overline{z}, \underline{z}) \in D$.
\end{itemize}
\end{defn}

In symplectic geometry, one usually studies pseudoholomorphic curves
$$\mathbf{u} = (u, C, j, W, J, D, \mu)$$
in symplectic manifolds $(W, \Omega)$, where $J$ is either \emph{tame} or \emph{compatible} with respect to the symplectic form $\Omega$.

\begin{defn}
\label{defn:tamedAndCompatible} An almost-complex structure $J$ on a symplectic manifold $(W, \Omega)$ is \textbf{tame} if 
$$\Omega(V, JV) > 0$$
for any tangent vector $V$. It is \textbf{compatible} if it is tame and additionally
$$\Omega(JV_1, JV_2) = \Omega(V_1, V_2)$$
for any pair of tangent vectors $V_1$ and $V_2$. 
\end{defn}

Gromov-Witten invariants count \textbf{stable} $J$-holomorphic curves in symplectic manifolds, where $J$ is tame. We introduce the notion of a stable pseudoholomorphic curve (which does not require the target to be symplectic or $J$ to be tame/compatible) in the following two definitions.

\begin{defn} \label{defn:connectedCurve}
A marked nodal pseudoholomorphic curve 
$$\mathbf{u} = (u, C, j, W, J, D, \mu)$$
is \textbf{connected} the surface constructed from $C$ by performing connect sum operations at each pair of nodal points is connected.
\end{defn}

\begin{defn} \label{defn:stableMap}
A \textbf{stable} $J$-holomorphic curve is a marked nodal boundary-immersed $J$-holomorphic curve
$$\mathbf{u} = (u, C, j, W, J, D, \mu)$$
such that for any connected component $C'$ of $C$ on which the map $u$ is constant, the Euler characteristic of $C' \setminus (\mu \cup D)$ is negative.
\end{defn}

\subsection{Pseudoholomorphic curves in adapted cylinders} \label{subsec:adaptedCylinders}

Fix $M$ to be a closed, smooth manifold of dimension $2n+1$. Equip $M$ with a framed Hamiltonian structure $\eta = (\lambda, \omega)$ and let $X$ be the associated Hamiltonian vector field. 

Before proceeding, we write down a couple of basic objects associated to $(M, \eta)$. 

There is a natural codimension-one tangent distribution on $M$ given by
$$\xi = \text{ker}(\lambda)$$
and the pair $(\xi, \omega)$ forms the datum of a symplectic bundle. 

Write $\langle X \rangle$ for the one-dimensional subbundle of $TM$ defined by the span of the Hamiltonian vector field $X$. Then there is a splitting
$$TM = \langle X \rangle \oplus \xi.$$

We will define the projection map
$$\pi_\xi: TM \to \xi$$
to be the projection with kernel $\langle X \rangle$.

Now recall the definition of an almost-Hermitian manifold.

\begin{defn} \label{defn:almostHermitian}
An \textbf{almost-Hermitian structure} on an even-dimensional manifold $W$ is the datum of a Riemannian metric $g$ and an almost-complex structure $J$ such that 
$$g(J-, J-) = g(-, -).$$

A \textbf{almost-Hermitian} manifold is an even-dimensional manifold $W$ equipped with an almost-Hermitian structure $(J, g)$.
\end{defn}

At one point, we will require the notion of an \emph{isomorphism} of two almost-Hermitian manifolds. 

\begin{defn}
\label{defn:almostHermitianIsomorphism}

Let $(W_1, J_1, g_1)$ and $(W_2, J_2, g_2)$ be two almost-Hermitian manifolds of the same dimension. We say that they are \textbf{isomorphic} if there is a diffeomorphism
$$\Phi: W_1 \simeq W_2$$
such that
$$\Phi^*g_2 = g_1$$
and
$$\Phi^*J_2 = J_1.$$
\end{defn}

There is a distinguished class of translation-invariant almost-complex structures $J$ on the cylinder $\mathbb{R} \times M$ called \emph{$\eta$-adapted} almost-complex structures which incorporate the Hamiltonian vector field $X$. 

Such an almost-complex structure (denoted for the moment by $J$) has a natural associated Riemannian metric $g$, and it is easy to show that the pair $(J, g)$ forms an almost-Hermitian structure on $\mathbb{R} \times M$. 

We now proceed with all of the relevant definitions. 

\begin{defn} \label{defn:adaptedJ}
An almost-complex structure on $\mathbb{R} \times M$ is \textbf{$\eta$-adapted} (or just \textbf{adapted} when the context is clear) if it satisfies the following properties:
\begin{enumerate}
    \item $J$ is invariant with respect to translation in the $\mathbb{R}$ factor. 
    \item Let $a$ denote the $\mathbb{R}$ coordinate. Then
    $$J(\partial_a) = X$$
    and
    $$J(X) = -\partial_a$$
    \item $J$ preserves $\xi = \text{ker}(\lambda)$ and the restriction $J|_\xi$ is compatible with $\omega$, i.e.
    $$\omega(JV_1, JV_2) = \omega(V_1,V_2)$$
    for any pair of tangent vectors $V_1$ and $V_2$ in $\xi$ and
    $$\omega(V, JV) > 0$$
    for any tangent vector $V$ in $\xi$. 
\end{enumerate}
\end{defn}

Now fix an adapted almost-complex structure $J$ on $\mathbb{R} \times M$. We will always use $a$ to denote the $\mathbb{R}$-coordinate on $\mathbb{R} \times M$. Then we may define a Riemannian metric $g$ by 
$$g(-, -) = (da \wedge \lambda + \omega)(-, J-).$$

Here $\lambda$ and $\omega$ denote the translation-invariant pullbacks to $\mathbb{R} \times M$ of the corresponding differential forms on $M$. 

\begin{rem} \label{rem:metric}
Note by definition that this metric is cylindrical, so it induces a natural metric $g$ on $M$ given explicitly by
$$g(-, -) = (\lambda \otimes \lambda)(-, -) + \omega(-, J-).$$

When we think of $g$ as a metric on $M$ in what follows, it will be this metric. 

It follows from this computation and the previous definitions that
$$\text{dvol}_g = \frac{1}{n!}\lambda \wedge \omega^n$$
and $X$ preserves this volume form. 
\end{rem}

The proof of the following lemma amounts to a quick algebraic computation using the above definitions.

\begin{lem}
The datum $(J, g)$ defines an almost-Hermitian structure on $\mathbb{R} \times M$. 
\end{lem}

We refer to almost-Hermitian manifolds of the form $(\mathbb{R} \times M, J, g)$ as \emph{adapted cylinders} over $M$. 

\begin{defn}
\label{defn:adaptedCylinder} An \textbf{$\eta$-adapted cylinder} (or \textbf{adapted cylinder} when the context is clear) over a framed Hamiltonian manifold $(M, \eta = (\lambda, \omega))$ is an almost-Hermitian manifold of the form $(\mathbb{R} \times M, g, J)$ where $(g,J)$ is the almost-Hermitian structure defined by a choice of an $\eta$-adapted almost-complex structure $J$ and 
$$g(-, -) = (da \wedge \lambda + \omega)(-, J-).$$
\end{defn}

\begin{rem}
The definition of an $\eta$-adapted almost-complex structure $J$ on $\mathbb{R} \times M$ was introduced in \cite{FishHoferFeral}, and mirrors the definition of an admissible almost-complex structure $J$ on the symplectization $\mathbb{R} \times M$ over a contact manifold introduced in \cite{HoferWeinstein}. Accordingly, \cite{FishHoferFeral} refers to $\mathbb{R} \times M$ as the ``symplectization'' due to this correspondence with the contact case. 

We prefer to use the terminology ``adapted cylinder'' because $\mathbb{R} \times M$ may not be symplectic if $M$ is not contact. More specifically, when $M$ is a framed Hamiltonian manifold that is not a contact manifold, $\mathbb{R} \times M$ to our knowledge does not necessarily admit a natural symplectic form that is compatible with an $\eta$-adapted almost-complex structure $J$. 
\end{rem}

Fix a choice of $\eta$-adapted cylinder $(\mathbb{R} \times M, g, J)$ over $M$. Now we will define ``$\omega$-finite'' pseudoholomorphic curves, which are $J$-holomorphic curves of the form
$$\mathbf{u} = (u, C, j, \mathbb{R} \times M, J, D, \mu)$$
satisfying some additional restrictions. 

We begin with defining a couple of useful notions of energy. Let $a$ denote the $\mathbb{R}$-coordinate projection function on $\mathbb{R} \times M$. 

Write $\mathcal{R}$ for the set of regular values of the real-valued map $a \circ u$. By Sard's theorem it follows that $\mathcal{R}$ has full Lebesgue measure in $\mathbb{R}$. 

\begin{defn} \label{defn:energy}
Let 
$$\mathbf{u} = (u, C, j, \mathbb{R} \times M, J, D, \mu)$$
be a $J$-holomorphic curve. The \textbf{$\lambda$-energy} of $\mathbf{u}$ is the function
$$E_\lambda(u, -): \mathcal{R} \to \mathbb{R}$$
defined by
$$E_\lambda(u, t) = \int_{(a \circ u)^{-1}(t)} u^*\lambda.$$

The \textbf{$\omega$-energy} of $\mathbf{u}$ is
$$E_\omega(u) = \int_C u^*\omega.$$
\end{defn}

The following lemma is a consequence of the adaptedness of $J$.

\begin{lem} \label{lem:omegaNonnegative}
Let
$$\mathbf{u} = (u, C, j, \mathbb{R} \times M, J, D, \mu)$$
be a connected $J$-holomorphic curve.

The $\omega$-energy $E_\omega(u)$ is non-negative, and it is equal to zero if and only if either the map $u$ is constant or has image contained in $\mathbb{R} \times \gamma \subset \mathbb{R} \times M$, where $\gamma$ is a trajectory of the Hamiltonian vector field.  

Moreover, $u^*\omega$ is non-negative pointwise on the tangent planes of $C$. 
\end{lem}

Now we can define $\omega$-finite pseudoholomorphic curves.

\begin{defn}\label{defn:omegaFiniteCurve}
A $J$-holomorphic curve 
$$\mathbf{u} = (u, C, j, \mathbb{R} \times M, J, D, \mu)$$ is \textbf{$\omega$-finite} if
\begin{itemize}
\item The map $u$ is proper.
\item The image $u(C)$ has finitely many connected components.
\item $E_\omega(u) < \infty$.
\item The domain $C$ is has at least one non-compact connected component. 
\item There is a compact set $K \subseteq \mathbb{R} \times M$ such that the image $u(\partial C)$ of the boundary of $C$ is contained in $K$.
\end{itemize}
\end{defn}

The second condition regarding the connectedness of $u(C)$ is unnatural, but required to rule out some pathological cases in the proof of Proposition $2.1$ in \cite{prequel}. However, it is extremely weak and quite easy to verify in practice.

Definition \ref{defn:omegaFiniteCurve} is rather similar to the definition of a ``feral curve'' given in \cite{FishHoferFeral}. 
However, the domain of a feral pseudoholomorphic curve satisfies several additional ``finite topology'' assumptions that the domain of an $\omega$-finite curve is not required to satisfy. One of them, the number of ``generalized punctures'' of a pseudoholomorphic curve, requires additional explanation.

We define this below, but again one should either check that this is well-defined or refer to \cite{FishHoferFeral}.

\begin{defn} \label{defn:punct}
For any compact sub-surface $C' \subseteq C$, write $\mathcal{C}_{\text{non-compact}}(C\setminus C')$ for the number of non-compact connected components of $C \setminus C'$. Then for any exhaustion of $C$ by compact sub-surfaces $\{C_k\}_{k \geq 0}$ as in Definition \ref{defn:curveGenus}, define
$$\text{Punct}(C) = \lim_{k \to \infty} \#\mathcal{C}_{\text{non-compact}}(C \setminus (C_k \setminus \partial C_k)) \in \mathbb{Z}_{\geq 0} \cup \{\infty\}.$$
\end{defn}

We leave it to the reader to check that this is well-defined. A description of a proof can be found in \cite[Remark $1.4$]{FishHoferFeral}. Now we can define a feral pseudoholomorphic curve.

\begin{defn} \label{defn:feralCurve}
A $J$-holomorphic curve
$$\mathbf{u} = (u, C, j, \mathbb{R} \times M, J, D, \mu)$$
is \textbf{feral} if it is $\omega$-finite, and moreover satisfies the following conditions:
\begin{itemize}
\item $\text{Genus}(C) < \infty$.
\item $\#D < \infty$.
\item $\#\mu < \infty$.
\item $\text{Punct}(C) < \infty.$
\end{itemize}
\end{defn}

\subsection{Statement of main theorems} \label{subsec:thmStatements}

We are now ready to state our main results. The first result requires non-vanishing of a certain Gromov-Witten invariant, which we will discuss following its statement. 

\begin{thm}
\label{thm:mainExample} Let $(W, \Omega)$ be a closed, oriented symplectic manifold of dimension $2n+2$ where $n \geq 1$. Let $H$ be a smooth Hamiltonian function on $W$ with $0$ as a regular value, and let $M = H^{-1}(0)$ denote the corresponding regular energy level. Let $\eta = (\lambda, \omega)$ denote a framed Hamiltonian structure on $M$ with $\omega = \Omega|_M$. 

Suppose that the Gromov-Witten invariant given in Definition \ref{defn:GWHypersurface} does not vanish:
$$\Psi^{W,M} \neq 0.$$

Then for any adapted cylinder
$$(\mathbb{R} \times M, J, g)$$
over $M$, there is an $\omega$-finite pseudoholomorphic curve denoted by
$$\mathbf{u} = (u, C, j, \mathbb{R} \times M, J, D, \mu).$$
\end{thm}

As mentioned in the statement above, the Gromov-Witten invariant $\Psi^{W,M}$ is defined in Definition \ref{defn:GWHypersurface}. It is non-zero if there is $G \in \mathbb{N}$, $A \in H_2(W; \mathbb{Z})$, and submanifolds $Y_-$ and $Y_+$ lying in distinct components of $W \setminus M$ such that, for any tame almost-complex structure $J$ on $W$ (see Definition \ref{defn:tamedAndCompatible}), there is a stable, connected pseudoholomorphic curve
$$\mathbf{u} = (u, C, j, W, J, D, \mu)$$
such that $C$ is closed, $\text{Genus}_{\text{arith}}(C, D) = G$, $u$ represents the homology class $A$, and $u(C)$ intersects both $Y_-$ and $Y_+$, crossing the hypersurface $M$. This is the statement of Corollary \ref{cor:GWCrossingExistence}. 

The second result concerns hypersurfaces in a four-dimensional symplectic manifold $(W, \Omega)$ satisfying certain topological conditions. It can be viewed as an analogue of the main theorem of \cite{ChenWeinstein} for non-contact-type hypersurfaces. 

To understand these topological conditions, we need to establish two more pieces of notation. 

First, if $W$ is any smooth four-manifold, then $b_2^+(W)$ denotes the maximal possible dimension of a subspace of $H_2(W; \mathbb{R})$ that is positive-definite with respect to the intersection form
$$H_2(W; \mathbb{R}) \otimes H_2(W; \mathbb{R}) \to \mathbb{R}.$$

Second, recall that for any symplectic manifold $(W', \Omega')$, the space of almost-complex structures tamed by $\Omega'$ is contractible. It follows that the first Chern classes $c_1(TW', J)$ coincide for any almost-complex structure $J'$ tamed by $\Omega$. We use $c_1(W')$ to denote this common cohomology class. 

\begin{thm}
\label{thm:swExample} Let $(W, \Omega)$ be a closed, oriented symplectic manifold of dimension $4$. Let $H$ be a smooth Hamiltonian function on $W$ with $0$ as a regular value, and let $M = H^{-1}(0)$ denote the corresponding regular energy level. Let $\eta = (\lambda, \omega)$ denote a framed Hamiltonian structure on $M$ with $\omega = \Omega|_M$. 

Suppose that the following topological conditions hold.

\begin{itemize} 
\item $b_2^+(W) > 1$.
\item $c_1(W)^2 \neq 0$. 
\item The smooth manifold 
$$W_- = H^{-1}((-\infty, 0])) \subset W$$
bounding $M$ is such that the restriction to $\Omega$ of $W_-$ is exact and $c_1(W_-) \neq 0$. 
\end{itemize}

Then for any adapted cylinder
$$(\mathbb{R} \times M, g, J)$$
over $M$, there is an $\omega$-finite pseudoholomorphic curve denoted by
$$\mathbf{u} = (u, C, j, \mathbb{R} \times M, J, D, \mu).$$
\end{thm}

The proof uses Taubes' work on Seiberg-Witten invariants on symplectic manifolds. 

First, Taubes shows in \cite{Taubes94} that, if $b_2^+(W) > 1$, the Seiberg-Witten invariant of $(W, \Omega)$ corresponding to the class $-c_1(W)$ of the symplectic form does not vanish. Here the symplectic form is necessary to identify spin-c structures on $W$ with classes in $H^2(W; \mathbb{Z})$, as the Seiberg-Witten invariants usually take the former as an input.

Second, Taubes shows in \cite{TaubesSWGR} that, if $b_2^+(W) > 1$, the Seiberg-Witten invariants of $W$ equal the \emph{Gromov invariants} of $W$, a pseudoholomorphic curve counting invariant that we will discuss in more detail in Section \ref{sec:taubesGromov}. The non-vanishing of the Gromov invariant corresponding to $-c_1(W)$ implies that, if $c_1(W) \neq 0$, then for a \emph{generic} choice of tame almost-complex structure $J$, there is a pseudoholomorphic curve
$$\mathbf{u} = (u, C, j, W, J, \emptyset, \emptyset)$$
satisfying the following conditions:
\begin{itemize}
    \item $C$ is a finite, disjoint union of closed surfaces $C_k$ for $k = 1, \ldots, N$. 
    \item The images $u(C_k)$ are disjoint, embedded submanifolds of $W$, and the restriction of $u$ to $C_k$ is an embedding unless $C_k$ is a torus.
    \item The map $u$ represents the Poincar\'e dual of $-c_1(W)$. 
\end{itemize}

Third, the assumption that $c_1(W_-) \neq 0$ then implies that at least one component $C'$ of $C$ must be such that $u(C')$ intersects $W_-$. The assumption that $\Omega$ is exact on $W_-$ implies that $u(C')$ is not contained within $W_-$ for any component $C'$ of $C$. Therefore, $C$ always has at least one component $C'$ such that $u(C')$ intersects both $W_-$ and $W_+$. 

It is then a consequence of the Gromov compactness theorem, along with our topological assumptions, that there is a pseudoholomorphic curve 
$$\mathbf{u} = (u, C', j', W, J, D, \mu)$$
with connected domain $C'$ and image $u(C')$ intersecting both $W_-$ and $W_+$ for \emph{any} tame almost-complex structure $J$

The proof of Theorem \ref{thm:swExample} then proceeds in a similar manner to the proof of Theorem \ref{thm:mainExample}. However, there are two places where we must take care. 

First, Taubes only works with \emph{compatible} almost-complex structures in his definition of his Gromov invariant. Our neck stretching procedure requires existence of pseudoholomorphic curves for generic \emph{tame} almost-complex structures, not just generic compatible almost-complex structures. However, as we will discuss, Taubes' Gromov invariant can just as easily be defined for our required set of almost-complex structures, and moreover the definition of the Gromov invariant is independent of the choice of this tame almost-complex structure. 

Second, the existence statement for pseudoholomorphic curves deduced from nonvanishing of Taubes' Gromov invariant does not explicitly state the genus of the domain of the curve $C$. However, the fact that the image $u(C)$ is embedded, along with the adjunction formula, immediately allows us to bound the genus (see \cite[Lemma $2.4$]{ChenWeinstein} for an example of this argument in a related setting).

\subsection{Applications to non-unique ergodicity}

Theorems \ref{thm:mainExample} and \ref{thm:swExample} allow us to apply the results of \cite{prequel} to show that various Hamiltonian flows are not uniquely ergodic. The following general result is a restatement of \cite[Corollary $1.7$]{prequel}, with the references to other results within the statement of the former expanded for clarity.

\begin{prop}
\label{prop:nonUniqueErgodicity} \cite{prequel} Suppose $(M^{2n+1}, \eta = (\lambda, \omega))$ is a framed Hamiltonian manifold satisfying the following two conditions:
\begin{enumerate}
\item There is an $\eta$-adapted cylinder $(\mathbb{R} \times M, g, J)$ and an $\omega$-finite pseudoholomorphic curve
$$\mathbf{u} = (u, C, j, \mathbb{R} \times M, J, \emptyset, \emptyset).$$
\item Either $\omega^n$ is not exact, or $\omega$ is exact with primitive $\nu$ and the ``self-linking number''
$$\text{Lk}(\omega) = \int_M \nu \wedge \omega^n$$
does not vanish. 
\end{enumerate}

Then the Hamiltonian vector field $X$ is not uniquely ergodic. 
\end{prop}

The following is a consequence of Proposition \ref{prop:nonUniqueErgodicity} and Theorem \ref{thm:mainExample}. The proof is given in \cite[Proposition $1.10$]{prequel}. 

\begin{prop}
\label{prop:application1} Let $(W, \Omega)$ be a closed symplectic manifold of dimension $2n+2$ that is $\text{GW}_G$-connected for some integer $G \geq 0$. Let 
$$H: W \to \mathbb{R}$$
be any smooth Hamiltonian such that $0$ is a regular value, and set $M = H^{-1}(0)$.

Suppose that one of the following two conditions hold:
\begin{itemize}
    \item The restriction of the two-form $\Omega^n$ is not exact on $M$. 
    \item $\Omega$ is exact on one of the two components of $W \setminus M$.
\end{itemize} 

Then the Hamiltonian vector field $X_H$ on $M$ is not uniquely ergodic. 
\end{prop}

The following is a direct consequence of Proposition \ref{prop:nonUniqueErgodicity} and Theorem \ref{thm:swExample}. We provide a proof below. 

\begin{prop}
\label{prop:application2} Let $M = H^{-1}(0)$ be a closed, regular energy level inside of a closed, symplectic $4$-manifold $(W, \Omega)$. 

Suppose that the following topological conditions hold.

\begin{itemize} 
\item $b_2^+(W) > 1$.
\item $c_1(W) \neq 0$. 
\item The smooth manifold 
$$W_- = H^{-1}((-\infty, 0])) \subset W$$
bounding $M$ is such that the restriction to $\Omega$ of $W_-$ is exact and $c_1(W_-) \neq 0$. 
\end{itemize}

Then the Hamiltonian vector field $X_H$ is not uniquely ergodic. 
\end{prop}

\begin{proof}
The proof follows from verifying that $M$ and $W$ satisfy the assumptions of Proposition \ref{prop:nonUniqueErgodicity}. Let $\eta = (\lambda, \omega)$ denote the framed Hamiltonian structure on $M$ given by the construction of \ref{exe:hypersurfaces}. Then the vector field $X_H$ equals the Hamiltonian vector field $X$ associated to $\eta$. 

Theorem \ref{thm:swExample} shows, given the conditions on $M$ and $W$, that the first condition of Proposition \ref{prop:nonUniqueErgodicity} is satisfied: there is an $\eta$-adapted cylinder $(\mathbb{R} \times M, g, J)$ and an $\omega$-finite pseudoholomorphic curve
$$\mathbf{u} = (u, C, j, \mathbb{R} \times M, J, \emptyset, \emptyset).$$

It remains to verify the second condition. By assumption, $\omega = \Omega|_M$ is exact. We need to show
$$\int_M \nu \wedge \omega \neq 0$$
where $\nu$ is a primitive of $\omega$. By Stokes' theorem and the fact that $\omega = \Omega|_M$, 
$$\int_M \nu \wedge \omega = \int_{W_-} \Omega \wedge \Omega \neq 0.$$

The proposition follows. 
\end{proof}

\subsection{Roadmap}

The rest of the paper is organized as follows. In Section \ref{sec:outline} we give a detailed outline of the proofs of Theorem \ref{thm:mainExample} and \ref{thm:swExample}. In Section \ref{sec:gw} we give an overview of the Gromov-Witten invariants and the relevant result regarding existence of pseudoholomorphic curves that we require for the proof of Theorem \ref{thm:mainExample}. In Section \ref{sec:taubesGromov} we give an overview of Taubes' Gromov invariant and the relevant result regarding existence of pseudoholomorphic curves that we require for the proof of Theorem \ref{thm:swExample}. In Section \ref{sec:analysis} we review the analytical background from \cite{FishHoferFeral} and elsewhere needed for the proofs of our main theorems. In Section \ref{sec:proofs} we introduce our neck stretching construction and prove our main theorems. 

\textbf{Acknowledgements:} I would like to thank my advisor, Helmut Hofer, for his encouragement, discussions regarding the work, and for his suggestions regarding the exposition. I would also like to thank Joel Fish and Chris Gerig for enlightening discussions regarding pseudoholomorphic curves, and Shaoyun Bai and John Pardon for discussions regarding Gromov-Witten invariants. 

\section{Outlines of Theorems \ref{thm:mainExample} and \ref{thm:swExample}} \label{sec:outline}

In this section, we will outline the proofs of Theorems \ref{thm:mainExample} and \ref{thm:swExample}. Let $(W, \Omega)$ be a closed symplectic manifold of dimension $2n+2$ and $M = H^{-1}(0)$ a closed, regular energy level of a smooth Hamiltonian $H$. Suppose that $W$ and $M$ satisfy the assumptions of Theorem \ref{thm:mainExample} or of Theorem \ref{thm:swExample}. Let $\eta = (\lambda, \omega)$ denote a Hamiltonian structure on $M$ constructed as in Example \ref{exe:hypersurfaces}, with $\omega$ equal to the restriction of $\Omega$ to $M$. 

The proofs of Theorems \ref{thm:mainExample} and \ref{thm:swExample} uses the ``neck-stretching'' technique from gauge theory and symplectic geometry. The most similar instances of this technique to our own are the recent work \cite{FishHoferFeral} and the symplectic field theory neck stretching construction of \cite{BEHWZ03}. The neck stretching construction in Section \ref{subsec:neckStretching} constructs from $W$ a sequence of almost-Hermitian manifolds $(W_L, \overline{J}_L, g_L)$ for any $L > 0$. Topologically, $W_L$ is obtained from $W$ by cutting open along $M$ and gluing in a cylinder $[-L, L] \times M$. 

The construction is performed so that the restriction of the  almost-complex structure $\bar{J}_L$ to $[-L, L] \times M$ agrees with a model $\eta$-adapted almost-complex structure $J_{\text{Neck}}$ on $\mathbb{R} \times M$. The restriction of the metric $g_L$ agrees with a model metric $g_{\text{Neck}}$ on $\mathbb{R} \times M$, so that 
$$(\mathbb{R} \times M, J_{\text{Neck}}, g_{\text{Neck}})$$
is an $\eta$-adapted cylinder.

Write $W_-$ and $W_+$ for the closures of the two components of $W \setminus M$. Both $W_-$ and $W_+$ have boundary $M$. 

It is also useful for us to define an almost-Hermitian manifold 
$$(\widetilde{W}_-, \bar{J}_-, g_-),$$ 
which topologically is given by gluing a cylindrical end $[0, \infty) \times M$ to $W_-$. The restriction of the almost-Hermitian structure $(\bar{J}_-, g_-)$ to the cylindrical end also agrees with $(J_{\text{Neck}}, g_{\text{Neck}})$. 

In Proposition \ref{prop:tameJ}, we show that, for sufficiently large $L$, there is a diffeomorphism 
$$\Phi_L: W \to W_L$$
such that the pullback $\Phi_L^*\bar{J}_L$ is tamed by the symplectic form $\Omega$. Although the existence of such a diffeomorphism is expected, the proof is surprisingly subtle and requires a careful choice of the diffeomorphism $\Phi_L$.  

In the setting of Theorem \ref{thm:mainExample}, the Gromov-Witten invariant $\Psi^{W,M}$ is defined in a manner that, when it does not vanish, implies the existence of a marked nodal $J$-holomorphic curve for any tame almost-complex structure $J$ in $W$ with closed domain that ``crosses'' the hypersurface, i.e. it intersects both connected components of $W \setminus M$. See Proposition \ref{prop:GWExistence} and Corollary \ref{cor:GWCrossingExistence}. 
In the setting of Theorem \ref{thm:swExample}, the topological conditions on $W$ and $M$, along with the non-vanishing of Taubes' Gromov invariant, imply the same statement. See Propositions \ref{prop:nonzeroGr} and \ref{prop:nonzeroGr2}. 

A marked nodal pseudoholomorphic curve in $W$ crossing the hypersurface $M$ is depicted in Figure \ref{fig:crossingCurve}.

\begin{figure}[ht] \includegraphics[width=.4\textwidth]{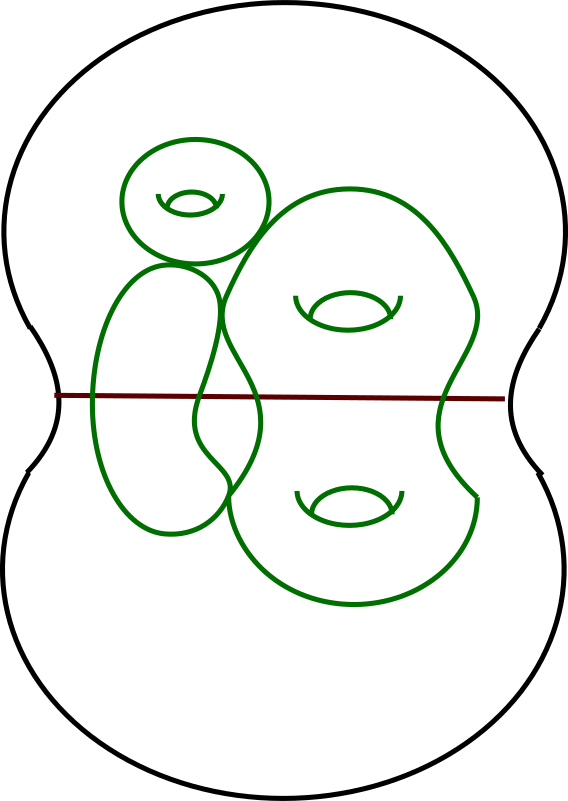} \caption{\label{fig:crossingCurve} A nodal curve (green) crossing the hypersurface $M$ (red) in $W$ (black).}\end{figure}

Combining this existence result, and the fact that the almost-complex structures in our neck stretching construction are tame, we deduce the existence of marked nodal $J_L$-holomorphic curves passing through the cylindrical region in $W_L$ for any sufficiently large $L > 0$. This is depicted in Figure \ref{fig:stretchedCrossingCurves}. 

\begin{figure}[ht] \includegraphics[width=.7\textwidth]{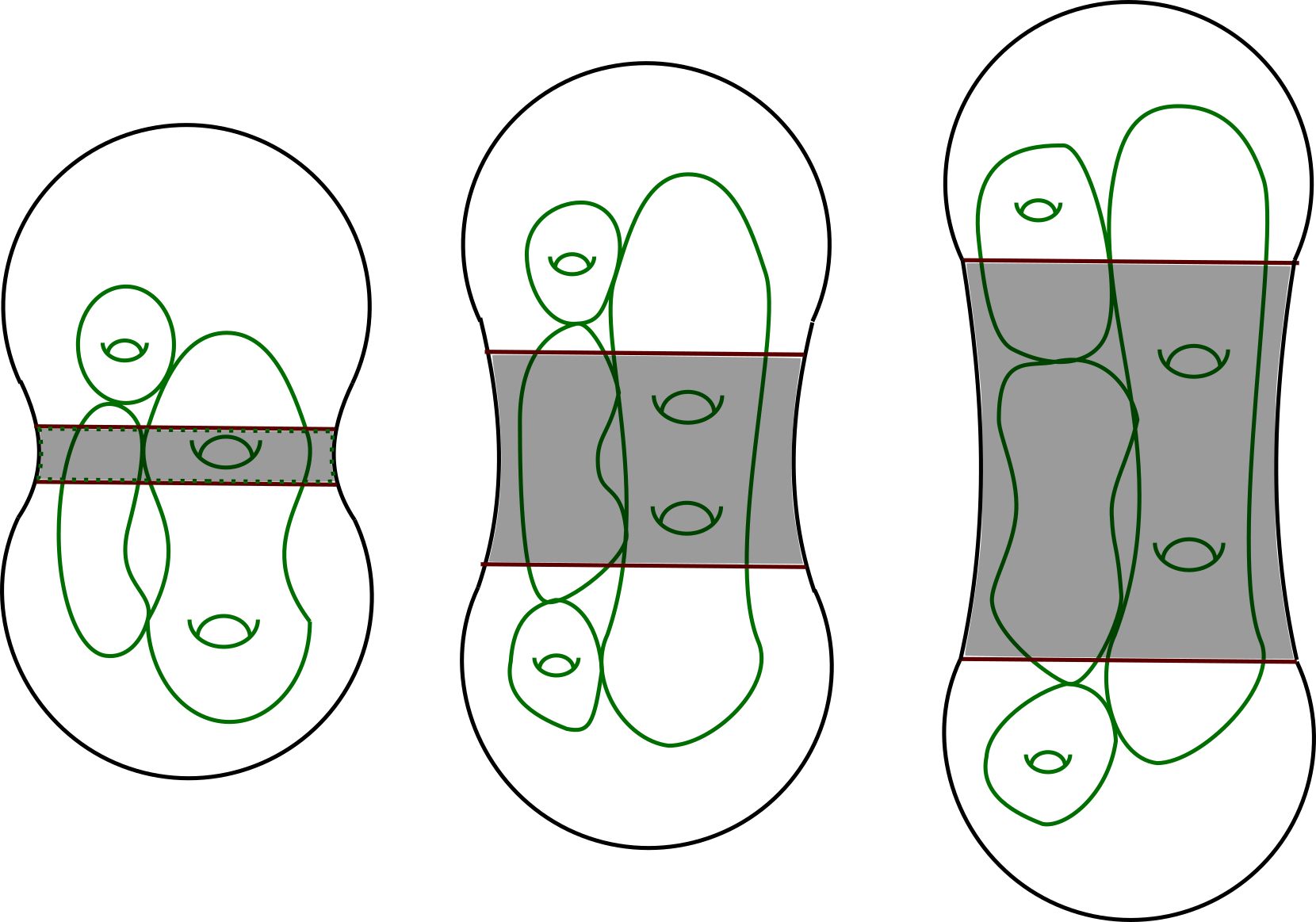} \caption{\label{fig:stretchedCrossingCurves} Nodal curves (green) in the neck-stretched manifolds (black) crossing longer and longer neck regions (dark gray).}\end{figure}

Next, write $W_-$ and $W_+$ for the closures of the two components of $W \setminus M$. The neck stretching construction yields natural embeddings
$$W_- \sqcup W_+ \hookrightarrow W_L$$
for any $L > 0$. 

We trim off the portions of these pseudoholomorphic curves that lie inside a neighborhood of $W_+$. This produces a sequence of pseudoholomorphic curves inside almost-Hermitian manifolds of the form
$$(W_- \cup [-L, L] \times M, \bar{J}_L, g_L)$$
which naturally embed (see Lemma \ref{lem:obviousInclusions} and the surrounding discussions) into the cylindrical-end manifold
$$(\widetilde{W}_-, \bar{J}_-, g_-)$$
which was defined previously. 

To eventually produce a feral pseudoholomorphic curve, the trimming needs to be done so that the pseudoholomorphic curves have domains with a uniformly bounded number of boundary components. The details are in Proposition \ref{prop:boundaryComponentBound}. 

Restricting to a suitable sequence $L_i \to \infty$, we obtain a sequence of pseudoholomorphic curves 
$$\mathbf{w}_i = (w_i, \widetilde{C}_i, \widetilde{j}_i, \widetilde{W}_-, \bar{J}_-, \widetilde{D}_i, \widetilde{\mu}_i)$$ in $\widetilde{W}_-$. The image $w_i(\partial\widetilde{C}_i)$ will have boundary lying near $\{2L_i\} \times M$ in the cylindrical end of $\widetilde{W}_-$, so as we take $i \to \infty$, the images of the curves ascend into the positive end. 

This is depicted in Figure \ref{fig:trimmedCurves}, ignoring the part of the pseudoholomorphic curves that lie outside the cylindrical end $[0,\infty) \times M$.  

\begin{figure}[ht] \includegraphics[width=.8\textwidth]{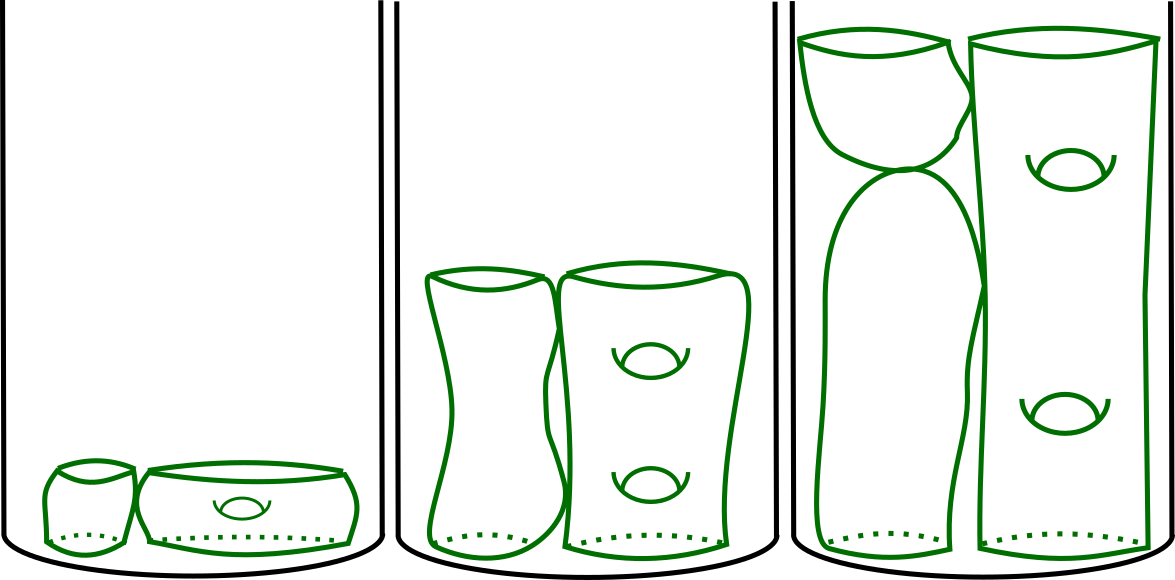}
 \caption{\label{fig:trimmedCurves} Longer and longer trimmed curves (green) in $[0,\infty) \times M$ (black).} \end{figure}
 
There is a compactness theorem for these types of $J$-holomorphic maps known as Fish-Hofer's ``exhaustive Gromov compactness'' theorem (\cite{FishHoferExhaustive}, also stated in Section \ref{subsec:exhaustiveCompactness}). By this theorem, after passing to a subsequence, the sequence of trimmed curves that we have constructed converge in a certain sense to a proper $\bar{J}_-$-holomorphic curve $\bar{\mathbf{w}}$ mapping into $\widetilde{W}_-$. The portion of $\bar{\mathbf{w}}$ in the cylindrical end $[0,\infty) \times M$ is depicted in Figure \ref{fig:limitCurve}. 

\begin{figure}[ht] \includegraphics[width=.5\textwidth]{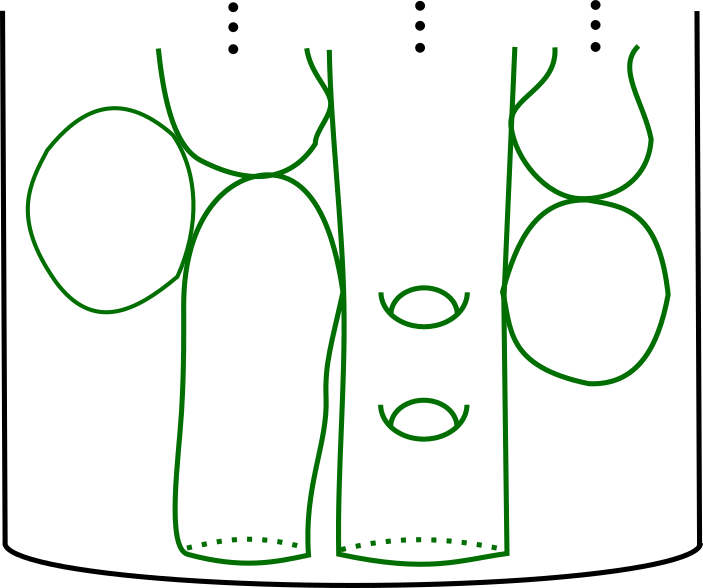}
 \caption{\label{fig:limitCurve} The limit curve (green) in $[0,\infty) \times M$ (black). It can get complicated!} \end{figure}

The first major technical step is to show that this sequence of curves satisfies the a priori estimates required by the exhaustive Gromov compactness theorem. Roughly, this means that the portions of the curves in any compact subset of $\widetilde{W}_-$ must have uniformly bounded area, genus, and marked/nodal points, where the uniform bound in question may depend on the choice of compact subset. 

The uniform bound on genus and marked points is evident from the construction of the sequence of pseudoholomorphic curves. We prove the area bound in Corollary \ref{cor:transitionAreaBounds} and the bound on the number of nodal points in Proposition \ref{prop:nodalPointBound}. The bounds on the genus, number of marked points, and number of nodal points we derive are stronger than what is required for exhaustive Gromov compactness. However, the genus bound is essential for our proof that the limit curve is $\omega$-finite. All three bounds are essential for our proof that the limit curve is feral. 

The second technical step is the extraction of an $\omega$-finite pseudoholomorphic curve from the limit pseudoholomorphic curve $\bar{\mathbf{w}}$. We restrict $\bar{\mathbf{w}}$ to the portion of the domain mapping into the cylindrical end, which produces a pseudoholomorphic curve $\hat{\mathbf{w}}$ in the $\eta$-adapted cylinder
$$(\mathbb{R} \times M, J_{\text{Neck}}, g_{\text{Neck}}).$$ 

We need to show that the limit pseudoholomorphic curve $\hat{\mathbf{w}}$ has finite $\omega$-energy, boundary mapping into a compact subset of $\mathbb{R} \times M$, the image of $\hat{\mathbf{w}}$ has finitely many connected components, and that the domain has a non-compact component.

From the specific properties of the constructed sequence in our situation, the second and third properties are not difficult to prove. 

The first property, the $\omega$-energy bound, is a consequence of a uniform $\omega$-energy bound on the sequence of pseudoholomorphic curves $\mathbf{w}_i$. The proof of this bound relies on the careful choices made in Proposition \ref{prop:tameJ}. See Lemma \ref{lem:sameSmallOmega} for details.

We also note that the domain of the limit pseudoholomorphic curve has finite (arithmetic) genus. 

The hardest part is showing that $\hat{\mathbf{w}}$ has a \emph{non-compact} component in its domain. It is not difficult to show that the image of the limit pseudoholomorphic curve $\hat{\mathbf{w}}$ is unbounded in $\mathbb{R} \times M$. Therefore, the only possible pathological case is where the image of $\hat{\mathbf{w}}$ contains an infinite string of closed surfaces in $\mathbb{R} \times M$ moving off to infinity, depicted in Figure \ref{fig:badLimitCurve}.

\begin{figure}[ht] \includegraphics[width=.5\textwidth]{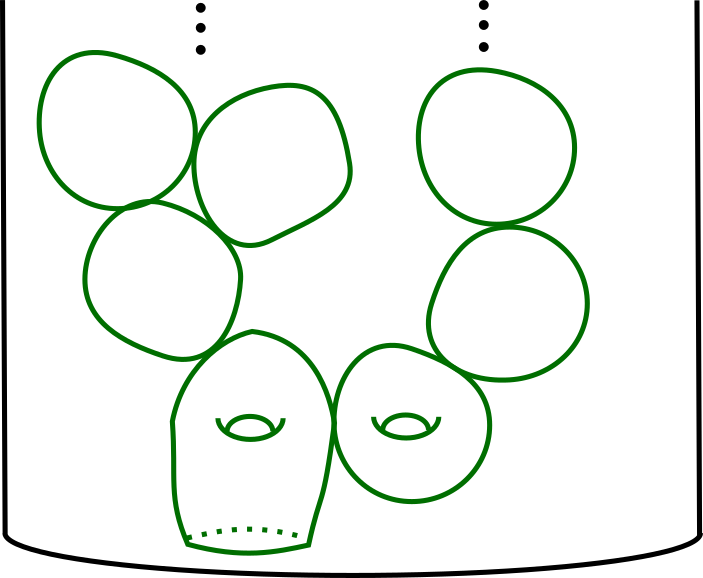}
 \caption{\label{fig:badLimitCurve} An example of a potential limit curve (green) with no non-compact components in $[0,\infty) \times M$ (black).} \end{figure}

We take care of this case by making the following two observations. First, since the domain of $\hat{\mathbf{w}}$ has finite genus, infinitely many of these closed surfaces must be spheres. Second, a ``$\omega$-energy quantization'' result of Fish-Hofer (see \cite[Theorem $4$]{FishHoferFeral} and its restatement in Theorem \ref{thm:omegaEnergyQuantization}) can prove that \emph{any} closed non-constant $J$-holomorphic sphere in $\mathbb{R} \times M$ has $\omega$-energy larger than some positive geometric constant. Therefore, the existence of infinitely many spherical components in the image would imply that the pseudoholomorphic curve $\hat{\mathbf{w}}$ has infinite $\omega$-energy, but it is known that the $\omega$-energy is finite. See the argument in Lemma \ref{lem:limitCurveNonCompactComponent}. 

The third and final technical step is to show that the limit curve $\hat{\mathbf{w}}$ is feral. We find that, given the previous work, it suffices to show that the number of nodal points and the number of ``generalized punctures'' (see Definition \ref{defn:punct}) of $\hat{\mathbf{w}}$ are finite. 

The fact that the number of nodal points is finite follows from combining the previous bounds with an explicit combinatorial formula for the arithmetic genus. See Proposition \ref{prop:limitCurveNodalPoints}. 

The fact that the number of generalized punctures is finite is shown using the same argument as in the proof of Proposition $4.49$ in \cite{FishHoferFeral}. See Proposition \ref{prop:limitCurveFinitePunctures}.

The details end up being rather technical. Results of Fish-Hofer from \cite{FishHoferExhaustive} and \cite{FishHoferFeral} are used throughout the proof. 

\subsection{Relationship to work of Fish-Hofer} It is worth noting that Fish and Hofer carry out a similar strategy in the proof of their main dynamical result in \cite{FishHoferFeral}, which concerns the Hamiltonian flow on an energy level $M$ in $W = \mathbb{CP}^2$ which lies in the complement of a complex line. 

They construct an almost-Hermitian manifold essentially identical to the manifold $$(\widetilde{W}_-, J_-, g_-)$$
that we construct, with a cylindrical end modeled on $[0, \infty) \times M$. 

Then they show that there is a sequence of pseudoholomorphic spheres
$$\mathbf{w}_i = (w_i, S^2, \sqrt{-1}, \widetilde{W}_-, J_-, \emptyset, \emptyset)$$
ascending into the cylindrical end. 

Their existence result does not use Gromov-Witten invariants or any other pseudoholomorphic curve-counting invariants. Instead, they use the fact that moduli spaces of pseudoholomorphic spheres on $\mathbb{CP}^2$ are extraordinarily well-behaved, along with the fact that $\widetilde{W}_-$ already contains a $J_-$-holomorphic sphere in the interior. The latter fact implies that the required moduli space of spheres is not empty. Also note that in their case it is not required for $J_-$ to be tamed by some ambient symplectic form. 

An application of exhaustive Gromov compactness for appropriate trimmings of these pseudoholomorphic curves produces a feral pseudoholomorphic curve in an $\eta$-adapted cylinder over $M$. 

A strategy like this is unlikely to work for arbitrary energy levels in arbitrary symplectic manifolds $(W, \Omega)$. Moduli spaces of pseudoholomorphic curves in the non-compact manifold $\widetilde{W}_-$ are not in general as well-behaved as in the case where $W$ is $\mathbb{CP}^2$. Indeed, we have no way of telling, for general hypersurfaces $M$, whether any such moduli space is even non-empty. 

Our contribution is the extension of the neck-stretching procedure to work in the very general setting of an arbitrary hypersurface in a symplectic manifold, which yields a plethora of new examples of feral pseudoholomorphic curves, in addition to the dynamical applications arising from the results in \cite{prequel}. 

This more general setting gives rise to additional issues that are not encountered in \cite{FishHoferFeral}. 

First, our sequence of trimmed pseudoholomorphic curves may have many components in their domain, so we need to establish a priori control on the number of components in each curve in the sequence. In \cite{FishHoferFeral}, the initial sequence of pseudoholomorphic curves (before trimming) simply have domain given by the Riemann sphere.

Second, as we mentioned before, in order to leverage pseudoholomorphic curve-counting invariants we must show that the almost-complex structures $\bar{J}_L$ on the stretched manifolds $W_L$ are tamed by some ambient symplectic form. This result, given in Proposition \ref{prop:tameJ} using a technical construction from Lemma \ref{lem:tameJ1}, requires careful analysis which is not required in \cite{FishHoferFeral}. The careful analysis in Lemma \ref{lem:tameJ1} is also required to show that the pseudoholomorphic curves $\mathbf{w}_i$ have uniformly bounded $\omega$-energy, among other results. 

\section{Gromov-Witten invariants} \label{sec:gw}

\subsection{Review of the invariant} We give a description of Gromov-Witten invariants of closed symplectic manifolds, concluding with a definition of the invariant $\Psi^{W,M}$ from the statement of Theorem \ref{thm:mainExample}.

The papers \cite{FukayaOno99, Pardon16, HWZ17} form a short list of references regarding various constructions of Gromov-Witten invariants. We will describe the construction in \cite{HWZ17}, following the exposition in \cite{Schmaltz19}. 

Fix a homology class $A \in H_2(W; \mathbb{Z})$, integers $G$ and $m$ greater than or equal to zero, and a tame almost-complex structure $J$. The \textbf{Gromov-Witten moduli space} 
$$\mathcal{S}_{A, G, m+2}(p, J)$$
of \cite{HWZ17} is, as a set, the set of all stable, connected pseudoholomorphic curves
$$\mathbf{u} = (u, C, j, W, J, D, \mu)$$
where $\text{Genus}_{\text{arith}}(C, D) = G$, $\#\mu = m+2$, and $u: C \to W$ is a map satisfying a \emph{perturbed} $J$-holomorphic curve equation and pushing forward $[C]$ to $A$. The perturbation parameter is denoted by $p$. 

The \textbf{polyfold regularization theorem} \cite[Theorem $15.3.7$, Corollary $15.3.10$]{BigPolyfoldBook} enables one to give $\mathcal{S}_{A, G, m+2}(p, J)$ the structure of a compact, oriented ``weighted branched suborbifold'' of an ambient polyfold which we will denote by $\mathcal{Z}_{A, g, m+2}$. 

Informally, a weighted branched orbifold is an orbifold equipped with structures necessary to define ``branched integrals'' of differential forms. 

Denote by 
$$\overline{\mathcal{M}}_{G, m+2}$$
the \textbf{Deligne-Mumford orbifold}, which as a set is equal to the set of stable marked nodal Riemann surfaces $(C, j, D, \mu)$ such that $\text{Genus}_{\text{arith}}(C, D) = G$, $\#\mu = m+2$, and the surface constructed by performing connect sums at each pair of nodal points is connected. The Gromov-Witten moduli space has a natural map
$$\pi: \mathcal{S}_{A, G, m+2}(p, J) \to \overline{\mathcal{M}}_{G, m+2}$$
given by forgetting the data of the $J$-holomorphic map and contracting any unstable components in the underlying marked nodal Riemann surface. 

For every $k \in \{1, \ldots, m+2\}$ is also an evaluation map
$$\text{ev}_k: \mathcal{S}_{A, G, m+2}(p, J) \to W$$
given by evaluation of the $J$-holomorphic map at the $k$th marked point. 

We can now write down the polyfold Gromov-Witten invariant.

\begin{defn}
\label{defn:polyfoldGromovWitten} Fix a homology class $A \in H_2(W; \mathbb{Z})$, integers $G$ and $m$ greater than or equal to zero, and a tame almost-complex structure $J$.

The \textbf{polyfold Gromov-Witten invariant} is the homomorphism 
$$\Psi^W_{A,G,m+2}: H_*(W; \mathbb{Q})^{\otimes m+2} \otimes H_*(\overline{\mathcal{M}}_{G, m+2}; \mathbb{Q}) \to \mathbb{Q}$$
defined by the branched integral of sc-smooth differential forms
$$\Psi^W_{A,G,m+2}(\alpha_1, \ldots, \alpha_{m+2}; \beta) = \int_{\mathcal{S}_{A, G, m+2}(p, J)} \bigwedge_{k=1}^{m+2} \text{ev}_k^*\text{PD}(\alpha_k) \wedge \pi^*\text{PD}(\beta).$$
\end{defn}

It is a consequence of the polyfold regularization theorem \cite[Theorem $15.3.7$, Corollary $15.3.10$]{BigPolyfoldBook} that $\Psi^W_{A,g,m+2}$ as described in Definition \ref{defn:polyfoldGromovWitten} is independent of the choice of abstract perturbation $p$ or tame almost-complex structure $J$. 

\subsection{Gromov-Witten invariants as intersection numbers} 

Informally, 
$$\Psi^W_{A,G,m+2}(\alpha_1, \ldots, \alpha_{m+2}; \beta)$$
can be thought of as a count of $J$-holomorphic curves intersecting (rationally-weighted) submanifolds representing the classes $\{\alpha_k\}_{k = 1}^{m+2}$ with the domain Riemann surface mapping via $\pi$ into a suborbifold of $\overline{\mathcal{M}}_{G, m+2}$ which represents the class $\beta$. 

This interpretation was put on rigorous footing by Schmaltz in \cite{Schmaltz19}.

\begin{thm}
\label{thm:polyfoldGWIntersection} \cite[Theorem $1.6$, Corollary $1.7$]{Schmaltz19} Fix a homology class $A \in H_2(W; \mathbb{Z})$, integers $G$ and $m$ greater than or equal to zero, and a tame almost-complex structure $J$.

Fix classes $\{\alpha_k\}_{k=1}^{m+2}$ in $H_*(W; \mathbb{Q})$ and $\beta \in H_*(\overline{\mathcal{M}}_{G, m+2}; \mathbb{Q})$ such that the classes $\{\alpha_k\}_{k=1}^{m+2}$ are represented by submanifolds $Y_k \subset W$ and the class $\beta$ is represented by an embedded suborbifold $O \subset \overline{\mathcal{M}}_{G, m+2}$ whose normal bundle has trivial fiberwise isotropy.

The polyfold Gromov-Witten invariant 
$$\Psi^W_{A,G,m+2}(\alpha_1, \ldots, \alpha_{m+2}; \beta)$$
may be defined as the intersection number
$$(\text{ev}_1 \times \ldots \times \text{ev}_{m+2} \times \pi)|_{\mathcal{S}_{A, G, m+2}(p, J)} \cdot (Y_1 \times \ldots \times Y_k \times O).$$
\end{thm}

As noted in \cite{Schmaltz19}, Thom's resolution of the Steenrod problem for manifolds and the solution of the corresponding problem for orbifolds in \cite{Schmaltz19} imply that the entire Gromov-Witten invariant can be defined using the above theorem, by taking appropriate rational multiples of the classes $\{\alpha_k\}_{i=1}^{m+2}$ and $\beta$ so that they are represented by submanifolds/suborbifolds, and then using the fact that the Gromov-Witten invariant is a homomorphism. Theorem \ref{thm:polyfoldGWIntersection} yields the following concrete existence result for $J$-holomorphic curves. 

\begin{prop}
\label{prop:GWExistence} Fix a homology class $A \in H_2(W; \mathbb{Z})$, integers $G$ and $m$ greater than or equal to zero such that $2G + m \geq 1$, and a tame almost-complex structure $J$.

Fix classes $\{\alpha_k\}_{k=1}^{m+2}$ in $H_*(W; \mathbb{Q})$ and $\beta \in H_*(\overline{\mathcal{M}}_{G, m+2}; \mathbb{Q})$ such that the polyfold Gromov-Witten invariant
$$\Psi^W_{A, G, m+2}(\alpha_1, \ldots, \alpha_{m+2}; \beta)$$
is not equal to zero. 

Suppose that the classes $\{\alpha_k\}_{k=1}^{m+2}$ are represented by submanifolds $Y_k \subset W$. 

Then for any tame almost-complex structure $J$, there is a stable, connected pseudoholomorphic curve
$$\mathbf{u} = (u, C, j, W, J, D, \mu)$$
where $\text{Genus}_{\text{arith}}(C, D) = G$, $\#\mu = m+2$, and $u: C \to W$ is a $J$-holomorphic map pushing forward $[C]$ to $A$. Moreover, the image $u(C)$ intersects each of the submanifolds $\{Y_k\}_{k=1}^{m+2}$.
\end{prop}

\begin{proof}
Theorem \ref{thm:polyfoldGWIntersection} and the invariance of the polyfold Gromov-Witten invariants imply that for a sequence of abstract perturbations $p_k \to 0$, we obtain stable, connected pseudoholomorphic curves
$$\mathbf{u}_k = (u_k, C_k, j_k, W, J, D_k, \mu_k)$$
satisfying the desired properties, but instead with $u_k$ satisfying the $J$-holomorphic curve equation perturbed by $p_k$. Applying the Gromov compactness theorem to the sequence $\mathbf{u}_k$, along with the fact that $p_k \to 0$, immediately yields the desired pseudoholomorphic curve $\mathbf{u}$. 
\end{proof}

\subsection{The invariant $\Psi^{W,M}$} We are interested in guaranteeing the existence of pseudoholomorphic curves whose images cross the hypersurface $M = H^{-1}(0) \subset W$. 

This motivates the following definitions. Denote by $W_-$ and $W_+$ the two components of $W \setminus M$, defined concretely as the sets $H^{-1}((-\infty, 0))$ and $H^{-1}((0, \infty))$. 

\begin{defn}
A pair of integral, non-torsion cohomology classes $\alpha_+, \alpha_- \in H^*(W; \mathbb{Z})$ are \emph{separated} by $M$ if they are Poincar\'e dual to closed submanifolds $Y_+$ and $Y_-$ such that $Y_+ \subset W_+$ and $Y_- \subset W_-$, respectively.
\end{defn}

We denote the set of all pairs $(\alpha_+, \alpha_-)$ of cohomology classes separated by $M$ by the set
$$\text{Sep}(M) \subset H^*(W; \mathbb{Z}) \times H^*(W; \mathbb{Z}).$$

We package together the Gromov-Witten invariants with insertions from $\text{Sep}(M)$ into a single invariant denoted by $\Psi^{W,M}$. 

\begin{defn} \label{defn:GWHypersurface}
Define the number
$$\Psi^{W,M} \in \{0,1\}$$
to be equal to $1$ if there exists some $G \geq 0$, $m \geq 0$, $A \in H_2(X; \mathbb{Z})$ and $(\alpha_+, \alpha_-) \in \text{Sep}(M)$ such $2G+m \geq 1$ and
$$\Psi^W_{A,G,m+2}(\alpha_+, \alpha_-, \ldots): H^*(X^k; \mathbb{Q}) \otimes H^*(\overline{\mathcal{M}}_{g,m+2}; \mathbb{Q})\to \mathbb{Q}$$
is not the zero map. 

On the other hand, if this is the zero map for any $A$, $G$, $m$, and $(\alpha_+, \alpha_-) \in \text{Sep}(M)$, set $\Psi^{W,M} = 0$.
\end{defn}

The invariant $\Psi^{W,M}$ was first considered in a similar form by Liu-Tian in \cite{LiuTian00}, generalizing a result by Hofer-Viterbo \cite{HoferViterbo92} using the formalism of Gromov-Witten theory. They showed that non-vanishing of this invariant implies the Weinstein conjecture.

\begin{thm}
\cite[Theorem $1.1$]{LiuTian00} Suppose $M$ is a contact-type hypersurface and 
$$\Psi^{W,M} \neq 0.$$

Then the Reeb vector field on $M$ has a periodic orbit. 
\end{thm}

\subsection{Existence of pseudoholomorphic curves} 

Proposition \ref{prop:GWExistence} implies the following existence result for pseudoholomorphic curves when $\Psi^{W,M}$ does not vanish.

\begin{cor}
\label{cor:GWCrossingExistence} Suppose
$$\Psi^{W,M} \neq 0.$$

Then for any tame almost-complex structure $J$, the following is true. There exists some $G \geq 0$, $m \geq 0$, $A \in H_2(X; \mathbb{Z})$, and submanifolds $Y_+ \subset W_+$, $Y_- \subset W_-$ and a stable, connected pseudoholomorphic curve
$$\mathbf{u} = (u, C, j, W, J, D, \mu)$$
where $\text{Genus}_{\text{arith}}(C, D) = G$, $\#\mu = m+2$, and $u: C \to W$ is a $J$-holomorphic map pushing forward $[C]$ to $A$. Moreover, the image $u(C)$ intersects each of the submanifolds $Y_+$ and $Y_-$. 
\end{cor}

\begin{exe} \label{exe:gwConnectedness}
Denote by $e \in H_0(W; \mathbb{Z})$ the point class. 
Then Proposition \ref{prop:GWExistence} implies that non-vanishing of 
$$\Psi^W_{A, G, m+2}(e, e; -)$$
implies that for any two points $p$ and $q$ in $W$, there is a stable, connected $J$-holomorphic curve 
$$\mathbf{u} = (u, C, j, W, J, D, \mu)$$
such that $C$ has arithmetic genus $G$, $u$ pushes forward $[C]$ to $A$, and the image $u(C)$ passes through $p$ and $q$. 
\end{exe}

Example \ref{exe:gwConnectedness} motivates the following definition, which is used to prove the dynamical result stated in Proposition \ref{prop:application1}.

\begin{defn} \label{defn:gwConnectedness}
A symplectic manifold where 
$$\Psi^W_{A, 0, m+2}(e, e; -)$$
does not vanish for some choice of $A$ and $m \geq 1$ is called \emph{rationally connected}. Informally, this means that there is a $J$-holomorphic sphere through any two points. 

A symplectic manifold where 
$$\Psi^W_{A, G, m+2}(e, e; -)$$
does not vanish for some choice of $A$ and $m$ is called \emph{$\text{GW}_G$-connected}. 
\end{defn}

Note that $(e, e) \in \text{Sep}(M)$ since by definition $M$ is a separating hypersurface. We conclude the following statement.

\begin{lem}\label{lem:gwConnectedness}
If a symplectic manifold $(W, \Omega)$ is $\text{GW}_G$-connected for any $G \geq 0$, then for any closed, separating hypersurface $M \subset W$, 
$$\Psi^{W, M} \neq 0.$$
\end{lem}

\section{Taubes' Gromov invariant} \label{sec:taubesGromov}

\subsection{Review of the invariant} 

Let $(W, \Omega)$ be a symplectic $4$-manifold with $b_2^+(W) > 1$.

Taubes' Gromov invariant, constructed by Taubes in \cite{Taubes96}, is a map
$$\text{Gr}: H_2(W; \mathbb{Z}) \to \mathbb{Z}.$$

For each $A \in H_2(W; \mathbb{Z})$, the integer $\text{Gr}(A)$ is a count of pseudoholomorphic curves that takes a more geometric approach than the definition of the Gromov-Witten invariants. 

Recall that an important part of the data of a pseudoholomorphic curve
$$\mathbf{u} = (u, C, j, W, J, D, \mu)$$ 
is the pseudoholomorphic map $u$. Suppose for simplicity that $D$ and $\mu$ are empty and $C$ is closed and connected. It is a standard fact (see \cite[Section $2.5$]{McduffSalamon12}) that the map $u$ factors into a composition $u' \circ \phi$ where 
$$\phi: (C, j) \to (C', j')$$
is a holomorphic branched covering of degree $m \geq 1$ and
$$u': (C', j) \to (W, J)$$
is a $J$-holomorphic map that is injective on an open, dense subset of $C'$. The degree $m$ of the covering is called the \emph{multiplicity} of $u$.

Two pseudoholomorphic curves of the type above, from the perspective of Taubes' Gromov invariant, are considered equivalent if the pseudoholomorphic maps have the same image and multiplicity.  

Given this consideration, it is more accurate to say that the primary objects of concern in Taubes' Gromov invariant are \emph{pseudoholomorphic submanifolds} rather than pseudoholomorphic curves as in Gromov-Witten theory. We define pseudoholomorphic submanifolds below.

\begin{defn}
\label{defn:jHolomorphicSubmanifold} Let $(W, J)$ be an almost-complex manifold. A smooth, embedded two-dimensional submanifold $\Sigma$ is a \textbf{pseudoholomorphic submanifold} or \textbf{$J$-holomorphic submanifold} if it satisfies either of the two equivalent conditions below:
\begin{itemize}
    \item For every point $p \in \Sigma$, the tangent plane $T_p\Sigma \subset T_p W$ is invariant under the action of the almost-complex structure $J$. 
    \item There is a pseudoholomorphic curve
    $$\mathbf{u} = (u, C, j, W, J, \emptyset, \emptyset)$$
    where $u$ is an embedding and the image $u(C)$ is equal to $\Sigma$. 
\end{itemize}
\end{defn}

The first definition in Definition \ref{defn:jHolomorphicSubmanifold} is more geometric, while the second definition is useful for the construction of moduli of pseudoholomorphic submanifolds. The definition of Taubes' Gromov invariant relies heavily on two properties of pseudoholomorphic curves unique to $4$-dimensional symplectic manifolds: the adjunction inequality and positivity of intersection.

\begin{thm}[Adjunction inequality]
\label{thm:adjunctionInequality} Let $(W, \Omega)$ be a closed symplectic manifold of dimension $4$, and let 
$$\mathbf{u} = (u, C, j, W, J, \emptyset, \emptyset)$$
be a $J$-holomorphic curve where $J$ is $\Omega$-tame, $C$ is a closed, connected surface of genus $G$, and $u$ represents the class $A$.

Then 
$$2 - 2G + A \cdot A - \text{PD}(c_1(W)) \cdot A \geq 0$$
with equality if and only if $u$ is an embedding. 
\end{thm}

A proof of Theorem \ref{thm:adjunctionInequality} can be found in \cite{McDuff91}. 

\begin{thm}[Positivity of intersection]
\label{thm:positivityOfIntersection} 
Let $(W, \Omega)$ be a closed symplectic manifold of dimension $4$, and let  $J$ be an $\omega$-tame almost-complex structure. 
Let
$$\mathbf{u}_1 = (u_1, C_1, j_1, W, J, \emptyset, \emptyset)$$
and
$$\mathbf{u}_2 = (u_2, C_2, j_2, W, J, \emptyset, \emptyset)$$
be two pseudoholomorphic curves where $C_1$ and $C_2$ are closed, connected surfaces and $u_1$ and $u_2$ represent homology classes $A_1$ and $A_2$, respectively. Furthermore, assume that $u_1$ and $u_2$ do not have the same image. 

Then the set of intersections of $u_1(C_1)$ and $u_2(C_2)$ is discrete. Moreover, for each point $p$ lying in $u_1(C_1) \cap u_2(C_2)$, the intersection multiplicity 
$$m_p(u_1, u_2)$$
is positive, and equal to $1$ if and only if $u_1$ and $u_2$ are embeddings onto their images near $p$, and moreover $u_1(C_1)$ and $u_2(C_2)$ intersect transversely at $p$. 
\end{thm}

A proof of Theorem \ref{thm:positivityOfIntersection} can be found in \cite{McDuff91}. 

We now proceed with the definition of Taubes' Gromov invariant. Fix a homology class $A \in H_2(W; \mathbb{Z})$, where $(W, \Omega)$ is a closed, symplectic $4$-manifold with $b_2^+(W) > 1$. 

We define from $A$ an integer
$$I(A) = \frac{1}{2}(\text{PD}(c_1(W)) \cdot A + A \cdot A)$$
where the dot indicates the intersection pairing of $W$. Using the adjunction inequality in Theorem \ref{thm:adjunctionInequality}, the genus of a pseudoholomorphic submanifold reprenting $A$ can be recovered from $I(A)$. 

\begin{lem}
\label{lem:jHolomorphicSubmanifoldGenusBound} Let $\Sigma$ be a $J$-holomorphic submanifold in $W$ representing the homology class $A$. Then the genus $G$ of $\Sigma$ is 
$$G = 1 - I(A) + A \cdot A.$$
\end{lem}

For any integer $d \in \mathbb{Z}$, write $\mathcal{A}_d$ for the set of pairs $(J, P)$ where $J$ is an $\Omega$-tame almost-complex structure and $P$ is either the empty set if $d \geq 0$ or a set of $d$ distinct, unordered points in $W$ if $d \geq 1$. 

We give $\mathcal{A}_d$ a topology by considering it as a subset of the space 
$$C^\infty(W; \text{Hom}(TW, TW)) \times \text{Sym}^d(W)$$
where $\text{Sym}^d(W)$ is the $d$-fold symmetric product of $W$ if $d \geq 1$ and a single point if $d \leq 0$. 

\begin{defn}
\label{defn:grGenerators} Fix $A \in H_2(W; \mathbb{Z})$ and a pair $(J, P) \in \mathcal{A}_{I(A)}$. 

Then the set 
$$\mathcal{H}(A, J, P)$$
is the set where each element $\mathfrak{C}$ is a finite set of pairs $\{(\Sigma_k, m_k)\}_{k=1}^N$ of connected $J$-holomorphic submanifolds $\Sigma_k$ and natural numbers $m_k \geq 1$ satisfying the following conditions:
\begin{itemize}
    \item The sum $\sum_{k=1}^N m_k[\Sigma_k]$ is equal to $A$.
    \item The submanifolds $\{\Sigma_k\}_{k=1}^N$ are all pairwise disjoint.
    \item The union
    $$\cup_{i=1}^N \Sigma_k$$
    contains the set $P$. 
    \item If, for any $k$, the submanifold $\Sigma_k$ is a sphere with negative self-intersection, the multiplicity $m_k$ is equal to $1$. 
\end{itemize}
\end{defn}

The Gromov invariant is defined only for certain choices of $(J, P) \in \mathcal{A}_{I(A)}$. We list all of the properties required of such a choice $(J, P)$ in the definition below. Two of these properties make reference to \emph{nondegeneracy} and \emph{$n$-nondegeneracy} of a pseudoholomorphic submanifold $\Sigma$. We point the reader to \cite[Definition $2.1$]{Taubes96} and \cite[Definition $4.1$]{Taubes96} for definitions of these notions, although we sketch them below.

Fix $(J, P) \in \mathcal{A}_{I(A)}$. Then a connected $J$-holomorphic submanifold $\Sigma$ representing the class $A$ and containing the set $P$ is nondegenerate when the direct sum of its normal deformation operator and the linearized evaluation map at $P$ has no kernel. A geometric consequence is that $\Sigma$ is isolated in the space of connected pseudoholomorphic submanifolds that represent $A$ and contain $P$. 

A connected pseudoholomorphic submanifold of genus $1$ representing the class $A$ is $n$-nondegenerate if the pullback of its associated normal deformation operator by an $n$-fold cover of tori has no kernel for any such cover. A geometric consequence is that there is no sequence of connected pseudoholomorphic submanifolds $\Sigma_k$ of genus $1$ collapsing onto an $n$-fold cover of $\Sigma$. 

\begin{defn}
\label{defn:admissibleJ} \cite[Definition $4.2$]{Taubes96} Fix an integer $d \geq 0$ and a pair $(J, P) \in \mathcal{A}_d$. The pair $(J, P)$ is \emph{admissible} if the following conditions hold for any class $A$ with $I(A) \leq d$:
\begin{itemize}
    \item There are finitely many connected pseudoholomorphic submanifolds representing $A$ that pass through $I(A)$ points in the set $P$.
    \item Every connected pseudoholomorphic submanifold as in the item above is nondegenerate.
    \item There are no connected pseudoholomorphic submanifolds representing $A$ that pass through more than $I(A)$ points in the set $P$. 
    \item There is an open neighborhood in $\mathcal{A}_d$ containing $(J, P)$ such that for every $(J', P')$ in this neighborhood, the first three items are satisfied. Moreover, the number of pseudoholomorphic manifolds as in the $(J', P')$ version of the first item coincides with the number in the $(J, P)$ version.
    \item There are no connected pseudoholomorphic submanifolds representing $A$ if $I(A) < 0$. 
    \item If $c_1(W) \cdot A = A \cdot A = 0$, then any pseudoholomorphic submanifold as in the first point is $n$-nondegenerate for every $n \in \mathbb{Z}$. 
\end{itemize}
\end{defn}

The work \cite{Taubes96} defines $\mathcal{A}_d$ as a set of pairs $(J, P)$ where $J$ is \emph{$\Omega$-compatible} rather than $\Omega$-tame. The definition \cite[Definition $4.2$]{Taubes96} of an admissible pair is the same apart from this distinction and minor cosmetic changes. 

In any case, admissible pairs are plentiful, as the following proposition shows.

\begin{prop}
\label{prop:genericAdmissiblePairs} \cite[Proposition $4.3$]{Taubes96} Fix $d \geq 0$. The the set of admissible pairs in $\mathcal{A}_d$ is a Baire set, i.e. it is a countable intersection of open and dense subsets. 
\end{prop}

Proposition \ref{prop:genericAdmissiblePairs} is a restatement of the first statement in \cite[Proposition $4.3$]{Taubes96}, except again our space $\mathcal{A}_d$ uses tame almost-complex structures rather than compatible almost-complex structures. 

Before sketching the proof of Proposition \ref{prop:genericAdmissiblePairs}, we recall the following property of tame/compatible almost-complex structures $J$ in closed symplectic manifolds $(W, \Omega)$. If $J$ is compatible, then there is a natural almost-Hermitian structure $(J, g)$ on $W$ where the metric $g$ is defined by
$$g(-, -) = \Omega(-, J-).$$

More generally, if $J$ is tame, then we can define $(J, g)$ by writing
$$g(-, -) = \frac{1}{2}(\Omega(J-, -) + \Omega(-, J-)).$$

In either case, let 
$$\mathbf{u} = (u, C, j, W, J, \emptyset, \emptyset)$$
be a pseudoholomorphic curve with closed domain such that $u$ pushes forward the fundamental class $[C]$ to the homology class $A$. Then a quick computation shows that the area of $\mathbf{u}$ is determined by the homology class:
$$\text{Area}_{u^*g}(C) = \langle \Omega, A \rangle.$$

In particular, if
$$\mathbf{u}_k = (u_k, C_k, j_k, W, J, \emptyset, \emptyset)$$
is a sequence of $J$-holomorphic curves with closed domain all representing the same homology class $A$, there is automatically a uniform bound on their area with respect to $(g, J)$ whenever $J$ is tame/compatible and $g$ is the canonically introduced metric.

\begin{proof}[Proof of Proposition \ref{prop:genericAdmissiblePairs}]

Proposition \ref{prop:genericAdmissiblePairs} can be proved by copying the proof of Proposition $4.3$ in \cite{Taubes96}, changing ``compatible'' to ``tame'' everywhere, and applying minor cosmetic changes where our notation differs from that of \cite{Taubes96}. 

The first main idea in the proof is the use of the Sard-Smale theorem and the implicit function theorem for smooth maps between Banach manifolds with Fredholm linearization. Taubes writes down a ``universal moduli space'' consisting of the relevant moduli spaces of pseudoholomorphic curves for each pair in his version of $\mathcal{A}_d$. The standard projection from this universal space to $\mathcal{A}_d$ is shown to have surjective differential everywhere, which via an argument using the Sard-Smale theorem and the implicit function theorem shows the second, third, fourth and fifth bullets in the proposition hold for $(J, P)$ lying in a Baire subset of $\mathcal{A}_d$. 

The key technical step is the proof that the projection has surjective differential, which works just as well whether $\mathcal{A}_d$ is defined using tame almost-complex structures versus compatible almost-complex structures. 

The first bullet in the proposition is proved using the Gromov compactness theorem, along with another Sard-Smale argument to eliminate the possibility of a sequence of pseudoholomorphic submanifolds limiting to something undesirable. The key reason that Taubes uses compatible $J$ in \cite{Taubes96} is to ensure that any sequence of connected pseudoholomorphic submanifolds representing a homology class $A$ will have uniformly bounded area with respect to the almost-Hermitian metric
$$g(-, -) = \Omega(-, J-)$$
which is necessary to appeal to the Gromov compactness theorem. 

However, as we discussed above the beginning of the proof, such a statement holds for tame $J$ as well, by using the metric
$$g(-, -) = \frac{1}{2}(\Omega(-, J-) - \Omega(J-, -))$$ 
so there is no issue in using tame $J$ instead. 

The sixth bullet can be proved either by repeating the proof in \cite{Taubes96} and again changing ``compatible'' to ``tame'' and applying minor cosmetic changes, or by applying \cite[Theorem $B$]{WendlSuperrigidity} from the recent work of Wendl on equivariant transversality for pseudoholomorphic curves. 
\end{proof}

The following proposition is proved using Definition \ref{defn:admissibleJ} and the adjunction inequality in Theorem \ref{thm:adjunctionInequality}. 

\begin{prop} \label{prop:admissibleJConsequences}
\cite[Proposition $4.3$]{Taubes96} Fix $A \in H_2(W; \mathbb{Z})$ with $I(A) \geq 0$ and suppose $(J, P) \in \mathcal{A}_{I(A)}$ is admissible. Then the set $\mathcal{H}(A, J, P)$ satisfies the following properties:
\begin{itemize}
    \item $\mathcal{H}(A, J, P)$ is finite.
    \item Any element $\mathfrak{C} = \{(\Sigma_k, m_k)\}$ has the property that, for every $k$, $\Sigma_k$ is nondegenerate when $m_k = 1$ and, if $m_k > 1$, $\Sigma_k$ is has genus $1$, has
    $$c_1(W) \cdot [\Sigma_k] = [\Sigma_k] \cdot [\Sigma_k] = 0$$
    and is $m_k$-nondegenerate. 
    \item If $(J', P')$ is sufficiently close to $(J, P)$ in $\mathcal{A}_{I(A)}$, then $\mathcal{H}(A, J', P')$ and $\mathcal{H}(A, J, P)$ have the same number of elements. 
\end{itemize}
\end{prop}

\begin{proof}
The third point follows immediately from Definition \ref{defn:admissibleJ}. 

We prove the second point before the first point. Pick any element
$$\mathfrak{C} = \{(\Sigma_k, m_k)\}_{k=1}^N$$
of $\mathcal{H}(A, J, P)$. 

Set $A_k$ to be the class represented by $\Sigma_k$ for every $k$. Since $(J, P)$ is admissible, we find that $\Sigma_k$ contains at most $I(A_k)$ points in $P$. 

However, the union of the $\Sigma_k$ contain all of the points in $P$, so we find
$$I(A) \leq \sum_{k=1}^N I(A_k).$$

On the other hand, since any two of the submanifolds $\Sigma_k$ are disjoint, by positivity of intersections (Theorem \ref{thm:positivityOfIntersection}) we find 
$$A_{k_1} \cdot A_{k_2} = 0$$
for any $k_1 \neq k_2$. 

Then we compute
$$I(A) = \sum_k m_k^2 I(A_k).$$

It follows that
$$\sum_k m_k^2 I(A_k) \leq \sum_{k=1}^N I(A_k).$$

Since $(J, P)$ is admissible, each of the $I(A_k)$ are nonnegative. We conclude that if $I(A_k) > 0$, we must have $m_k = 1$. This proves the second point. 

The first point now follows immediately if we can show that there are finitely many homology classes $A'$ represented by a member $(\Sigma, m)$ of an element $\mathfrak{C} \in \mathcal{H}(A, J, P)$. 

Again fix an element
$$\mathfrak{C} = \{(\Sigma_k, m_k)\}_{k=1}^N$$
of $\mathcal{H}(A, J, P)$. Again write $A_k$ for the homology class represented by each $\Sigma_k$. Write $G_k$ for the genus of $\Sigma_k$. Then the formula in Lemma \ref{lem:jHolomorphicSubmanifoldGenusBound} shows
$$A_k \cdot A_k = -1 + G_k + I(A_k) \geq -1.$$

Therefore, the self-intersection of every class $A_k$ is bounded below by $-1$, and is equal to $-1$ exactly when $G_k = I(A_k) = 0$. In other words, when $\Sigma_k$ is an embedded sphere of self-intersection $-1$.

Let $b_2^-(W)$ be the maximum size of a negative definite subspace for the intersection form. It follows that at most $b_2^-(W)$ of the $\Sigma_k$ can have negative self-intersection, in which case they are spheres with $I(A_k) = 0$. 

It follows that, for any $k$ such that $A_k \cdot A_k \geq 0$, we have $A_k \cdot A_k \leq A \cdot A + b_2^-(W)$. This in turn implies that $G_k \leq A \cdot A + b_2^-(W) + 1$. 

Define a Riemannian metric by 
$$g(-, -) = \frac{1}{2}(\Omega(-, J-) - \Omega(J-, -)).$$

The identity given above the sketch of Proposition \ref{prop:genericAdmissiblePairs} shows for any $k$ that
$$\text{Area}_g(\Sigma_k) \leq \sum_{k=1}^N m_k\text{Area}_g(\Sigma_k) = \langle \Omega, A \rangle.$$

It follows that there is are a priori area and genus bounds on $\Sigma$ for any member $(\Sigma, m)$ of an element $\mathfrak{C} \in \mathcal{H}(A, J, P)$. 

Suppose for the sake of contradiction that there is a sequence of distinct cohomology classes $A_k$ represented by a $J$-holomorphic submanifold $\Sigma_k$, where $(\Sigma_k, m_k)$ is a member of an element of $\mathcal{H}(A, J, P)$. 

By our a priori area and genus bounds, a subsequence of these $\Sigma_k$ converge in the Gromov sense. This implies, once we pass to this sequence, the cohomology classes $A_k$ must coincide for sufficiently large $k$. This contradicts our initial assumption, and there are therefore only finitely many cohomology classes $A_k$ represented in elements of $\mathcal{H}(A, J, P)$. 
\end{proof}

Given Proposition \ref{prop:admissibleJConsequences}, the Gromov invariant is defined via the following theorem, which is a restatement of \cite[Theorem $1.1$]{Taubes96} with more exposition added in the statement. 

\begin{thm}
\cite[Theorem $1.1$]{Taubes96} \label{thm:gromovInvariantDefinition} Fix $A \in H_2(W; \mathbb{Z})$ and suppose $(J, P) \in \mathcal{A}_{I(A)}$ is admissible. Then there is an integer-valued function $R$ on the set of pairs $(\Sigma, m)$ of a connected $J$-holomorphic submanifold $\Sigma$ and an integer $m \geq 1$ such that the integers 
$$\text{Gr}(A, J, P) = \sum_{\mathfrak{C} \in \mathcal{H}(A, J, P)} \prod_{(\Sigma_k, m_k) \in \mathfrak{C}} r(\Sigma_k, m_k)$$
are all equal regardless of the choice of admissible $(J, P)$, with the common integer being the \textbf{Gromov invariant} $\text{Gr}(A)$. 
\end{thm}

The function $R(\Sigma, 1)$ takes values in $\pm 1$, and is defined using the spectral flow from a complex-linear Cauchy-Riemann operator to the normal deformation operator of $C$. It is $1$ if this spectral flow is even, and $-1$ if it is odd. This function can be regarded as an orientation on the zero-dimensional space of pseudoholomorphic submanifolds $\Sigma$ such that the pair $(\Sigma, 1)$ appears in some element of $\mathcal{H}(A, J, P)$. 

The definition of $R(C, m)$ is more complicated. We refer the reader to \cite[Section $3$]{Taubes96} for a definition. 

The proof of Theorem \ref{thm:gromovInvariantDefinition} was written down in \cite{Taubes96} in the case where $\mathcal{A}_d$ is defined using compatible almost-complex structures rather than tame almost-complex structures. 

Again, however, the proof works when $\mathcal{A}_d$ is defined using tame almost-complex structures by verbatim applying the argument in \cite{Taubes96} and making the necessary cosmetic changes. The only place where the fact that the almost-complex structures are compatible is used is to derive area bounds for the use of the Gromov compactness theorem. Again, it suffices to use tame almost-complex structures, as we noted, for example, in our sketch of Proposition \ref{prop:genericAdmissiblePairs}. 

\subsection{Existence of pseudoholomorphic curves} 

Now that we have discussed the construction of the Gromov invariant, we can write down the existence result for pseudoholomorphic curves required for the proof of Theorem \ref{thm:swExample}. 

Using Gromov compactness, we can deduce the following proposition is true when the Gromov invariant does not vanish. 

\begin{prop}
\label{prop:nonzeroGr} Let $(W, \Omega)$ be a closed, symplectic $4$-manifold with $b_2^+(W) > 1$. Suppose there is a homology class $A$ such that $\text{Gr}(A) \neq 0$. Then for any tame almost-complex structure $J$ and any set of $I(A)$ points $P$, there is an integer $N \geq 1$ and, for every integer $k$ between $1$ and $N$ inclusive, a stable, connected $J$-holomorphic curve
$$\mathbf{u}_k = (u_k, C_k, j_k, W, J, D_k, \emptyset)$$
and an integer $m_k \geq 1$ such that the following conditions are satisfied:
\begin{itemize}
    \item The sum $\sum_{k=1}^N m_k (u_k)_*[C_k]$ is equal to $A$.
    \item For every $k$, the arithmetic genus of $(C_k, D_k)$ is bounded above by $A \cdot A + b_2^-(W) + 1$. 
    \item For every $k$, the self-intersection of $(u_k)_*[C_k]$ is greater than or equal to $-1$, with equality if and only if $(C_k, D_k)$ has arithmetic genus zero. 
    \item For every $k_1 \neq k_2$, the algebraic intersection number of $(u_{k_1})_*[C_{k_1}]$ and $(u_{k_2})_*[C_{k_2}]$ is zero. 
    \item The set of points $P$ is contained in the union of the images $u_k(C_k)$ over all $k$. 
\end{itemize}
\end{prop}

\begin{proof}
Fix a tame almost-complex structure $J$ and a set of $I(A)$ points $P$. 

Then we can pick a sequence $(J_k, P_k) \in \mathcal{A}_{I(A)}$ of admissible pairs converging to $(J, P)$. 

The fact that $\text{Gr}(A) \neq 0$ implies that, for every $k$, there is an integer $N_k \geq 1$ and, for every integer $\ell$ between $1$ and $N_k$, inclusive, a $J_k$-holomorphic curve
$$\mathbf{u}_k^\ell = (u_k^\ell, C_k^\ell, j_k^\ell, W, J, \emptyset, \emptyset)$$
and an integer $m_k^\ell \geq 1$ such that the following conditions are satisfied:
\begin{itemize}
    \item The sum $\sum_{\ell = 1}^{N_k} m_k^\ell (u_k^\ell)_*[C_k^\ell]$ is equal to $A$ for every $k$. 
    \item For every $k$ and $\ell$, the map $u_k^\ell$ is an embedding.
    \item For every $k$ and $\ell$, the Riemann surface $(C_k^\ell, j_k^\ell)$ is a closed, connected surface with genus equal to $$1 - I( (u_k^\ell)_*[C_k^\ell]) + (u_k^\ell)_*[C_k^\ell] \cdot (u_k^\ell)_*[C_k^\ell]$$ and bounded above by $1 + b_2^-(W) + A \cdot A$. 
    \item The integer $m_k^\ell$ is equal to $1$ if the genus of $C_k^\ell$ is not equal to $1$.
    \item For any $k$ and $\ell_1 \neq \ell_2$, the images $u_k^{\ell_1}(C_k^{\ell_1})$ and $u_k^{\ell_2}(C_k^{\ell_2})$ are disjoint. Moreover, the algebraic intersection number of $(u_k^{\ell_1})_*[C_k^{\ell_1}]$ and $(u_k^{\ell_2})_*[C_k^{\ell_2}]$ is equal to $0$. 
    \item For any $k$, the set of points $P$ is contained in the union of the images $u_k^\ell(C_k^\ell)$ over all $\ell$ between $1$ and $N_k$, exclusive.
\end{itemize}

Every point in the above, apart from the remarks regarding genus in the second point and algebraic intersection in the fourth point, is immediate from the definition of the set $\mathcal{H}(A, J_k, P_k)$ of elements contributing to $\text{Gr}(A)$. 

The formula for the genus in the third point arises from Lemma \ref{lem:jHolomorphicSubmanifoldGenusBound}. The bound for the genus was derived in the proof of Proposition \ref{prop:admissibleJConsequences}. 

The fifth point, regarding algebraic intersection, follows from the fact that $u_k^{\ell_1}(C_k^{\ell_1})$ and $u_k^{\ell_2}(C_k^{\ell_2})$ are disjoint and positivity of intersections for $J_k$-holomorphic curves (Theorem \ref{thm:positivityOfIntersection}). 

To produce a set of $J$-holomorphic curves as in the statement of the proposition, we will first show that the sequence $N_k$ is uniformly bounded in $k$, and then apply the Gromov compactness theorem. 

For any $k$ and $\ell$, use the notation $A_k^\ell$ to denote the homology class $(u_k^\ell)_*[C_k^\ell]$. Use the notation $G_k^\ell$ to denote the genus of $C_k^\ell$, which is equal to $1 - I(A_k^\ell) + A_k^\ell \cdot A_k^\ell$. 

For any $k$, the assumptions above imply that
$$A \cdot A = \sum_{\ell = 1}^{N_k} (m_k^\ell)^2 A_k^\ell \cdot A_k^\ell$$
and
$$I(A) = \sum_{\ell = 1}^{N_k} I(A_k^\ell).$$

Also recall that we must have $I(A_k^\ell) \geq 0$ for every $\ell$ by the admissibility of $(J_k, P_k)$. 

To bound $N_k$, it suffices to show that there are finitely many homology classes among the union of the classes $A_k^\ell$ over all $k$ and all $\ell$ between $1$ and $N_k$, inclusive. 

We can show that there is no sequence $k_n \to \infty$ and a corresponding sequence $\ell_n$ of integers between $1$ and $N_{k_n}'$, inclusive, such that the homology classes $A_{k_n}^{\ell_n}$ are all pairwise distinct. This implies our desired conclusion, since if there were infinitely many distinct $A_k^\ell$, we could construct such a sequence. 

Indeed, the pseudoholomorphic curves $\mathbf{u}_{k_n}^{\ell_n}$ satisfy uniform area bounds of $\langle \Omega, A \rangle$ and uniform genus bounds of $1 - I(A) + A \cdot A$. It follows by the Gromov compactness theorem that, after passing to a subsequence, $\mathbf{u}_{k_n}^{\ell_n}$ converges in the Gromov sense. This implies that the classes $A_{k_n}^{\ell_n}$ must coincide for large $n$, so the homology classes $A_{k_n}^{\ell_n}$ cannot all be pairwise distinct. Now we observe that, due to this finiteness statement, there is some constant $\kappa > 0$ independent of $k$ and $\ell$ such that, for any $k$ and any $\ell$ between $1$ and $N_k$, 
$$\langle \Omega, A_k^\ell \rangle \geq \kappa.$$

Observe that for any $k$ and $\ell$ between $1$ and $N_k$, inclusive, $\langle \Omega, A_k^\ell \rangle > 0$, since none of the pseudoholomorphic maps $u_k^\ell$ are constant. 

It follows that, for any $k$,
\begin{align*}
\langle \Omega, A \rangle &\geq \sum_{\ell = 1}^{N_k} m_k^\ell \langle \Omega, A_k^\ell \rangle\\
&\geq N_k\kappa
\end{align*}
so we find
$$N_k \leq \kappa^{-1}\langle \Omega, A \rangle.$$

We conclude that the sequence $N_{k}$ is uniformly bounded. Another consequence of the fact that $\langle \Omega, A_k^\ell \rangle \geq \kappa$ is that $m_k^\ell$ must be bounded above by $\kappa^{-1}\langle \Omega, A \rangle$ as well for any $k$ and $\ell$. 

We pass to a subsequence so that $N_k$ is equal to some positive integer $N$ for every $k$. We pass to a further subsequence so that for every $k$ and every $\ell$ between $1$ and $N$, inclusive, the integers $m_k^\ell$ are all equal to an integer $m^\ell \geq 1$. 

Finally, the homology classes $A_k^\ell$ come from a set of finitely many distinct classes. We can therefore pass to another subsequence so that for every $k$ and every $\ell$ between $1$ and $N$, inclusive, the classes $A_k^\ell$ are all equal to a class $A^\ell$. 

It follows that for every $k$ and every $\ell$ between $1$ and $N$, the genera $G_k^\ell$ are all equal to an integer $G^\ell \geq 0$ as well. 

We conclude that after passing to another subsequence, we can assume that for every $\ell$ between $1$ and $N$, the sequences of pseudoholomorphic curves $\mathbf{u}_k^\ell$ all converge in the Gromov sense to a stable, connected $J$-holomorphic curve $\mathbf{u}^\ell$. 

The set $\mathbf{u}^\ell$ form the desired set of $J$-holomorphic curves and the integers $m^\ell$ for the desired set of multiplicities in the statement of the proposition. 
\end{proof}

We now apply Proposition \ref{prop:nonzeroGr} to the situation of Theorem \ref{thm:swExample}. 

\begin{prop}
\label{prop:nonzeroGr2} Let $(W, \Omega)$ be a symplectic $4$-manifold and $M = H^{-1}(0)$ a closed, regular energy level such that the assumptions of Theorem \ref{thm:swExample} are satisfied. Then there is an open neighborhood $U$ of $M$ in $W$ and a constant $\rho > 0$ depending only on $W$, $\Omega$, and the class $c_1(W)$ such that for any tame almost-complex structure $J$ on $W$, there is a stable, connected $J$-holomorphic curve
$$\mathbf{u} = (u, C, j, W, J, D, \emptyset)$$
satisfying the following properties:
\begin{itemize}
    \item The image $u(C)$ intersects both components of $W \setminus M$. 
    \item The arithmetic genus of $(C, D)$ is bounded above by $\text{PD}(c_1(W)) \cdot \text{PD}(c_1(W)) + b_2^-(W) + 1$. 
    \item The integral of $u^*\Omega$ over $C$ is bounded above by $\rho$. 
\end{itemize}
\end{prop}

\begin{proof}
In \cite{Taubes94}, it is shown that the Seiberg-Witten invariant of $W$ corresponding to $-\text{PD}(c_1(W))$ does not vanish if $c_1(W) \neq 0$. 

In \cite{TaubesSWGR}, it is shown that the Seiberg-Witten and Gromov invariants coincide, so $\text{Gr}(-\text{PD}(c_1(W))) \neq 0$. Observe that $I(-\text{PD}(c_1(W))) = 0$. 

Fix any tame almost-complex structure $J$. Write $A = -\text{PD}(c_1(W))$. 

It follows from Proposition \ref{prop:nonzeroGr} that there is an integer $N \geq 1$ and, for every integer $k$ between $1$ and $N$ inclusive, a stable, connected $J$-holomorphic curve
$$\mathbf{u}_k = (u_k, C_k, j_k, W, J, D_k, \emptyset)$$
and an integer $m_k \geq 1$ such that the following conditions are satisfied:
\begin{itemize}
    \item The sum $\sum_{k=1}^N m_k (u_k)_*[C_k]$ is equal to $A$.
    \item For every $k$, the arithmetic genus of $(C_k, D_k)$ is bounded above by $A \cdot A + b_2^-(W) + 1$. 
    \item For every $k$, the self-intersection of $(u_k)_*[C_k]$ is greater than or equal to $-1$, with equality if and only if $(C_k, D_k)$ has arithmetic genus zero. 
    \item For every $k_1 \neq k_2$, the algebraic intersection number of $(u_{k_1})_*[C_{k_1}]$ and $(u_{k_2})_*[C_{k_2}]$ is zero. 
\end{itemize}

Since $J$ is tamed by $\Omega$, the integral of $u_k^*\Omega$ over $C_k$ for any $k$ is nonnegative. Moreover, by the first bullet, we have that
$$\sum_{k=1}^N m_k \int_{C_k} u_k^*\Omega = \langle \Omega, A \rangle.$$

These two observations together imply that, for any $k$,
$$\int_{C_k}u_k^*\Omega \leq \langle \Omega, A \rangle.$$

It suffices to show from the above that there is one $k$ such that $u_k(C_k)$ intersects both components of $W \setminus M$. This implies the proposition by taking $\mathbf{u}$ to be $\mathbf{u}_k$ and $\rho = \langle \Omega, A \rangle$. 

By assumption, the smooth manifold
$$W_- = H^{-1}((-\infty, 0]) \subset W$$
bounding $M$ is such that the restriction of $\Omega$ to $W_-$ is exact and $c_1(W_-) \neq 0$. 

It follows from the fact that $\Omega$ restricts to an exact two-form on $W_-$ that for every $k$, the image $u_k(C_k)$ does not lie entirely within $W_-$ for any $k$.

Now suppose for the sake of contradiction that $u_k(C_k)$ is disjoint from $W_-$ for every $k$. Since the sum
$$\sum_{k=1}^N m_k (u_k)_*[C_k]$$
is Poincar\'e dual to $-c_1(W)$, this would imply that the restriction of $c_1(W)$ to $W_-$, which is equal to $c_1(W_-)$, is equal to zero. Therefore, we arrive at a contradiction and there is some $k$ such that $u_k(C_k)$ intersects both components of $W \setminus M$. 
\end{proof}

\section{Analytic preliminaries}
\label{sec:analysis}

In this section, we provide the analytical background for the proof of our main theorems. 

There are three new major tools to introduce. The first is Fish-Hofer's exhaustive Gromov compactness theorem \cite{FishHoferExhaustive}, a version of the standard Gromov compactness theorem for pseudoholomorphic curves mapping into non-compact almost-complex manifolds. 

The second tool is Fish-Hofer's ``exponential area bound'' theorems from \cite{FishHoferFeral}. These serve to estimate the local area of pseudoholomorphic curves. They are essential in situations where we do not have a priori local area bounds on pseudoholomorphic curves, which is the case in our setting because we are performing a neck stretching operation around an arbitrary hypersurface $M \subset W$ rather than a contact-type or stable Hamiltonian hypersurface. 

The third tool is Fish-Hofer's ``$\omega$-energy quantization'' theorem from \cite{FishHoferFeral}. This is a variant of similar theorems in the pseudoholomorphic curve literature. It gives a uniform lower bound, depending only on ambient geometry, on the $\omega$-energy of a pseudoholomorphic curve in an adapted cylinder $(\mathbb{R} \times M, J, g)$ satisfying certain conditions on its image. 

We also write down the monotonicity theorem for pseudoholomorphic curves in almost-complex manifolds proved by Fish in \cite{FishCurves}. 

\subsection{Exhaustive Gromov compactness} \label{subsec:exhaustiveCompactness}

In this subsection, we present the exhaustive Gromov compactness theorem of \cite{FishHoferExhaustive}, along with several related definitions and results. Everything in this subsection is from either \cite{FishTarget}, \cite{FishHoferExhaustive}, or \cite{FishHoferFeral}. 

Fix an almost-Hermitian manifold $(W, g, J)$. Recall that this means $g$ is a Riemannian metric on the almost-complex manifold $(W,J)$ for which
$$g(J-, J-) = g(-,-).$$

Fix some stable, connected $J$-holomorphic curve 
$$\mathbf{u} = (u, C, j, W, J, D, \mu).$$ 

\begin{defn}
We say $\mathbf{u}$ is \textbf{boundary-immersed} if $u|_{\partial C}$ is an immersion, and \textbf{generally-immersed} if the restriction of $u$ to any component of $C$ is non-constant. Moreover, we say $\mathbf{u}$ is \textbf{compact} if the domain $C$ is compact. 
\end{defn}

Next, we define the notion of the \emph{oriented blow-up} of a Riemann surface, a \emph{decoration} on a Riemann surface, as well as the corresponding \emph{normalization}.

\begin{defn} \label{defn:orientedBlowup}
Let $C$ be a Riemann surface with a set of nodal points 
$$D = \{(\bar z_1, \underline z_1), ..., (\bar z_n, \underline z_n)\}$$. 
\begin{itemize}
    \item The \textbf{oriented blow-up} $C^D$ as in \cite{BEHWZ03} is $C \setminus D$ compactified with boundary circles $\bar \Gamma_i = (T_{\bar z_i} C \setminus \{0\}) / \mathbb{R}_+^*$ and $\underline \Gamma_i = (T_{\underline z_i} C \setminus \{0\}) / \mathbb{R}_+^*$ corresponding to the nodal points $\bar z_i$, $\underline z_i$ for every $i$. 
    \item A \textbf{decoration} $r$ on $C$ is a set of orientation-reversing orthogonal homeomorphisms $r_i: \bar \Gamma_i \to \underline \Gamma_i$ for every $i$. 
    \item The \textbf{normalization} $C^{D,r}$ is defined to be $C^D / (p \sim r_i(p))$. Write $\Gamma_i$ for the circle in $C^{D,r}$ obtained by identifying $\bar\Gamma_i$ and $\underline\Gamma_i$ via $r_i$, and $\Gamma = \cup_{i=1}^n \Gamma_i$. 
\end{itemize}
\end{defn}

To clarify, in the above an orientation-reversing orthogonal homeomorphism $r$ is defined to be a homeomorphism $r$ between two circles such that $r(e^{i\theta}z) = e^{-i\theta}r(z).$

We need a notion of area for our curves. In the case where $g$ is a metric associated to a symplectic form $\Omega$ for which $J$ is compatible, then this would merely be the integral of $\Omega$ over the curve. The definition we give below is analogous.

\begin{defn}
Define the two-form
$$\Theta(-,-) = g(J-,-).$$

If $\mathbf{u} = (u, C, j, W, J, D, \mu)$ is a stable, generally-immersed pseudoholomorphic curve then we set
$$\text{Area}_{u^*g}(C) = \int_C u^*\Theta.$$

If $\mathbf{u} = (u, C, j, W, J, D, \mu)$ is any stable pseudoholomorphic curve, define $C_{\text{nc}} \subseteq C$ to be the (possibly empty) union of all of the components of $C$ on which $u$ is not constant. Then we set
$$\text{Area}_{u^*g}(C) = \text{Area}_{u^*g}(C_{\text{nc}}).$$
\end{defn}

Now we define Gromov convergence, target-local Gromov compactness, and exhaustive Gromov-compactness. 

First, the following is a consequence of the uniformization theorem.

\begin{defn}
Suppose $(C,j)$ is a compact Riemann surface and there is some set of points $\Lambda \subset C$ such that
$$\chi(C \setminus \Lambda) < 0.$$

Write $h^{j,\Lambda}$ for the unique complete hyperbolic metric of constant curvature $-1$ on $C \setminus \Lambda$. 
\end{defn}

Now we can define Gromov convergence. The definition below is from \cite[Definition $2.35$]{FishHoferFeral}. 

\begin{defn}[Gromov convergence] \label{defn:gromovConvergence}  Let $J_k$ be a sequence of almost-complex structures on $W$, and 
$$\mathbf{u}_k = (u_k, C_k, j_k, W, J_k, D_k, \mu_k)$$ be a sequence of stable, connected, compact, boundary-immersed $J_k$-holomorphic curves. We say $u_k$ \textbf{converges in the Gromov sense} to a stable, connected, compact, boundary-immersed $J$-holomorphic curve 
$$\mathbf{u} =  (u, C, j, W, J, D, \mu)$$
provided the following holds for sufficiently large $k$:
\begin{itemize}
    \item $J_k \to J$ in $C^{\infty}(W)$.
    \item There exist marked points
    $$\mu_k' \subset C_k \setminus (\mu_k \cup D_k)$$
    and
    $$\mu' \subset C \setminus (\mu \cup D)$$
    such that $\# \mu_k' = \#\mu'$,
    for each connected component $\widetilde{C}_k \subset C_k$, 
    $$\chi(\widetilde{C}_k \setminus \Lambda_k) < 0,$$
    and for each connected component $\widetilde{C} \subset C$, 
    $$\chi(\widetilde{C} \setminus \Lambda) < 0,$$
    where $\Lambda_k = \mu_k \cup \mu'_k \cup D_k$ and $\Lambda = \mu \cup \mu' \cup D$. 
    \item There exists a decoration $r$ for $C$, decorations $r_k$ for $C_k$ and diffeomorphisms $\phi_k: C^{D, r} \to C_k^{D_k, r_k}$ of the normalizations such that the following hold:
    \begin{itemize}
        \item $\phi_k(\mu) = \mu_k.$
        \item $\phi_k(\mu') = \mu'_k.$
        \item For any special circle $\Gamma_i \subset C^{D, r}$, $\phi_k(\Gamma_i)$ is a $h^{j_k,\Lambda_k}$-geodesic in $C_k \setminus \Lambda_k$. 
    \end{itemize}
    \item $\phi_k^*h^{j,\Lambda_k}$ converges to $h^{j,\Lambda}$ in $C^{\infty}_{loc}(C^{D,r} \setminus (\mu \cup \mu' \cup \Gamma)$.
    \item $\phi_k^*u_k$ converges to $u$ in $C^0(C^{D,r})$ and in $C^{\infty}_{loc}(C^{D, r}\setminus \Gamma)$.
    \item Any connected component of $\partial C$ has $\phi_k^*h^{j_k,\Lambda_k}$-length uniformly bounded away from $0$ and $\infty$. 
\end{itemize}
\end{defn}

Now we can state the \emph{target-local Gromov compactness theorem}, which was proved in \cite{FishTarget}. 

\begin{rem}
The target-local Gromov compactness theorem, and many other constructions in this paper, require restricting a pseudoholomorphic curve
$$\mathbf{u} = (u, C, j, W, J, D, \mu)$$
to a sub-surface $\widetilde{C} \subset C$. The natural definition of such a restriction is a pseudoholomorphic curve of the form
$$\widetilde{\mathbf{u}} = (u, \widetilde{C}, j, W, J, D \cap \widetilde{C}, \mu \cap \widetilde{C}).$$

However, it may be the case that $\widetilde{C}$ contains one but not the other of a pair of nodal points in $D$, in which case the above is not well-defined. We fix the convention that $D \cap \widetilde{C}$ is the set of pairs of nodal points such that \emph{both} nodal points lie in $\widetilde{C}$. 

We also remark that $\widetilde{\mathbf{u}}$ is well-defined only if $D$ and $\mu$ do not have any points lying on the boundary of $\widetilde{C}$. We will always assume that this is the case in any restriction that we define. 
\end{rem}

To understand the statement of the target-local Gromov compactness theorem, we briefly recall the standard Gromov compactness theorem, originally formulated in \cite{Gromov85}. Let 
$$\mathbf{u}_k = (u_k, C_k, j_k, W, J_k, D_k, \mu_k)$$
be a sequence of stable, connected $J_k$-holomorphic curves with \emph{closed} domains $C_k$ mapping into a \emph{closed} almost-complex manifold $W$, with the sequence $J_k$ converging in $C^\infty$ to some almost-complex structure $J$ on $W$. Then if the curves have uniformly bounded area, genus, and number of marked/nodal points, a subsequence converges to a stable $J$-holomorphic curve as in Definition \ref{defn:gromovConvergence}.  

 The target-local Gromov compactness theorem generalizes this to the case where the curves $\mathbf{u}_k$ have \emph{compact} domain, possibly with boundary, mapping into a \emph{compact} manifold $W$, which may also have boundary. 
 
 In this case, even if the curves $\mathbf{u}_k$ have uniformly bounded area, genus, and number of marked/nodal points, in the absence of a nice boundary condition they could have wild behavior near the boundary. The target-local Gromov compactness theorem can be thought of as saying that we can still recover a Gromov-type limit \emph{away} from the boundary, given uniformly bounded area, genus, and number of marked/nodal points.

\begin{thm}[Target-local Gromov compactness, \cite{FishTarget} and \cite{FishHoferExhaustive}] \label{thm:targetLocal} 
 Let $(W, g, J)$ be an almost-Hermitian manifold and $(g_k, J_k)$ a sequence of almost-Hermitian structures on $W$ converging in $C^{\infty}$ to $(g, J)$. Let $\mathcal{K}_1, \mathcal{K}_2\subset \text{Int}(W)$ be compact regions (i.e. smooth, compact codimension zero submanifolds that possibly have boundary) such that $\mathcal{K}_1 \subset \text{Int}(\mathcal{K}_2).$

Let $$\mathbf{u}_k = (u_k, C_k, j_k, W, J_k, D_k, \mu_k)$$ be a sequence of stable, connected, compact $J_k$-holomorphic curves such that $u_k(\partial C_k) \cap \mathcal{K}_2 = \emptyset$. Suppose there exists some uniform constant $\kappa > 0$ such that for every $k$,
\begin{itemize}
    \item  $\text{Area}_{u_k^*g_k}(C_k) \leq \kappa$,
    \item $\text{Genus}_{\text{arith}}(C_k, D_k) \leq \kappa,$
    \item $\# (\mu_k \cup D_k) \leq \kappa.$
\end{itemize}

Then there exist compact surfaces with boundary $\widetilde{C}_k \subset C_k$ satisfying $u_k(C_k \setminus \widetilde{C}_k) \subset W \setminus \mathcal{K}_1$ such that the restrictions 
$$\widetilde{\mathbf{u}}_k = (u_k, \widetilde{C}_k, j_k, W, J_k, D_k \cap \widetilde{C}_k, \mu_k \cap \widetilde{C}_k)$$
converge in the Gromov sense to a stable, connected, compact boundary-immersed $J$-holomorphic curve.
\end{thm}

Next, we define exhaustive Gromov convergence. This notion and the following compactness theorem were introduced and proven, respectively in \cite{FishHoferExhaustive}. 

One way of thinking about the exhaustive Gromov compactness is as an ``iterated'' version of target-local Gromov compactness. Consider a sequence of $J$-holomorphic curves 
$$\mathbf{u}_k = (u_k, C_k, j_k, W, J_k, D_k, \mu_k)$$ 
where now the domains of the curves $C_k$ are still compact, but the target manifold $W$ is not required to be compact. If the images of the curves are contained within a compact subset of $W$, then it is clear that, as long as they have uniform bounds on their area, genus, and number of marked/nodal points, we can extract a limit by use of the target-local Gromov compactness theorem.

If not, then we can still perform the following procedure. Exhaust $W$ by a sequence of compact, codimension zero submanifolds $\{W^\ell\}_{\ell \geq 1}$. Then for any $\ell$, fix
$$C_k^\ell = u_k^{-1}(W^\ell)$$
and 
$$\mathbf{u}_k^\ell = (u_k, C_k^\ell, j_k, W^\ell, J_k, D_k \cap C_k^\ell, \mu_k \cap C_k^\ell)$$ to be the restriction of $\mathbf{u}_k$ to $C_k^\ell$. 

Then the sequence $\mathbf{u}_k^\ell$
is a sequence of $J$-holomorphic curves mapping into the compact manifold $W^\ell$. Suppose that this sequence of curves has a uniform bound (in $k$, not necessarily in $\ell$) on their area, genus, and number of marked/nodal points. Therefore, we can apply target-local Gromov compactness to pass to a subsequence in the index $k$ and obtain a Gromov limit curve 
$$\mathbf{u}^\ell = (u^\ell, C^\ell, j^\ell, W^\ell, J, D^\ell, \mu^\ell)$$.

It follows by a diagonal trick that we can pass to a subsequence in the index $k$ such that, for any $\ell \geq 1$, the restricted curves $\mathbf{u}_k^\ell$ converge to the curve $\mathbf{u}^\ell$ in $W^\ell$.

Recall that target-local Gromov compactness requires a uniform bound on area, genus, and number of marked/nodal points. The corresponding requirements for the exhaustive Gromov compactness theorem were hinted at above. To use the exhaustive Gromov compactness theorem, we require a uniform bound in $k$ on area, genus, and number of marked/nodal points for the pseudoholomorphic curves $\mathbf{u}_k^\ell$ in $W^\ell$. However, this uniform bound may depend on $\ell$ and increase as $\ell$ increases. 

Of course, this description skims over several technical details in the proof of this theorem. We state the complete definitions and results below.

\begin{defn} 
Let $(W, J, g)$ be a possibly non-compact almost-Hermitian manifold. An \textbf{exhausting sequence} of $W$ is a sequence of almost-Hermitian manifolds $(W_k, J_k, g_k)$ such that 
\begin{itemize}
    \item $W_k$ embeds smoothly as an proper open subset of $W_{k+1}$, and the closure of $W_k$ in $W_{k+1}$ is a compact submanifold with boundary.
    \item $W = \cup_k W_k$ with respect to the embeddings given above.
    \item $(J_k, g_k)$ converge to $(J, g)$ in $C^{\infty}_{loc}(W)$. 
\end{itemize}
\end{defn}

\begin{defn}[Exhaustive Gromov convergence] Let $(W, g, J)$ be an almost-Hermitian manifold equipped with an exhausting sequence $(W_k, g_k, J_k)$. 

Let 
$$\mathbf{u}_k = (u_k, C_k, j_k, W_k, J_k, D_k, \mu_k)$$
be a sequence of stable, connected, compact, boundary-immersed $J_k$-holomorphic curves. 

We say $\mathbf{u}_k$ \textbf{converges in the exhaustive Gromov sense} to a stable, connected, compact boundary-immersed $J$-holomorphic curve 
$$\mathbf{u} = (u, C, j, W, J, D, \mu)$$ provided there exists a sequence $\{C^\ell\}_{\ell \in \mathbb{N}}$ of compact surfaces with boundary $C^\ell \subset C$, and a collection $\{C_k^\ell\}_{k \geq \ell}$ of compact surfaces with boundary $C_k^\ell \subset C_k$ such that the following hold:
\begin{itemize}
    \item $C^\ell \subset C^{\ell + 1} \setminus \partial C^{\ell + 1}$ for every $\ell$. 
    \item $C \setminus \partial C = \cup_{\ell} C^\ell$. 
    \item For every pair of natural numbers $k > \ell$, $C^\ell_k \subset C^{\ell+1}_k \setminus \partial C^{\ell+1}_k$. 
    \item For every pair of natural numbers $k \geq \ell$, $u_k^{-1}(W_\ell) \subset C^\ell_k$.
    \item For any fixed $\ell \in \mathbb{N}$, the restrictions 
    $$\mathbf{u}_k^\ell = (u_k, C_k^\ell, j_k, W^\ell, J_k, D_k \cap C_k^\ell, \mu_k \cap C_k^\ell)$$
    of the curves $\mathbf{u}_k$ to the sub-surfaces $C_k^\ell$ converge in the Gromov sense to the restriction 
    $$\mathbf{u}^\ell = (u^\ell, C^\ell, j^\ell, W^\ell, J, D^\ell, \mu^\ell)$$
    of the curve $\mathbf{u}$ to the sub-surface $C^\ell$. 
\end{itemize}

\end{defn}

\begin{thm}[Exhaustive Gromov compactness, \cite{FishHoferFeral}] \label{thm:exhaustiveGromov} 
Let $(W, g, J)$ be an almost-Hermitian manifold equipped with an exhausting sequence $(W_k, g_k, J_k)$. 

Let 
$$\mathbf{u}_k = (u_k, C_k, j_k, W_k, J_k, D_k, \mu_k)$$ be a sequence of stable, connected, compact, boundary-immersed $J_k$-holomorphic curves. Set $\hat C^\ell_k = u_k^{-1}(W_\ell)$ for any pair $k \geq \ell$. Suppose there is a sequence of constants $\kappa_\ell$ uniform in $k \geq \ell$ satisfying the following:
\begin{itemize}
    \item $\sup_{k \geq \ell} \text{Area}_{u_k^*g_\ell}(\hat C^\ell_k) \leq \kappa_\ell$,
    \item $\text{Genus}_{\text{arith}}(\hat C^\ell_k, \hat C^\ell_k \cap D_k) \leq \kappa_\ell,$
    \item $\#(\hat C^\ell_k \cap (\mu_k \cup D_k) \leq \kappa_\ell.$
\end{itemize}

Then a subsequence of the $\mathbf{u}_k$ converges in the exhaustive Gromov sense to a stable, connected, compact boundary-immersed $J$-holomorphic curve
$$\mathbf{u} = (u, C, j, W, J, D, \mu).$$. 
\end{thm}

\subsection{Exponential area bounds} \label{subsec:exponentialAreaBounds}

Fix a framed Hamiltonian manifold $(M, \eta = (\lambda, \omega))$ and an $\eta$-adapted cylinder $(\mathbb{R} \times M, g, J)$. 

Suppose we have a marked nodal compact boundary-immersed pseudoholomorphic curve
$$\mathbf{u} = (u, C, j, \mathbb{R} \times M, J, D, \mu).$$
 Moreover, suppose that there is a compact interval $I = [b_0, b_1] \subseteq \mathbb{R}$ such that
$$(a \circ u)(C) = I$$
and
$$(a \circ u)(\partial C) \subseteq \{b_0, b_1\}.$$

In other words, $u$ maps $C$ into $I \times M$ and maps the boundary $\partial C$ into the boundary of $I \times M$. 

Fish-Hofer's ``exponential area bound'' theorem, namely \cite[Theorem $3$]{FishHoferFeral}, provides bounds for the area
$$\text{Area}_{u^*g}(C) = \int_C u^*(da \wedge \lambda + \omega)$$
in terms of the $\lambda$-energy, $\omega$-energy, and ``height'' $b_1 - b_0$ of the curve. A virtually identical estimate (\cite[Theorem $8$]{FishHoferFeral}) also holds in the slightly more general case of pseudoholomorphic curves in ``realized Hamiltonian homotopies'', which we will describe soon after the present case of pseudoholomorphic curves in $\eta$-adapted cylinders.

Now let us make a first attempt to bound $\text{Area}_{u^*g}(C)$ and then discuss the refined bound from \cite{FishHoferFeral}. Suppose that the $\omega$-energy is finite:
$$E_\omega(u) = \int_C u^*\omega < \infty.$$ This is always going to be the case in any curves that we consider. The remaining part of the area to estimate is the integral
$$\int_C u^*(da \wedge \lambda).$$

However, by the co-area formula (see \cite[Lemma $4.13$]{FishHoferFeral} for a statement of this formula specific to this situation) and the information given above, we can write
$$\int_C u^*(da \wedge \lambda) = \int_I \big(\int_{(a \circ u)^{-1}(t)} u^*\lambda \big) dt = \int_I E_\lambda(u, t) dt.$$

It follows immediately that
$$\text{Area}_{u^*g}(C) \leq (b_1 - b_0)\sup_{t \in I} E_\lambda(u,t) + E_0.$$

Therefore, we have a rough estimate for the area of $C$ that depends only on the ``height'' $b_1 - b_0$, its $\omega$-energy, and its $\lambda$-energy. 

This can be refined by establishing more precise a priori control of $E_\lambda(u,t)$. While the $\lambda$-energy may not be bounded independently of $t$, it is shown in \cite{FishHoferFeral} that $E_\lambda(u, t)$ grows at worst \emph{exponentially} in $t$. This follows from applying Gronwall's inequality to a differential inequality of the form
$$\partial_t E_\lambda(u, t) \leq \kappa_0 E_\lambda(u,t) + \kappa_1(t)$$
where $\kappa_1(t)$ is some function that is controlled by the $\omega$-energy of $C$. 

It follows that for any pseudoholomorphic curve $\mathbf{u}$ such as the one above, as long as we have a priori control over $E_\lambda(u, t)$ for some $t \in [b_0,b_1]$, we get an area estimate that depends only on this a priori bound, the ``height'' $b_1 - b_0$ of the curve, and the $\omega$-energy. The full statement is given below, along with a precise formulation of the exponential growth of the $\lambda$-energy. Note that it is stated differently than in \cite{FishHoferFeral} to align more with our discussion above.

\begin{thm} \label{thm:fishHoferAreaBound}
\cite[Theorem $3$]{FishHoferFeral} Let $I = [b_0,b_1]$ be a compact interval in $\mathbb{R}$ and
$$\mathbf{u} = (u, C, j, I \times M, J, D, \mu)$$
a marked nodal compact boundary-immersed pseudoholomorphic curve such that
$$(a \circ u)(C) = I$$
and
$$u(\partial C) \subseteq \{b_0, b_1\} \times M.$$

Then write 
$$E_0 = \inf\{E_\lambda(u, b_0), E_\lambda(u, b_1)\}.$$

Then there is a constant $\kappa = \kappa(M, \eta, g, J) > 0$ such that
$$\sup_{t \in [b_0,b_1]} E_\lambda(u, t) \leq \kappa(e^{\kappa(b_1 - b_0)} + 1)(E_0 + E_\omega(u))$$
and
$$\text{Area}_{u^*g}(C) \leq \kappa(E_0 + E_\omega(u))(e^{\kappa (b_1 - b_0)} + 1) + E_\omega(u).$$
\end{thm}

At one point, we will need a version of this area bound for pseudoholomorphic curves in ``realized Hamiltonian homotopies''. A realized Hamiltonian homotopy, as introduced in \cite[Definition $2.9$]{FishHoferFeral} and defined below, is a more general version of an $\eta$-adapted cylinder. 

\begin{defn} \label{defn:realizedHamiltonianHomotopy}
Let $M$ be a smooth, closed manifold of dimension $2n+1$ and let $I \subseteq \mathbb{R}$ be some connected interval. Write $a: \mathbb{R} \times M \to \mathbb{R}$ for the coordinate projection. Then a \textbf{realized Hamiltonian homotopy} on $I \times M$ is the datum $\hat\eta = (\hat\lambda, \hat\omega)$ of a one-form $\hat\lambda$ and a two-form $\hat\omega$ such that the following holds:
\begin{itemize}
    \item $\hat\lambda(\partial_a) = 0$,
    \item $\hat\omega(\partial_a, -) = 0$,
    \item The restriction of $d\hat\omega$ to any slice $a^{-1}(t) = \{t\} \times M \subseteq I \times M$ is $0$,
    \item $da \wedge \hat\lambda \wedge \hat\omega^n > 0$,
    \item The one-form $\hat\lambda$ is translation-invariant,
    \item If $I$ is unbounded, the two-form $\hat\omega$ is translation-invariant near $\pm\infty$. 
\end{itemize}
\end{defn}

A realized Hamiltonian homotopy has an associated Hamiltonian vector field as well.

\begin{defn} \label{defn:homotopyHamiltonianVectorField}
Let $(I \times M, \hat\lambda, \hat\omega)$ be a realized Hamiltonian homotopy. The \textbf{Hamiltonian vector field} $\hat X$ associated to the realized Hamiltonian homotopy is the unique vector field such that:
\begin{itemize}
    \item $da(\hat X) = 0$,
    \item $\hat\lambda(\hat X) = 1$,
    \item $\hat\omega(\hat X, -) = 0$.
\end{itemize}
\end{defn}

Finally, we define the almost-Hermitian structure associated to a realized Hamiltonian homotopy.

\begin{defn} \label{defn:homotopyadaptedJ} 
Let $(I \times M, \hat\eta = (\hat\lambda, \hat\omega))$ be a realized Hamiltonian homotopy with associated Hamiltonian vector field $\hat X$. An almost-complex structure $J$ on $I \times M$ is \textbf{$\hat\eta$-adapted} (or \textbf{adapted} when the context is clear) if the following holds:
\begin{itemize}
    \item $J(\partial_a) = \hat X$,
    \item $J$ preserves the distribution $\ker(da \wedge \hat\lambda)$,
    \item If $I$ is unbounded, then $J$ is translation-invariant near $\pm\infty$,
    \item The tensor $g(-,-) = (da \wedge \hat\lambda + \hat\omega)(-,J-)$ is a Riemannian metric on $I \times M$ such that
    $$g(J-, J-) = g(-,-).$$
\end{itemize}
\end{defn}

Given a choice of $\hat\eta$-adapted almost-complex structure $J$ as in Definition \ref{defn:homotopyadaptedJ}, it follows that the pair $(J,g)$ where $g$ is the induced metric given in the definition forms an almost-Hermitian structure on $I \times M$. 

There is an analogous version of the exponential area bound in Theorem \ref{thm:fishHoferAreaBound} for realized Hamiltonian homotopies. Note that we can define analogues of the $\lambda$- and $\omega$-energies for $J$-holomorphic curves in realized Hamiltonian homotopies by integrating $\hat\lambda$ and $\hat\omega$ instead, respectively.

\begin{defn}
Let $(I \times M, \hat\eta = (\hat\lambda, \hat\omega))$ be a realized Hamiltonian homotopy, and $J$ an $\hat\eta$-adapted almost-complex structure on $I \times M$.

Let 
$$\mathbf{u} = (u, C, j, I \times M, J, D, \mu)$$
be a $J$-holomorphic curve. As usual, let $a: I \times M \to I$ be the $I$-coordinate projection, and $\mathcal{R}$ the set of regular values of $a \circ u$. 

Then the \textbf{$\hat\lambda$-energy} of $\mathbf{u}$ is defined by
$$E_{\hat\lambda}(u, t) = \int_{(a \circ u)^{-1}(t)} u^*\hat\lambda$$
for any $t \in \mathcal{R}$. 

The \textbf{$\hat\omega$-energy} of $\mathbf{u}$ is defined by
$$E_{\hat\omega}(u) = \int_{C} u^*\hat\omega.$$
\end{defn}

\begin{rem}
Consider the realized Hamiltonian homotopy given by the region 
$$[-L-3\epsilon_0, L+3\epsilon_0] \times M \subseteq W_L$$
as in Lemma \ref{lem:transitionIsHomotopy}. 

Consider some $J$-holomorphic curve 
$$\mathbf{u} = (u, C, j, [-L-3\epsilon_0, L+3\epsilon_0] \times M, J, D, \mu).$$

Then the $\hat\omega$-energy is given explicitly by the integral
$$E_{\hat\omega}(u) = \int_C u^*(\omega + d(\beta_L(t)\lambda)).$$

If the image of $u$ is contained in $[-L, L] \times M$, then we can define its $\omega$-energy $E_\omega(u)$ by integrating the two-form $u^*\omega$ over the domain $C$. Since $\hat\omega = \omega + d(\beta_L(t)\lambda)$ agrees with $\omega$ on $[-L, L] \times Y$, we see in this case that
$$E_{\hat\omega}(u) = E_\omega(u),$$
that is the $\hat\omega$-energy coincides with the ordinary $\omega$-energy. 
\end{rem}

Now we can state the exponential area bound for realized Hamiltonian homotopies. 
\begin{thm} \label{thm:fishHoferHomotopyAreaBound}
Let $I = [b_0,b_1]$ be a compact interval and $(I \times M, \hat\eta = (\hat\lambda, \hat\omega))$ be a realized Hamiltonian homotopy. Pick an $\hat\eta$-adapted $J$ and let $(J,g)$ be the induced almost-Hermitian structure. Let 
$$\mathbf{u} = (u, C, j, I \times M, J, D, \mu)$$
be a marked nodal compact boundary-immersed $J$-holomorphic curve such that
$$(a \circ u)(C) = I$$
and
$$(a \circ u)(\partial C) \subseteq \{b_0, b_1\}.$$

Then write 
$$E_0 = \inf\{E_{\hat\lambda}(u, b_0), E_{\hat\lambda}(u, b_1)\}.$$

Then there is a constant $\kappa = \kappa(M, \hat\eta, g, J) > 0$  such that
$$\sup_{t \in [b_0,b_1]} E_{\hat\lambda}(u, t) \leq \kappa(e^{\kappa(b_1 - b_0)} + 1)(E_0 + E_{\hat\omega}(u))$$
and
$$\text{Area}_{u^*g}(C) \leq \kappa(E_0 + E_{\hat\omega}(u))(e^{\kappa (b_1 - b_0)} + 1) + E_{\hat\omega}(u).$$
\end{thm}

\subsection{Monotonicity and energy quantization} 

The following statement, which is due to \cite{FishCurves}, is a version of the monotonicity theorem for pseudoholomorphic curves, which states that a proper pseudoholomorphic curve in a small ball in an almost-Hermitian manifold has a uniform lower bound on its area. 

\begin{thm} \label{thm:monotonicity}
\cite[Corollary $3.7$]{FishCurves} Let $(W, J,g)$ be a compact almost-Hermitian manifold, possibly with boundary. Let 
$$\mathbf{u} = (u, C, j, W, J, \emptyset, \mu)$$
be a connected, compact, generally-immersed pseudoholomorphic curve without nodal points such that $u(\partial C) \subset \partial W$. Then for any $\epsilon > 0$, there exists a constant $r_0 = r_0(\epsilon, W, g, J)$ such that the following holds. For any point $z \in C$ and any $$r < \inf\{r_0, \text{dist}_g(u(z), \partial W)\},$$ we have the inequality
$$\text{Area}_{u^*g}(u^{-1}(B_r(u(z)))) \geq (1 + \epsilon)^{-1}\pi r^2.$$
\end{thm}

The second statement is a slightly less general version of Theorem $4$ in \cite{FishHoferFeral}, which asserts the existence of a lower bound on the $\omega$-energy of a pseudoholomorphic curve in an adapted cylinder over a framed Hamiltonian manifold given certain restrictions on its image.

\begin{thm} \cite[Theorem $4$]{FishHoferFeral} \label{thm:omegaEnergyQuantization}
Let $(M, \eta = (\lambda, \omega))$ be a framed Hamiltonian manifold, and let $(\mathbb{R} \times M, J,g)$ be an $\eta$-adapted cylinder over $M$. Fix an integer $G \geq 0$ and real numbers $h > 0$ and $\Lambda > 0$.

Let $(C, j)$ be a compact, connected Riemann surface, possibly with boundary, with genus bounded above by $G$.
Let
$$\mathbf{u} = (u, C, j, \mathbb{R} \times M, J, \emptyset, \mu)$$
be a $J$-holomorphic curve with $u$ non-constant. 

Suppose that there exists $a_0 \in \mathbb{R}$ such that $a_0 \in (a \circ u)(C)$, $(a \circ u)(\partial C)$ does not intersect the interval $[a_0 + h, a_0 - h]$, and $(a \circ u)(C)$ lies in either $[a_0, \infty)$ or $(-\infty, a_0]$. 

Then there is a constant $\hbar = \hbar(M, \eta, g, J, G, h) > 0$ such that 
$$\int_C u^*\omega \geq \hbar.$$
\end{thm}

\begin{rem}
We repeat the remark of \cite{FishHoferFeral} that the assumed genus bound in Theorem \ref{thm:omegaEnergyQuantization} can be done away with. The proof of Theorem \ref{thm:omegaEnergyQuantization} in \cite{FishHoferFeral} relies on a compactness-contradiction argument using target-local Gromov compactness, which requires a genus bound. 

We provide a sketch of how to remove the genus bound assumption here. The main idea is to replace the usage of target-local Gromov compactness with the compactness theorem for $J$-holomorphic currents, which does not require a genus bound. The interested reader can refer to \cite[Appendix B]{prequel} for a detailed overview of the terminology used below. 

Assume for the sake of contradiction that there is a sequence $\mathbf{u}_k$ of pseudoholomorphic curves satisfying the conditions of the theorem, but with possibly unbounded genus and with 
$$\int_{C_k} u_k^*\omega \to 0.$$

We observe that Theorem \ref{thm:fishHoferAreaBound} implies that the curves $\mathbf{u}_k$ have uniformly bounded area. Moreover, our assumptions imply that the associated $J$-holomorphic currents have no boundary in $(a_0 - h, a_0 + h) \times M$. 

Restricting to $(a_0 - h, a_0 + h) \times M$ and appealing to the compactness theorem for $J$-holomorphic currents, we obtain a closed $J$-holomorphic current $T$ in $(a_0 - h, a_0 + h) \times M$ such that $T(\omega) = 0$. We note by the assumptions on the curves $\mathbf{u}_k$ that the support of $T$ lies in $[a_0, a_0 + h) \times M$ or $(a_0 - h, a_0] \times M$. 

Using the fact that $T(\omega) = 0$ and repeating the arguments of Lemmas $B.13$ and $B.14$ in \cite{prequel} shows that $T$ is not zero and has support invariant under translation in the $\mathbb{R}$-factor. 

It follows that the support of $T$ is a nonempty closed subset of $(a_0 - h, a_0 + h) \times M$ that is invariant under translation, and lies in either $[a_0, a_0 + h) \times M$ or $(a_0 - h, a_0] \times M$. This is impossible, and therefore we arrive at a contradiction. 
\end{rem}

\section{Proofs of main theorems} \label{sec:proofs}

In this section, we prove Theorems \ref{thm:mainExample} and \ref{thm:swExample}. The theorems will be proved simultaneously, using either Corollary \ref{cor:GWCrossingExistence} or Proposition \ref{prop:nonzeroGr2} to ensure existence of pseudoholomorphic curves depending on whether we are in the situation of Theorem \ref{thm:mainExample} or \ref{thm:swExample}, respectively. 

Let $(W^{2n+2}, \Omega)$ be a closed symplectic manifold, and let $M = H^{-1}(0)$ be a closed, regular energy level of a smooth Hamiltonian $H$ on $W$. Recall from Example \ref{exe:hypersurfaces} that a framed Hamiltonian structure $(\lambda, \omega)$ was defined on $M$, in the following manner.

First, set $\omega = \Omega|_M$. Then, choose a compatible almost-complex structure $J$ on $W$. Recall that this is an almost-complex structure $J$ such that
$$\Omega(J-, J-) = \Omega(-, -)$$
and
$$\Omega(-, J-) > 0.$$

Let $\lambda$ be a one-form such that the kernel of $\lambda$ is $TM \cap J(TM)$. Then $\eta = (\lambda, \omega)$ is a framed Hamiltonian structure on $M$ that we fix for the remainder of this subsection.

The choice of complex structure $J$ defines an almost-Hermitian structure $(J,g)$ on $W$ by setting
$$g(-, -) = \Omega(-, J-).$$

Suppose further that the symplectic manifold $(W^{2n+2}, \Omega)$ and the hypersurface $M$ satisfy either the assumptions of Theorem \ref{thm:mainExample} or of Theorem \ref{thm:swExample}. 

\subsection{Neck stretching} \label{subsec:neckStretching}

We begin by introducing a neck stretching procedure for $W$ along the hypersurface $M$.  

Topologically, the neck-stretching procedure cuts $W$ along $M$ and glues in the cylinder $[-L, L] \times M$, for a real number $L > 0$. Call this manifold $W_L$.

Furthermore, we will define an almost-Hermitian structure $(\bar J_L, g_L)$ on $W_L$ that restricts to the almost-Hermitian structure $(J, g)$ defined above outside of a neighborhood of the cylinder $[-L, L] \times M$. On the other hand, the restriction of $\bar J_L$ to the cylinder $[-L,L] \times M$ is $\eta$-adapted, and so $([-L, L] \times M, \bar J_L, g_L)$ is canonically identified as a region inside of an $\eta$-adapted cylinder (see Definitions \ref{defn:adaptedJ} and \ref{defn:adaptedCylinder}). 

Before we proceed further, we briefly discuss neck stretching in the case where $(M, \lambda, \omega)$ is a \emph{contact-type hypersurface}, as carried out in \cite{BEHWZ03}. 

The hypersurface $M$ is contact-type when $\lambda$ is a contact form, $\omega = d\lambda$, and the hypersurface $M$ has a transverse vector field defined nearby whose flow conformally expands the ambient symplectic form. 

The flow of this vector field defines a collar neighborhood $(-\epsilon, \epsilon) \times M$ of $M$ on which we can write 
$$\Omega = d(e^a \lambda).$$ 

It follows that any $\eta$-adapted almost-complex structure on this neighborhood is in fact \emph{compatible} with the symplectic form.

If we delete $M$ and glue in a cylinder $[-L, L] \times M$ to form the stretched manifold $W_L$, it follows that we can define an almost-Hermitian structure $(\bar J_L, g_L)$ such that first, $\bar J_L$ is \emph{compatible} with the symplectic form $\Omega_L$ which restricts to $d(e^a\lambda)$ on $[-L, L] \times M$ and second, $\bar J_L$ restricts to an $\eta$-adapted almost-complex structure on the neck $[-L, L] \times M$. 

The existence of such a specific structure is very useful for proving area/energy estimates on $\bar J_L$-holomorphic curves in $W_L$. 

The $\lambda$-energy of a $\bar J_L$-holomorphic curve, i.e. the integral of $\lambda$ along the intersection of the curve with a slice 
$$\{t\} \times M \subset [-L, L] \times M \subset W_L,$$
is, in fact, controlled by the $\omega$-energy, which can in turn be controlled independently of the neck length $L$. See \cite[Lemma $9.2$]{BEHWZ03} for a proof of this. 

Now let us return to the case where $M$ is an arbitrary hypersurface.

We are not guaranteed to have such a normal form for the symplectic form $\Omega$ near $M$ in the case where $M$ is an arbitrary hypersurface. Correspondingly, we have no control over the $\lambda$-energy of our pseudoholomorphic curves, and as we take the stretching parameter $L$ to be arbitrarily large, it may increase without bound. 

We can, however, ensure that the almost complex structure $\bar J_L$ satisfies the following two properties. First, we can fix $\bar J_L$ to be compatible with respect to the symplectic form $\Omega$ \emph{outside of a neighborhood} of the neck $[-L, L] \times M$. Second, we can fix $\bar J_L$ to be $\eta$-adapted on the neck $[-L, L] \times M$. To define $\bar J_L$ on the entire manifold, we must make several auxiliary choices to interpolate between these two structures. 

We will see that a careful choice of this interpolation allows us to maintain enough control over $\bar J$-holomorphic curves in $W_L$ to apply the exhaustive Gromov compactness theorem described in Section \ref{subsec:exhaustiveCompactness}. 

For example, we are still able to control the $\lambda$-energy of pseudoholomorphic curves \emph{sufficiently near the boundary} of the neck $[-L, L] \times M$. If $$\mathbf{u} = (u, C, j, W_L, \bar J_L, D, \mu)$$
denotes a $\bar J_L$-holomorphic curve in $W_L$ with closed domain $C$, and $t$ is a real number such that $|t - L| < 1$, we can bound the $\lambda$-energy
$$E_\lambda(\mathbf{u}, t) = \int_{u^{-1}(\{t\} \times M)} u^*\lambda$$
by a constant depending only on the homology class represented by $u(C)$ (see Lemma \ref{lem:lambdaTrimBounds}). The starting point is the following lemma, which is immediate by a Moser-type argument.

\begin{lem} \label{lem:moser}
There is a constant $\epsilon \in (0,1)$, a neighborhood $U$ of $M$ and a diffeomorphism
$$\Phi: (-\epsilon, \epsilon) \times M \to U$$
such that
$$\Phi(0, m) = m$$
for any point $m \in M$ and 
$$\Phi^*\Omega = \omega + d(a\lambda)$$
where $a$ denotes the $\mathbb{R}$-coordinate in $(-\epsilon, \epsilon) \times M$.
\end{lem}

To carry out our neck stretching, we use the form for the neighborhood given to us by Lemma \ref{lem:moser} to define the gluing of a neck into the cut-open manifold. 

A rough schematic is given in Figure \ref{fig:neckStretching} below. 

\begin{figure}[ht] \includegraphics[scale=.25]{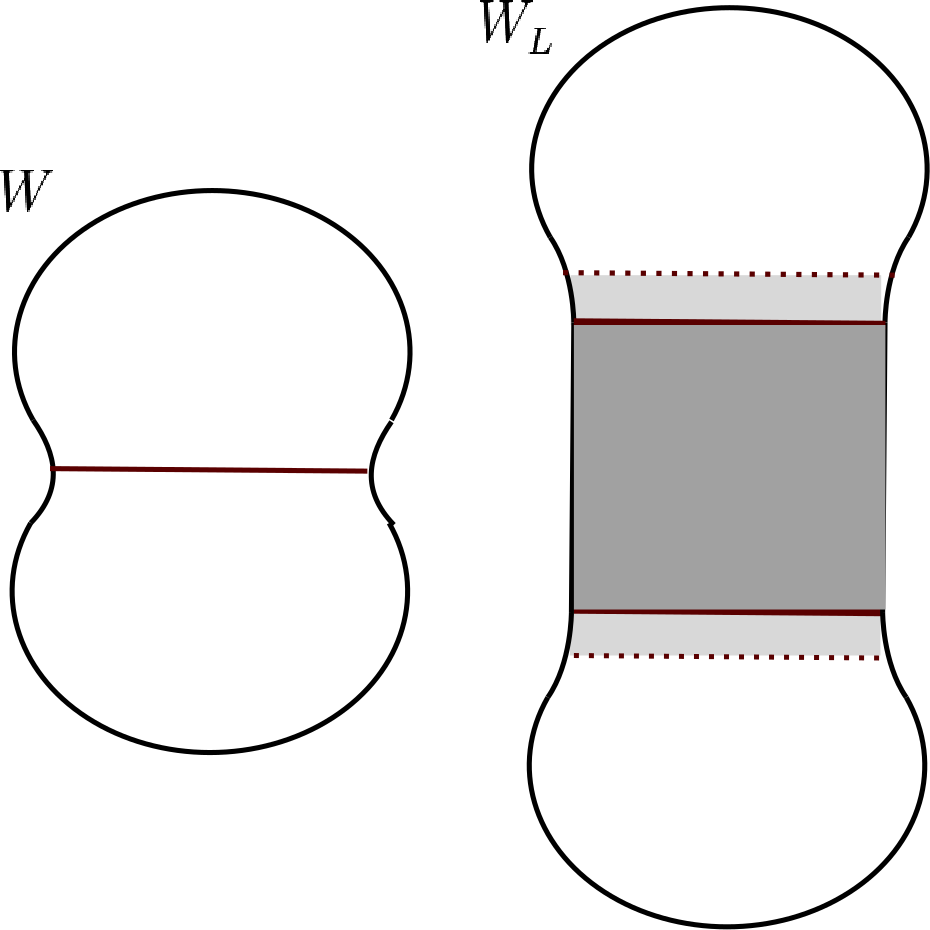} \caption{\label{fig:neckStretching} The neck stretching procedure. A ``neck'' $[-L,L] \times Y$ is glued in with a symplectization structure (dark gray). The light gray region indicates the ``transition'' region where the symplectization structure is interpolated with the almost-K\"ahler structure on $W$ induced by a choice of a compatible almost-complex structure.}\end{figure}

\subsubsection{Defining the manifolds} 

To define the manifolds $(W_L, \bar J_L, g_L)$, we begin by defining the smooth manifold $W_L$ topologically. 

Recall that the manifold $M$ is the zero set of some smooth function $H: W \to \mathbb{R}$. Fix
\begin{equation} \label{eq:wMinus} W_- = H^{-1}( (-\infty, 0]) \end{equation}
and
\begin{equation} \label{eq:wPlus} W_+ = H^{-1}( [0, \infty) ). \end{equation}

These definitions provide natural embeddings
$$M \hookrightarrow W_\pm.$$

Lemma \ref{lem:moser} provides, for any sufficiently small $\epsilon_0 > 0$, smooth embeddings 
\begin{equation} \label{eq:iotaMinus} \iota_-: [-8\epsilon_0, 0] \times M \hookrightarrow W_-\end{equation}
and
\begin{equation} \label{eq:iotaPlus} \iota_+: [0, 8\epsilon_0] \times M \hookrightarrow W_+ \end{equation}
such that the symplectic form $\Omega$ pulls back via $\iota_-$ or $\iota_+$ to the two-form
$$\omega + d(a\lambda).$$

Moreover, the composition of $\iota_-$ with the embedding
$$p \mapsto (0,p)$$
of $M$ onto 
$$\{0\} \times M \subset [-8\epsilon_0, 0] \times M$$
is the embedding
$$M \hookrightarrow W_-.$$

The analogous statement is true regarding the embedding $\iota_+$. 
It will be useful in what follows to define, for any $s \in \mathbb{R}$ the ``shift'' map 
\begin{equation} \label{eq:shiftMap} \text{Sh}_s: \mathbb{R} \times M \to \mathbb{R} \times M \end{equation}
by
$$(a, p) \mapsto (a + s, p).$$

We will at times abuse notation by regarding the shift map (\ref{eq:shiftMap}) as a map
$$\text{Sh}_s: I_1 \times M \to I_2 \times M$$
where $I_1$ and $I_2$ are subsets of $\mathbb{R}$. In any instance of this, the map will always be well-defined. That is, for any $a \in I_1$, we must have $a + s \in I_2$. 

Now pick any $L > 0$. We define the smooth manifold $W_L$ in the following manner. 

Fix a small $\epsilon_0 > 0$ and recall the embeddings $\iota_\pm$ from (\ref{eq:iotaMinus}) and (\ref{eq:iotaPlus}). 

There are also natural embeddings
$$[-8\epsilon_0, 0] \times M \hookrightarrow [-L+8\epsilon_0, L + 8\epsilon_0] \times M$$
and
$$[0, 8\epsilon_0] \times M \hookrightarrow [-L + 8\epsilon_0, L + 8\epsilon_0] \times M.$$

Define the smooth manifold $W_L$ by gluing $W_-$, $[-L + 8\epsilon_0, L + 8\epsilon_0] \times M$, and $W_+$ together. 

The gluing of $W_-$ and $[-L + 8\epsilon_0, L + 8\epsilon_0] \times M$ is done along $[-8\epsilon_0, 0] \times M$, which is realized as a subset of the former via the smooth embedding $\iota_-$, and realized as a subset of the latter via the shift map $\text{Sh}_{-L}$ from (\ref{eq:shiftMap}). 

Similarly, $W_+$ and $[-L + 8\epsilon_0, L + 8\epsilon_0] \times M$ are glued together along $[0, 8\epsilon_0] \times M$ using $\iota_+$ and $\text{Sh}_L$. 

This definition furnishes us with natural smooth embeddings
$$W_\pm \hookrightarrow W_L$$
and
$$[-L + 8\epsilon_0, L + 8\epsilon_0] \times M \hookrightarrow W_L.$$

We will use these embeddings without comment to regard $W_\pm$ and $[-L + 8\epsilon_0, L + 8\epsilon_0] \times M$ as embedded, compact codimension-zero submanifolds of $W_L$. 

We can also define two smooth manifolds $\widetilde{W}_\pm$
with cylindrical ends as follows.

Define $\widetilde{W}_-$ to be the gluing of $W_-$ and $[-8\epsilon_0, \infty) \times M$ along the region $[-8\epsilon_0, 0] \times M$. The region embeds into $W_-$ via $\iota_-$ and into the infinite half-cylinder via the tautological map. 

This definition yields natural smooth embeddings
$$W_- \hookrightarrow \widetilde{W}_-$$
and
$$[-8\epsilon_0, \infty) \times M \hookrightarrow \widetilde{W}_-.$$

Similarly, define $\widetilde{W}_+$ to be the gluing of $W_+$ and $(-\infty, 8\epsilon_0] \times M$ along the region $[0, 8\epsilon_0] \times M$. The region embeds into $W_+$ via $\iota_+$ and into the infinite half-cylinder via the tautological map. 

These definitions yields natural smooth embeddings 
$$W_+ \hookrightarrow \widetilde{W}_+$$
and
$$(-\infty, 8\epsilon_0] \times M \hookrightarrow \widetilde{W}_+.$$

We will use these embeddings without comment to regard $W_\pm$, as well as the respective half-cylinders, as embedded, compact codimension-zero submanifolds of $\widetilde{W}_\pm$. 

In what follows, we will define almost-Hermitian structures $(\bar J_\pm, g_\pm)$ on $\widetilde{W}_\pm$ as well.

\subsubsection{The almost-Hermitian structure on the neck} \label{subsubsec:neckStretchingNeck}

Choose an $\eta$-adapted cylinder $(\mathbb{R} \times M, J_{\text{Neck}}, g_{\text{Neck}})$ over $M$. 

Recall that for any $L > 0$, there is a natural embedding
$$[-L-2\epsilon_0, L+2\epsilon_0] \times M \hookrightarrow W_L.$$

For any $L > 0$, we define the almost-Hermitian structure $(\bar J_L, g_L)$ on the image of $[-L-2\epsilon_0, L+2\epsilon_0] \times M \hookrightarrow W_L$ to be the pushforward of $(J_{\text{Neck}}, g_{\text{Neck}})$ by this embedding. 

There is a natural embedding
$$[-2\epsilon_0, \infty) \times M \hookrightarrow \widetilde{W}_-.$$

Define $(\bar J_-, g_-)$ on the image of $[-2\epsilon_0, \infty) \times M$ to be the pushforward of $(J_{\text{Neck}}, g_{\text{Neck}})$ by this embedding. 

There is a natural embedding
$$(-\infty, 2\epsilon_0] \times M \hookrightarrow \widetilde{W}_+.$$

Define $(\bar J_+, g_+)$ on the image of $(-\infty, 2\epsilon_0] \times M$ to be the pushforward of $(J_{\text{Neck}}, g_{\text{Neck}})$ by this embedding. 

\subsubsection{The almost-Hermitian structure in between} \label{subsubsec:neckStretchingBetween}

We will define, for a sufficiently small real number $\epsilon_0$, two almost-Hermitian manifolds
$$(T_- = [-8\epsilon_0,0] \times M, J_{\text{Between},-}, g_{\text{Between},-})$$
and
$$(T_+ = [0,8\epsilon_0] \times M, J_{\text{Between},+}, g_{\text{Between},+}).$$

We define an almost-Hermitian manifold 
$$(T, J_{\text{Between}}, g_{\text{Between}})$$
as the disjoint union of $T_\pm$. 

The embeddings $\iota_\pm$ from (\ref{eq:iotaMinus}) and (\ref{eq:iotaPlus}) yield embeddings
$$T_\pm \hookrightarrow W_\pm$$
which define, for any $L$, an embedding
$$T \hookrightarrow W_L.$$

We will define the almost-Hermitian structure $(\bar J_L, g_L)$ on the image of $T$ to be the pushforward of $(J_{\text{Between}}, g_{\text{Between}})$ via this embedding. 

Similarly, the embeddings $\iota_\pm$ from (\ref{eq:iotaMinus}) and (\ref{eq:iotaPlus}) yield embeddings
$$T_\pm \hookrightarrow W_\pm$$
which define embeddings
$$T_\pm \hookrightarrow \widetilde{W}_\pm.$$

We will define the almost-Hermitian structures $(\bar J_\pm, g_\pm)$ on the respective images of $T_\pm$ to be the pushforwards of $(J_{\text{Between},\pm}, g_{\text{Between},\pm})$ via these embeddings.

We will carry out the construction in three steps. First, we define a two-form $\hat\omega$ on $T$. We use this to define the almost-complex structure $J_{\text{Between}, \pm}$. We then use this to define the metric $g_{\text{Between}, \pm}$.

Fix a smooth function 
$$\beta_*: [-8\epsilon_0, 8\epsilon_0] \to \mathbb{R}$$
satisfying the following properties:
\begin{itemize}
    \item $\beta_*(t) = t$ if $t \in [-8\epsilon_0, -4\epsilon_0]$ or $t \in [4\epsilon_0, 8\epsilon_0]$.
    \item $\beta_*(t) = 0$ if $t \in [-2\epsilon_0, 2\epsilon_0]$.
    \item $\beta_*'(t) > 0$ if $t \in [-4\epsilon_0, -2\epsilon_0)$ or $t \in (2\epsilon_0, 4\epsilon_0]$. 
\end{itemize}

It will also be convenient for the proof of Lemma \ref{lem:tameJ1} and other results later to assume that 
$$\beta_*''(t) < 0$$
for $t \in (-4\epsilon_0, -2\epsilon_0)$ and
$$\beta_*''(t) > 0$$
for $t \in (2\epsilon_0, 4\epsilon_0)$. 

Then define the two-form $\hat\omega$ on $T$ by
$$\hat\omega = \omega + \beta_*(a)d\lambda.$$

Now we define $J_{\text{Between},-}$. As long as we assume $\epsilon_0$ is sufficiently small, the tangent bundle of $T_-$ splits into a direct sum
$$\ker(da \wedge \lambda) \oplus \ker(\hat\omega).$$

Define a vector field $\bar X$ on $T_-$ by stipulating $da(\bar X) \equiv 0$, $\lambda(\bar X) \equiv 1$, and $\hat\omega(\bar X, -) \equiv 0$. 

Then define $J_{\text{Between},-}$ on $T_-$ by assuming it satisfies the following properties:
\begin{itemize}
    \item $J_{\text{Between},-}(\partial_a) = \bar X$ and $J_{\text{Between},-}(\bar X) = -\partial_a$,
    \item $J_{\text{Between},-}$ preserves the summands $\ker(da \wedge \lambda)$ and $\ker(\hat\omega)$,
    \item The tensor $\hat\omega(-, J_{\text{Between},-}-)$ is symmetric and positive-definite on $\ker(da \wedge \lambda)$.
    \item $J_{\text{Between},-} = J_{\text{Neck}}$ on $[-2\epsilon_0, 0] \times M$. 
\end{itemize}

An identical construction, replacing the symbol ``$-$'' with ``$+$'' defines $J_{\text{Between},+}$ on $T_+$. In particular, we have that $J_{\text{Between},+}$ coincides with $J_{\text{Neck}}$ on $[0, 2\epsilon_0] \times M$. 

To define the metric $g_{\text{Between}, \pm}$, define a smooth function $\chi: \mathbb{R} \to [0,1]$ by stipulating that 
\[
\chi(t) = \begin{cases}
1\text{ if $t < 3\epsilon_0$}, \\
0\text{ if $t > 5\epsilon_0$}.
\end{cases}
\]

It is not used anywhere, but we may as well suppose $\chi'(t) < 0$ on $[3\epsilon_0, 5\epsilon_0]$.

Then fix a smooth function
$$\theta: [-8\epsilon_0, 8\epsilon_0] \to \mathbb{R}$$
by setting
$$\theta(t) = \chi(|t|) + (1 - \chi(|t|))\beta_*'(t).$$

Define the metric $g_{\text{Between}, \pm}$ on $T_\pm$ by the formula
$$g_{\text{Between}, \pm}(-, -) = (\theta(a)(da \wedge \lambda) + \hat\omega)(-, J_{\text{Between},\pm}-).$$

We observe that $g_{\text{Between}, -}$ agrees with $g_{\text{Neck}}$ on $[-2\epsilon_0, 0] \times M$ and $g_{\text{Between}, +}$ agrees with $g_{\text{Neck}}$ on $[0, 2\epsilon_0] \times M$. 

\subsubsection{The almost-Hermitian structure on the core} \label{subsubsec:neckStretchingCore}

We will first define the almost-Hermitian structure $(\bar J_L, g_L)$ on $W_L$ on the complement of the cylinder $$[-L-6\epsilon_0,L+6\epsilon_0] \times M.$$

Choose an almost-complex structure $J_{\text{Core}}$ on $W$ that is compatible with the symplectic form $\Omega$. Recall that this defines an almost-Hermitian structure $(J_{\text{Core}}, g_{\text{Core}})$ by setting
$$g_{\text{Core}}(-, -) = \Omega(-, J_{\text{Core}}-).$$

We can choose $J_{\text{Core}}$ so that the pullback of $J_{\text{Core}}$ by the embeddings $\iota_-$ and $\iota_+$ coincide on $[-8\epsilon_0, -6\epsilon_0] \times M$ with $J_{\text{Between},-}$ and on $[6\epsilon_0, 8\epsilon_0] \times M$ with $J_{\text{Between},+}$. By definition, it follows that the pullbacks of $g_{\text{Core}}$ on these regions are $g_{\text{Between},-}$ and $g_{\text{Between},+}$ respectively. 

Define
$$\text{Core}(W)$$
to be the complement in $W$ of the union of the two regions
$$\iota_-( (-6\epsilon_0, 0] \times M)$$
and
$$\iota_+( [0, 6\epsilon_0) \times M).$$

There are as a result of this definition natural embeddings 
$$\text{Core}(W) \hookrightarrow W_L$$
for every $L > 0$. 

For any $L > 0$, we define $(\bar J_L, g_L)$ on the image of $\text{Core}(W)$ to be the pushforward of $(J_{\text{Core}}, g_{\text{Core}})$ via this embedding. 

Define 
$$\text{Core}(W_\pm) = \text{Core}(W) \cap W_\pm.$$

There are as a result of this definition natural embeddings
$$\text{Core}(W_\pm) \hookrightarrow \widetilde{W}_\pm.$$

Define the almost-Hermitian structures $(\bar J_\pm, g_\pm)$ on the image of $\text{Core}(W_\pm)$ in $\widetilde{W}_\pm$ to be the pushforward of $(J, g)$ via this embedding. 

The constructions in Sections \ref{subsubsec:neckStretchingNeck}, \ref{subsubsec:neckStretchingBetween} and \ref{subsubsec:neckStretchingCore} have produced the desired almost-Hermitian manifolds
$$(W_L, \bar J_L, g_L)$$
and
$$(\widetilde{W}_\pm, \bar J_\pm, g_\pm).$$

\subsection{Properties of the neck-stretching construction}

We now list some properties of the almost-Hermitian manifolds 
$$(W_L, \bar J_L, g_L)$$
and
$$(\widetilde{W}_\pm, \bar J_\pm, g_\pm)$$
that we will use later. 

It will be useful to define the following smooth functions. 

For any $L > 0$, define a smooth function 
$$\beta_L: [-L - 8\epsilon_0, L + 8\epsilon_0] \to \mathbb{R}$$
by
\begin{equation} \label{eq:betaL}
\beta_L(t) = \begin{cases}
\beta_*(t + L) \text{ if $t \in [-L - 8\epsilon_0, -L]$},\\
0\text{ if $t \in [-L, L]$},\\
\beta_*(t - L)\text{ if $t \in [L, L + 8\epsilon_0]$}.
\end{cases}
\end{equation}

The function $\beta_L$ is schematically graphed in Figure \ref{fig:betaFunction} below.

\begin{figure}[ht] \includegraphics[width=.5\textwidth]{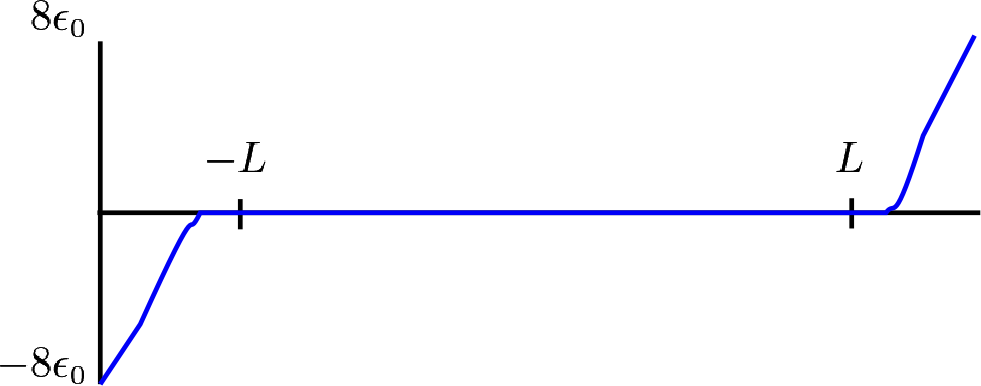} \caption{\label{fig:betaFunction} A graph of the function $\beta_L$.}
\end{figure}

Define a smooth function
$$\beta_-: [-8\epsilon_0, \infty) \to \mathbb{R}$$
by
\begin{equation}
\label{eq:betaMinus}
\beta_-(t) = \begin{cases}
\beta_*(t) \text{ if $t \in [- 8\epsilon_0, 0]$},\\
0\text{ if $t \in [0, \infty)$}.
\end{cases}
\end{equation}

Define a smooth function
$$\beta_+: (-\infty, 8\epsilon_0] \to \mathbb{R}$$
by
\begin{equation} \label{eq:betaPlus}
\beta_+(t) = \begin{cases}
\beta_*(t) \text{ if $t \in [0, 8\epsilon_0]$},\\
0\text{ if $t \in (-\infty, 0]$}.
\end{cases}
\end{equation}

We write down the following technical lemma.

\begin{lem} \label{lem:uniformOmegaHatBounds}
There is some constant $\kappa_0 \geq 1$ depending only on the almost-Hermitian structures $(J_{\text{Between},\pm}, g_{\text{Between},\pm})$ such that the following statements hold:
\begin{itemize}
    \item The two-form $d\lambda$ on 
    $$[-L-8\epsilon_0,L+8\epsilon_0] \times M \subset W_L$$
    has $g_L$-norm bounded above by $\kappa_0$. 
    \item The two-form $d\lambda$ on 
    $$[-8\epsilon_0, \infty) \times M \subset \widetilde{W}_-$$
    and the two-form $d\lambda$ on 
    $$(-\infty, 8\epsilon_0] \times M \subset \widetilde{W}_+$$
    have respectively their $g_-$ and $g_+$ norms bounded above by $\kappa_0$. 
    \item For any vector $V \in \ker(da \wedge \lambda)$, we have
    $$(\omega + \beta_L(a)d\lambda)(V, \bar J_L(V)) \geq \kappa_0^{-1}||V||^2_{g_L}$$
    and
    $$(\omega + \beta_\pm(a)d\lambda)(V, \bar J_\pm(V)) \geq \kappa_0^{-1}||V||^2_{g_\pm}.$$
\end{itemize}
\end{lem}

\begin{proof}
The first two bullet points are immediate. The restrictions of the metrics $g_L$ to slices of the form $\{t\} \times M$ in $[-L-8\epsilon_0,L+8\epsilon_0] \times M$ vary in a fixed compact family of Riemannian metrics on $M$. 

For the third bullet point, we observe that by construction $J_{\text{Between},\pm}$ is such that
$$(\omega + \beta_*(a)d\lambda)(-, J_{\text{Between},\pm}-)$$
is symmetric and positive-definite. 

Moreover, by construction
$J_{\text{Neck}}$ is such that
$$\omega(-, J_{\text{Neck}}-)$$
is symmetric and positive-definite on $\ker(da \wedge \lambda)$, and also translation-invariant, so the quantity
$$\omega(V, J_{\text{Neck}}(V))/||V||^2_{g_{\text{Neck}}}$$
is uniformly bounded away from zero over all nonzero $V$ in $\ker(da \wedge \lambda)$. Furthermore, note that $\omega + \beta_L(a)d\lambda = \omega$ in $[-L-2\epsilon_0, L+2\epsilon_0] \times M$. 

Putting these two facts together yields the third bullet point. 
\end{proof}

The next proposition is especially important and makes explicit use of Lemma \ref{lem:uniformOmegaHatBounds} above. It shows that, for large $L$, the almost-complex structures $\bar J_L$ are tamed by a symplectic structure on $W_L$, equal to the pullback of the symplectic form $\Omega$ via a diffeomorphism from $W_L$ to $W$. This is necessary for us to exploit pseudoholomorphic curve counting invariants to construct pseudoholomorphic curves in the manifolds $W_L$, since these invariants always require the ambient almost-complex structure to be tame. 

\begin{prop} \label{prop:tameJ}
There is a constant $L(\epsilon_0, \kappa_0, \beta_*) \geq 1$  depending only on $\epsilon_0$, the constant $\kappa_0$ from Lemma \ref{lem:uniformOmegaHatBounds}, and the function $\beta_*$ such that, for any $L \geq L(\epsilon_0, \kappa_0, \beta_*)$, there is a diffeomorphism 
$$\Phi_L: W_L \to W$$
such that $\bar J_L$ is tamed by the symplectic form $(\Phi_L)^*\Omega$ for any sufficiently large $L$.
\end{prop}

The proof of Proposition \ref{prop:tameJ} requires the careful construction of a family of smooth, increasing functions 
$$h_L: [-L - 8\epsilon_0, L+8\epsilon_0] \to [-8\epsilon_0, 8\epsilon_0]$$
for sufficiently large $L$. 

\begin{lem}
\label{lem:tameJ1} Let $\kappa_* \geq 1$ be any constant. There is a constant $L(\epsilon_0, \kappa_*, \beta_*) \geq 1$ depending only on $\epsilon_0$, $\kappa_*$ and the function $\beta_*$ so that, for any $L \geq L(\epsilon_0, \kappa_*, \beta_*)$, we can construct a smooth function 
$$h_L: [-L - 8\epsilon_0, L+8\epsilon_0] \to [-8\epsilon_0, 8\epsilon_0]$$
that has the following properties:
\begin{itemize}
    \item $h_L'(t) > 0$ everywhere,
    \item $h_L(t) = \beta_L(t)$ in a neighborhood of $[-L-8\epsilon_0,-L-11\epsilon_0/4]$ and of $[L+11\epsilon_0/4,L+8\epsilon_0]$,
    \item The function $h_L$ satisfies the bound $$|h_L - \beta_L|_{C^0} \leq 100\kappa_*^{-1} L^{-1}$$
    \item The derivative of $h_L$ satisfies the bound $$h_L'(t) \geq \frac{1}{2}\kappa_*^{-1}L^{-2}.$$
\end{itemize}
\end{lem}

A schematic example of $h_L$ is graphed in Figure \ref{fig:hFunction} below. The picture is essentially the same as in Figure \ref{fig:betaFunction}, but the flat part in the middle is replaced by a line of slope comparable to $L^{-2}$. 

\begin{figure}[ht]
    \centering
    \includegraphics[width=.5\textwidth]{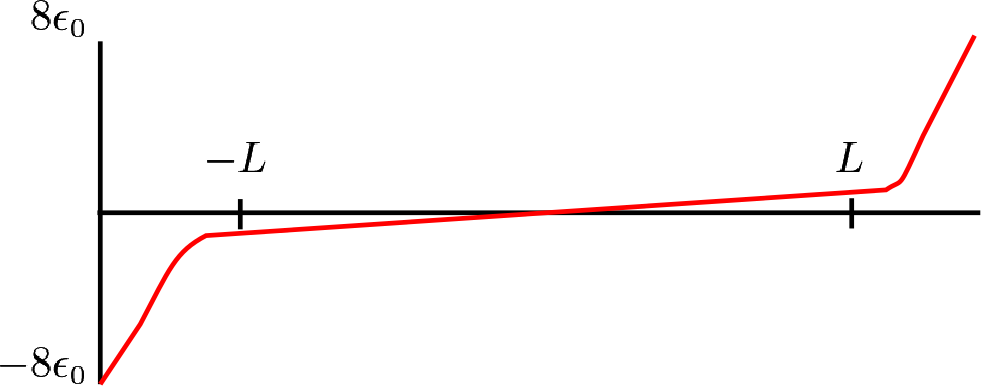}
    \caption{A graph of the function $h_L$.}
    \label{fig:hFunction}
\end{figure}

\begin{proof}[Proof of Lemma \ref{lem:tameJ1}]
We now give a construction of the function $h_L$. We begin with some initial observations regarding the function $\beta_L$. 

As long as $L$ is larger than a constant depending on $\epsilon_0$, $\kappa_*$, and the function $\beta_*$, it follows that $\beta_L'(t)$ is strictly greater than $100\epsilon_0^{-1}L^{-1}$ for $t$ in $[-L-8\epsilon_0, -L-9\epsilon_0/4]$ and in $[L+9\epsilon_0/4, L+8\epsilon_0]$. 

Recall by construction that $\beta_L'$ is decreasing from $1$ to $0$ on $[-L-4\epsilon_0, -L-2\epsilon_0]$, increasing from $0$ to $1$ on $[L+2\epsilon_0, L+4\epsilon_0]$ and zero on $[-L-2\epsilon_0, L+2\epsilon_0]$. Moreover, we have that $\beta_L''(t) = \beta_*''(t + L)$ on $[-L-4\epsilon_0, -L-2\epsilon_0]$ and $\beta_*''(t - L)$ on $[L+2\epsilon_0, L+4\epsilon_0]$. 

It follows that, as long as $L$ is larger than a constant depending only on $\kappa_*$ and the function $\beta_*$, there are unique constants $L_- \in (-L-9\epsilon_0/4, -L-2\epsilon_0)$ and $L_+ \in (L+2\epsilon_0, L + 9\epsilon_0/4)$ such that
$$\beta_L'(L_-) = \beta_L'(L_+) = \kappa_*^{-1}L^{-2}.$$

Now we define a positive smooth function 
$$f_L(t): [-L-8\epsilon_0, L+8\epsilon_0] \to \mathbb{R}$$
which satisfies the following properties:
\begin{itemize}
    \item The function $f_L(t) - \beta_L'(t)$ on $[-L-8\epsilon_0, -L-9\epsilon_0/4]$ is compactly supported in the interior of $[-L-11\epsilon_0/4, -L-5\epsilon_0/2]$ and bounded in absolute value by $10\epsilon_0^{-1}\kappa_*^{-1}L^{-1}$. 
    \item The function $f_L(t) - \beta_L'(t)$ on $[L+9\epsilon_0/4, L+8\epsilon_0]$ is compactly supported in the interior of $[L+5\epsilon_0/2, L+11\epsilon_0/4]$ and bounded in absolute value by $10\epsilon_0^{-1}\kappa_*^{-1}L^{-1}$. 
    \item $f_L(t) = \kappa_*^{-1}L^{-2}$ for $t$ in $[-L-2\epsilon_0, L+2\epsilon_0]$. 
    \item $f_L(t)$ is bounded away from $\max(\beta_L'(t), \kappa_*^{-1}L^{-2})$ on $[-L-9\epsilon_0/4, L+9\epsilon_0/4]$ by at most $L^{-4}$. 
    \item The integral of $f_L(t)$ from $-L-3\epsilon_0$ to $-L-2\epsilon_0$ is equal to 
    $$-\kappa_*^{-1}L^{-2}(L + 2\epsilon_0) - \beta_L(-3\epsilon_0)$$
    and the integral of $f_L(t)$ from $L+2\epsilon_0$ to $L+3\epsilon_0$ is equal to
    $$\beta_L(3\epsilon_0) - \kappa_*L^{-2}(L+2\epsilon_0).$$
\end{itemize}

The function $f_L$ can be constructed as follows. Start with the continuous function $\max(\beta_L'(t), \kappa_*^{-1}L^{-2})$. By our observations above, this is equal to $\beta_L'(t)$ outside of $[-L-9\epsilon_0/4, L+9\epsilon_0/4]$. 

Subtract off a nonnegative function compactly supported in the interior of $[-L-11\epsilon_0/4, -L-5\epsilon_0/2]$ with integral slightly larger than $\kappa_*^{-1}L^{-2}(L + 2\epsilon_0)$. Such a function can be chosen to be bounded above by $10\epsilon_0^{-1}\kappa_*^{-1}$.

Subtract off an nonnegative function compactly supported in the interior of $[L+5\epsilon_0/2, L+11\epsilon_0/4]$ that also satisfies the above properties regarding its integral and its upper bound. 

The resulting function is smooth except at the points $L_\pm$ in $(-L-9\epsilon_0/4, L+9\epsilon_0/4)$. Choose $f_L$ to be a smooth function which agrees with this function outside of a small neighborhood of $L_\pm$ and is at most $L^{-4}$ away from this function inside these neighborhoods. 

Now set
$$h_L(t) = \int_{-L-8\epsilon_0}^t f_L(s) ds.$$

The first bullet point of the lemma follows from the fact that $f_L$ is positive. 

To verify the second bullet point of the lemma, it suffices to verify that
$$\int_{-L-8\epsilon_0}^{L+3\epsilon_0} f_L(s) ds = \beta_L(L+3\epsilon_0).$$

Recall that the integral of $f_L$ from $-L-8\epsilon_0$ to $-L-2\epsilon_0$ is, by the first and last properties defining $f_L$, equal to 
$$-\kappa_*^{-1}L^{-2}(L + 2\epsilon_0) - \beta_L(-L-8\epsilon_0).$$

By the third property of $f_L$, the integral of $f_L$ from $-L-2\epsilon_0$ to $L+2\epsilon_0$ is $2\kappa_*^{-1}L^{-2}(L+2\epsilon_0)$. 

By the last property of $f_L$, the integral of $f_L(t)$ from $L+2\epsilon_0$ to $L+3\epsilon_0$ is equal to
    $$\beta_L(3\epsilon_0) - \kappa_*L^{-2}(L+2\epsilon_0).$$
    
Adding these all up yields the desired identity for the integral of $f_L$ from $-L-8\epsilon_0$ to $L+3\epsilon_0$ and therefore verifies the second bullet point of the lemma.

Next, we verify the third bullet point of the lemma. Suppose $L$ is sufficiently large so that
$$L^{-4} \leq \kappa_*^{-1}L^{-2}$$
and
$$L \geq 100\epsilon_0^{-1}.$$

For any $t \in [-L-8\epsilon_0, L+8\epsilon_0]$, 
\begin{align*}
    |h_L(t) - \beta_L(t)| &\leq \int_{-L-8\epsilon_0}^t |f_L(t) - \beta_L'(t)| \\
    &\leq \int_{-L-8\epsilon_0}^{L+8\epsilon_0} |f_L(t) - \beta_L'(t)| \\
    &\leq 20\kappa_*^{-1}L^{-1} + \int_{-L-9\epsilon_0/4}^{L+9\epsilon_0/4} |f_L(t) - \beta_L'(t)| \\
    &\leq 20\kappa_*^{-1}L^{-1} + \int_{-L-9\epsilon_0/4}^{L+9\epsilon_0/4} 10\kappa_*^{-1}L^{-2} \\
    &\leq 30\kappa_*^{-1}L^{-1}.
\end{align*}

The second inequality uses the first and second properties of $f_L$. The third inequality uses the fourth property of $f_L$. 

Finally, we verify the fourth bullet point of the lemma. Observe by construction that $h_L'(t) = f_L(t)$ everywhere. 

If $t$ lies outside the interval $[-L-9\epsilon_0/4, L+9\epsilon_0/4]$ then by the first and second properties of $f_L$, along with the fact that 
$$\beta_L'(\pm(L + 9\epsilon_0/4)) \geq 100\epsilon_0^{-1}L^{-1},$$
we find that $f_L(t)$ is bounded below by $90\epsilon_0^{-1}L^{-1}$. We can take $L$ sufficiently large so that this is at least $\frac{1}{2}\kappa^{-1}L^{-2}$. 

If $t$ lies inside the interval $[-L-9\epsilon_0/4, L+9\epsilon_0/4]$ then by construction it is bounded below by $\kappa_*^{-1}L^{-2} - L^{-4}$. We can take $L$ sufficiently large so that this is at least $\frac{1}{2}\kappa^{-1}L^{-2}$. 
\end{proof}

We use the function from Lemma \ref{lem:tameJ1} to prove Proposition \ref{prop:tameJ}. 

\begin{proof}[Proof of Proposition \ref{prop:tameJ}]
Let $\kappa_0 \geq 1$ be the constant given by Lemma \ref{lem:uniformOmegaHatBounds}. Let $\kappa_* \geq 1$ be a constant equal to $10^{10}\kappa_0^3$. Fix $\Phi_L$ to be the identity map sending
$$\text{Core}(W) \subset W_L$$
to
$$\text{Core}(W) \subset W.$$

We extend the definition of $\Phi_L$ to all of $W_L$ by defining it on the neck region as a composition of a diffeomorphism 
$$[-L - 8\epsilon_0, L + 8\epsilon_0] \times M \to [-8\epsilon_0, 8\epsilon_0] \times M$$
and the natural embedding
$$[-8\epsilon_0, 8\epsilon_0] \times M \hookrightarrow W$$
given by Lemma \ref{lem:moser}. 

This diffeomorphism is defined for $L$ sufficiently large be the map
$$(t, p) \mapsto (h_L(t),p)$$
where $h_L$ is the smooth function
$$h_L: [-L-8\epsilon_0, L+8\epsilon_0] \to [-8\epsilon_0,8\epsilon_0]$$
given to us by Lemma \ref{lem:tameJ1} with the parameter $\kappa_*$ fixed to be $10^{10}\kappa_0^3$. 

Recall that on $\text{Core}(W)$ we have that $(\Phi_L)^*\Omega$ is a symplectic form and $\bar J_L$ is compatible with $(\Phi_L)^*\Omega$.

It remains to check that $\bar J_L$ is tamed by $(\Phi_L)^*\Omega$ on $[-L-6\epsilon_0, L+6\epsilon_0] \times M$. 

Inside this region,
$$(\Phi_L)^*\Omega = \omega + h_L(a) d\lambda + h_L'(a) da \wedge \lambda.$$

By definition, since $h_L' > 0$, the two-form
$$\widetilde{\omega}_L = \omega + \beta_L(a)d\lambda + h_L'(a)da \wedge \lambda$$
is positive on any $\bar J_L$-complex line. 

Since $h_L = \beta_L$ on $[-L-8\epsilon_0,-L-3\epsilon_0]$ and $[L+3\epsilon_0,L+8\epsilon_0]$, it follows immediately that
$\bar J_L$ is tamed by $(\Phi_L)^*\Omega$ outside of $[-L-3\epsilon_0, L+3\epsilon_0] \times M$. 

We complete the proposition by verifying that $\bar J_L$ is tamed by $(\Phi_L)^*\Omega$ inside $[-L-3\epsilon_0, L+3\epsilon_0] \times M$.

The tangent bundle of $[-L-3\epsilon_0, L+3\epsilon_0] \times M$ admits a $g_L$-orthogonal splitting given by
$$\ker(\omega + \beta_L(a)d\lambda) \oplus \ker(da \wedge \lambda).$$

Observe that $\ker(\omega + \beta_L(a)d\lambda)$ is spanned by $\partial_a$ and $\bar J_L(\partial_a)$, and that these tangent vectors have unit $g_L$-length. 

Fix any tangent vector $V$ of unit $g_L$-length. It follows that it can be written uniquely as 
$$V = x V_1 + y \partial_a + z \bar{J}_L(\partial_a)$$
where $V_1$ is a tangent vector of unit length in $\ker(\omega + \beta_L(a)d\lambda)$ and $x^2 + y^2 + z^2 = 1$. 

We compute
\begin{equation}
\label{eq:tameJ1}
\begin{split}
    (\Phi_L)^*\Omega(V, \bar J_L(V)) &= (\Phi_L)^*\Omega(xV_1 + y \partial_a + z \bar{J}_L(\partial_a), x\bar J_L(V_1) + y \bar J_L(\partial_a) - z \partial_a) \\
    &= (\omega + h_L(a)d\lambda)(xV_1 + z\bar{J}_L(\partial_a), x\bar{J}_L(V_1) + y\bar{J}_L(\partial_a)) \\
    &\; + h_L'(a)(da \wedge \lambda)(y \partial_a + z \bar{J}_L(\partial_a), y \bar J_L(\partial_a) - z \partial_a) \\
    &= x^2(\omega + \beta_L(a)d\lambda)(V_1, \bar{J}_L(V_1)) \\
    &\; + x^2(h_L(a) - \beta_L(a))d\lambda(V_1, \bar{J}_L(V_1)) \\
    &\; + xy(h_L(a) - \beta_L(a))d\lambda(V_1, \bar{J}_L(\partial_a)) \\
    &\; - xz(h_L(a) - \beta_L(a))d\lambda(\bar{J}_L(V_1), \bar{J}_L(\partial_a)) \\
    &\; + h_L'(a)(y^2 + z^2).
\end{split}
\end{equation}

An application of Lemma \ref{lem:uniformOmegaHatBounds} and the Peter-Paul inequality shows that
\begin{equation}
    \label{eq:tameJ2}
    \begin{split}
        (\Phi_L)^*\Omega(V, \bar J_L(V)) &\geq x^2(\kappa_0^{-1} - \kappa_0|h_L(a) - \beta_L(a)|) \\
        &\; - x(y+z)\kappa_0|h_L(a) - \beta_L(a)| + h_L'(a)(y^2 + z^2) \\
        &\geq x^2(\kappa_0^{-1}/2 - \kappa_0|h_L(a) - \beta_L(a)|) \\
        &\; + (h_L'(a) - 10\kappa_0^3|h_L(a) - \beta_L(a)|^2)(y^2 + z^2).
    \end{split}
\end{equation}

We now appeal to the bounds guaranteed to us by Lemma \ref{lem:tameJ1}. 

First, suppose $L$ is greater than $1000\kappa_0^2$. It follows from the third bullet in Lemma \ref{lem:tameJ1} that
\begin{align*}
    \kappa_0^{-1}/2 - \kappa_0|h_L(a) - \beta_L(a)| &\geq \kappa_0^{-1}/2 - 100\kappa_0L^{-1} \\
    &\geq \kappa_0^{-1}/2 - \kappa_0^{-1}/10 \\
    &> 0.
\end{align*}

Second, the third and fourth bullets in Lemma \ref{lem:tameJ1} show that
\begin{align*}
    h_L'(a) - 10\kappa_0^3|h_L(a) - \beta_L(a)|^2 &\geq 10^{-9}\kappa_0^{-3}L^{-2} - 10^{-16}\kappa_0^{-3} L^{-2} \\
    &> 0. 
\end{align*}

Now observe that since $x^2 + y^2 + z^2 = 1$, at least one of the nonnegative numbers $x^2$ and $y^2 + z^2$ will be positive. We conclude from (\ref{eq:tameJ2}) and what was said above that
$$(\Phi_L)^*\Omega(V, \bar{J}_L(V)) > 0$$
as desired. 
\end{proof}

It will also be useful to write down some further observations regarding the geometry of the manifolds
$$(W_L, \bar J_L, g_L)$$
and
$$(\widetilde{W}_\pm, J_\pm, g_\pm).$$

First, we note that the cylindrical region 
$$[-L-3\epsilon_0, L+3\epsilon_0] \times M \subset W_L$$
can be given the structure of a \emph{realized Hamiltonian homotopy} as given in Definition \ref{defn:realizedHamiltonianHomotopy}. 

\begin{lem} \label{lem:transitionIsHomotopy} 
The region 
$$[-L-4\epsilon_0, L+4\epsilon_0] \times M \subset W_L$$
 is a realized Hamiltonian homotopy with $\hat\lambda = \lambda$ and $\hat\omega_L = \omega + \beta_L(t)d\lambda$. Moreover, the restriction $(J,g)$ of the almost-Hermitian structure $(\bar J_L, g_L)$ is such that $J$ is adapted with respect to the realized Hamiltonian homotopy $\hat\eta = (\hat\lambda, \hat\omega_L)$ and $g$ is the induced metric. 
\end{lem}

Second, it will be useful to define several compact, codimension zero submanifolds of $W_L$ and $\widetilde{W}_\pm$.

\begin{defn}
\label{defn:someCompactRegions}
Fix any $\ell > -8\epsilon_0$. Then we set
$$\widetilde{W}_{-}^\ell = \text{Core}(W_-) \cup [-8\epsilon_0, \ell] \times M \subset \widetilde{W}_-$$
and
$$\widetilde{W}_{+}^\ell = \text{Core}(W_+) \cup [-\ell, 8\epsilon_0] \times M \subset \widetilde{W}_+.$$

If furthermore $\ell < 2L + 8\epsilon_0$ we set
$$W_{L,-}^\ell = \text{Core}(W_-) \cup [-L - 8\epsilon_0, -L + \ell] \times M \subset W_L,$$
$$W_{L,+}^\ell = \text{Core}(W_+) \cup [L - \ell, L + 8\epsilon_0] \subset W_L,$$
and
$$W_{L}^\ell = W_{L,-}^\ell \cup W_{L,+}^\ell.$$
\end{defn}

The following lemma is immediate from our constructions of $(W_L, \bar J_L, g_L)$ and $(\widetilde{W}_\pm, \bar J_\pm, g_\pm)$. 

\begin{lem}
\label{lem:obviousInclusions}
Fix any $L > 0$ and any $\ell \in (-8\epsilon_0, 2L + 8\epsilon_0)$. 
Then the embedding
$$W_{L,-}^{\ell} \hookrightarrow \widetilde{W}_-$$
defined to be the identity on $\text{Core}(W_-)$ and the shift map $\text{Sh}_L$ on $[-L-8\epsilon_0, -L+\ell] \times M$ is an isomorphism from the almost-Hermitian manifold
$$(W_{L, -}^{\ell}, \bar J_L, g_L)$$
onto the almost-Hermitian manifold
$$(\widetilde{W}_{-}^{\ell}, \bar J_L, g_L).$$

The embedding
$$W_{L,+}^\ell \hookrightarrow \widetilde{W}_+$$
defined to be the identity on $\text{Core}(W_+)$ and the shift map $\text{Sh}_{-L}$ on $[L-\ell, L+8\epsilon_0] \times M$ is an isomorphism from the almost-Hermitian manifold
$$(W_{L, +}^\ell, \bar J_L, g_L)$$
onto the almost-Hermitian manifold
$$(\widetilde{W}_+^\ell, \bar J_+, g_+).$$
\end{lem}

\subsection{Constructing a sequence of pseudoholomorphic curves} 

Recall we have assumed that the symplectic manifold $(W, \Omega)$ and the hypersurface $M$ satisfy either the assumptions of Theorem \ref{thm:mainExample} or of Theorem \ref{thm:swExample}. 

\subsubsection{The construction in the setting of Theorem \ref{thm:mainExample}} 

Suppose first that the assumptions of Theorem \ref{thm:mainExample} are satisfied. By Corollary \ref{cor:GWCrossingExistence}, this implies that there is a closed submanifold $Y_+$ in $W_+$, a closed submanifold $Y_-$ in $W_-$, a homology class $A \in H_2(W; \mathbb{Z})$ and integers $G \geq 0$, $m \geq 0$ such that $2G + m \geq 1$ and for any $\Omega$-tame almost-complex structure $J$, there exists a stable, connected $J$-holomorphic curve
$$\mathbf{u} = (u, C, j, W, J, D, \mu)$$
where $\text{Genus}_{\text{arith}}(C, D) = G$, $\#\mu = m+2$, $u$ pushes forward $[C]$ to $A$, and $u(C)$ intersects both $Y_-$ and $Y_+$. 

We are free to choose the constant $\epsilon_0 > 0$ in the construction of Section \ref{subsec:neckStretching} sufficiently small so that $Y_-$ and $Y_+$ lie in the interior of $\text{Core}(W)$. By definition, we find they lie in the interiors of $\text{Core}(W_-)$ and $\text{Core}(W_+)$ respectively. For any $L > 0$, let $(W_L, \bar J_L, g_L)$ be the almost-Hermitian manifold constructed in Section \ref{subsec:neckStretching}. Recall by Proposition \ref{prop:tameJ}, for sufficiently large $L$, there is a diffeomorphism
$$\Phi_L: W_L \to W$$
such that $\bar J_L$ is tamed by $(\Phi_L)^*\Omega$. In other words, the almost-complex structure 
$$J_L = d\Phi_L \circ \bar J_L \circ d\Phi_L^{-1}$$
is tamed by $\Omega$, and there exists a $J_L$-holomorphic stable, connected pseudoholomorphic curve 
$$\mathbf{v}_L = (v_L, C_L, j_L, W, J_L, D_L, \mu_L)$$
of arithmetic genus $G$ with $m+2$ marked points representing the homology class $A$ that intersects $Y_-$ and $Y_+$.

Using the fact that $Y_\pm$ lie in the interior of $\text{Core}(W)$, we can define them as submanifolds of $W_L$ via the embedding
$$\text{Core}(W) \hookrightarrow W_L.$$

It follows that
$$\mathbf{u}_L = (u_L = \Phi_L^{-1} \circ v_L, C_L, j_L, W_L, \bar J_L, D_L, \mu_L)$$
is a stable, connected $\bar J_L$-holomorphic curve with arithmetic genus $G$ and $m+2$ marked points such that $u_L$ represents the class $(\Phi_L^{-1})_*A$ and intersects $Y_-$ and $Y_+$. 

Take the set of all $\mathbf{u}_L$ for $L \in \mathbb{N}$ sufficiently large to produce a sequence of stable, connected pseudoholomorphic curves
$$\mathbf{u}_k = (u_k, C_k, j_k, W_k, \bar J_k, D_k, \mu_k)$$
for $k \in \mathbb{N}$ sufficiently large. 

Now these pseudoholomorphic curves satisfy the following additional properties:
\begin{itemize}
    \item $\text{Genus}_{\text{arith}}(C_k, D_k) = G$ for every $k$. 
    \item $\#\mu_k = m + 2$ for every $k$. 
    \item There is a constant $\rho > 0$ depending only on $W$, $\Omega$, and $A$ such that for every $k$,
    $$\int_{C_k} (\Phi_k \circ u_k)^*\Omega \leq \rho.$$
    \item The image $u_k(C_k)$ intersects both components of $\text{Core}(W) \subset W_k$ for every $k$. 
\end{itemize}

Everything but the third bullet is immediate from the above considerations. The third bullet follows from the fact that $u_k$ represents the class $(\Phi_L^{-1})^*A$ for every $k$. The constant $\rho$ can be taken to be equal to the pairing $\langle \Omega, A \rangle$. 

\subsubsection{The construction in the setting of Theorem \ref{thm:swExample}} 

Now suppose instead that the assumptions of Theorem \ref{thm:swExample} are satisfied. Then by Proposition \ref{prop:nonzeroGr2}, there is an open neighborhood $V$ of $M$, a homology class $A \in H_2(W; \mathbb{Z})$, a real number $\rho > 0$ depending only on $W$, $\Omega$, and $A$, an integer $G \geq 0$ such that for any $\Omega$-tame almost-complex structure $J$, there exists a stable, connected $J$-holomorphic curve
$$\mathbf{u} = (u, C, j, W, J, D, \mu)$$
where $\text{Genus}_{\text{arith}}(C, D) = G$, $\#\mu = 0$, the integral of $u^*\Omega$ over $C$ is bounded above by $\rho$, and $u(C)$ intersects both components of $W \setminus V$. 

We assume that the constant $\epsilon_0$ from the neck stretching construction of Section \ref{subsec:neckStretching} is sufficiently small so that the collar neighborhood $[-8\epsilon_0, 8\epsilon_0] \times M \hookrightarrow W$ is contained in $V$. Proposition \ref{prop:nonzeroGr2} then produces for any sufficiently large $L$ a stable, connected $J_L$-pseudoholomorphic curves
$$\mathbf{v}_L = (v_L, C_L, j_L, W, J_L, D_L, \mu_L)$$
which we compose with the diffeomorphism $\Phi_L$ to produce a stable, connected $\bar J_L$-pseudoholomorphic curve
$$\mathbf{u}_L = (u_L = \Phi_L^{-1} \circ v_L, C_L, j_L, W_L, \bar J_L, D_L, \mu_L).$$

Take the set of all $\mathbf{u}_L$ for $L \in \mathbb{N}$ sufficiently large to produce a sequence of stable, connected pseudoholomorphic curves
$$\mathbf{u}_k = (u_k, C_k, j_k, W_k, \bar J_k, D_k, \mu_k)$$
for $k \in \mathbb{N}$ sufficiently large. 

The pseudoholomorphic curves $\mathbf{u}_k$ satisfy the following properties:
\begin{itemize}
    \item $\mathbf{u}_k$ is stable and connected.
    \item $\text{Genus}_{\text{arith}}(C_k, D_k) \leq G$ for every $k$. 
    \item $\#\mu_k = 0$ for every $k$. 
    \item There is a constant $\rho > 0$ depending only on $W$, $\Omega$, and $A$ such that for every $k$,
    $$\int_{C_k} (\Phi_k \circ u_k)^*\Omega \leq \rho.$$
    \item The image $u_k(C_k)$ intersects both components of $\text{Core}(W) \subset W_k$ for every $k$. 
\end{itemize}

This discussion has proved the following proposition, which we use as input for constructing a feral pseudoholomorphic curve.

\begin{prop}
\label{prop:sequenceOfCurves} Assume that $(W, \Omega)$ and the energy level $M = H^{-1}(0)$ satisfy the assumptions of either Theorem \ref{thm:mainExample} or of Theorem \ref{thm:swExample}. Then there is some $k_0 \geq 1000$, a constant $\rho > 0$ depending only on the symplectic manifold $(W, \Omega)$ and nonnegative integers $G$ and $m$ such that, for every $k \geq k_0$, there is a stable, connected pseudoholomorphic curve
$$\mathbf{u}_k = (u_k, C_k, j_k, W_k, \bar J_k, D_k, \mu_k)$$
satisfying the following properties:
\begin{itemize}
    \item $\mathbf{u}_k$ is stable and connected.
    \item $\text{Genus}_{\text{arith}}(C_k, D_k) \leq G$ for every $k$. 
    \item $\#\mu_k \leq m$ for every $k$. 
    \item There is a constant $\rho > 0$ depending only on $W$, $\Omega$, and $A$ such that for every $k$,
    $$\int_{C_k} (\Phi_k \circ u_k)^*\Omega \leq \rho.$$
    \item The image $u_k(C_k)$ intersects both components of $\text{Core}(W) \subset W_k$ for every $k$. 
\end{itemize}
\end{prop}

\subsection{Energy and topology bounds for exhaustive Gromov compactness}

Let 
$$\mathbf{u}_k = (u_k, C_k, j_k, W_k, \bar J_k, D_k, \mu_k)$$
be the stable, connected pseudoholomorphic curves given by Proposition \ref{prop:sequenceOfCurves}, and $\rho$, $G$, and $m$ the accompanying sequence of constants. 

Our plan now is to trim the curves $\mathbf{u}_k$ so that their images lie in the complement of a neighborhood of $\text{Core}(W_+) \subset W_k$ then apply the exhaustive Gromov compactness theorem introduced in Section \ref{subsec:exhaustiveCompactness}.

Along the way, we need to prove a few uniform analytic and topological estimates on the maps $u_k$. In what follows, we use the notation $\rho$ to denote any constant that independent of $k$ and depends only on $W$, $\Omega$, $M$, $\eta$, $g$, $J$, $\epsilon_0$, and the functions $\beta_*$ and $\{h_L\}_{L > 0}$. We will also need the notation $\rho_\ell$ to denote a constant independent of $k$, depending on all of the above that is also a function of a natural number $\ell$. We can assume that such constants increase in successive appearances. 

Our first bound is an area bound for the parts of the curves $\mathbf{u}_k$ in $\text{Core}(W)$. 

\begin{lem} \label{lem:outsideAreaBounds}
Write $$C_{k,\text{Core}} = u_k^{-1}(\text{Core}(W))$$ for every $k$. Then there is a constant $\rho > 0$ independent of $k$ such that, 
$$\text{Area}_{u_k^*g_k}(C_{k,\text{Core}}) \leq \rho$$
for every $k$. 
\end{lem}

\begin{proof}

On $C_{k,\text{Core}}$, the form $(\Phi_k \circ u_k)^*\Omega$ coincides with the two-form given by $g_k(\bar J_k-,-)$. The latter two-form is, for every $k$, the area form on $C_k$ with respect to the metric $g_k$ because the tangent planes of $u_k(C_k)$ are $J_k$-invariant.

It follows that
$$\text{Area}_{u_k^*g_k}(C_{k,\text{Core}}) = \int_{C_{k,\text{Core}}} (\Phi_k \circ u_k)^*\Omega.$$

Because $\bar J_k$ is tamed by $\Phi_k^*\Omega$, the two-form $(\Phi_k \circ u_k)^*\Omega$ is nonnegative on every tangent plane of $C_k$. Therefore, we have for every $k$ a bound of the form
\begin{align*}
    \text{Area}_{u_k^*g_k}(C_{k,\text{Core}}) &= \int_{C_{k,\text{Core}}} (\Phi_k \circ u_k)^*\Omega \\
    &\leq \int_{C_k} (\Phi_k \circ u_k)^*\Omega \\
    &\leq \rho.
\end{align*}

The second inequality follows from the integral bound given in Proposition \ref{prop:sequenceOfCurves}. 
\end{proof}

Our next bound is an area bound for the parts of the curves $\mathbf{u}_k$ in the regions 
$$[-k-8\epsilon_0,-k-3\epsilon_0] \times M \hookrightarrow W_k$$
and
$$[k+3\epsilon_0,k+8\epsilon_0] \times M \hookrightarrow W_k.$$

\begin{lem}
\label{lem:transitionAreaBounds} Write $C_{k,\text{Between}}$ for the preimage under $u_k$ of 
$$([-k-8\epsilon_0,-k-3\epsilon_0] \cup [k+3\epsilon_0,k+8\epsilon_0]) \times M.$$

Then there is a constant $\rho > 0$ independent of $k$ such that, 
$$\text{Area}_{u_k^*g_k}(C_{k,\text{Between}}) \leq \rho$$
for every $k$. 
\end{lem}

\begin{proof}
Recall the function $\chi_*$ taking values in $[0,1]$ from Section \ref{subsubsec:neckStretchingBetween}. Define for any $L > 0$ a smooth function
$$\chi_L(t) = \chi_*(t + L).$$

The area-form of the metric $u_k^*g_k$ on $C_{k,\text{Between}}$ is equal to the pullback 
$$u_k^*(\omega + \beta_k(a)d\lambda + \theta_k(a)(da \wedge \lambda)$$
where $\theta_k$ is equal to $\chi_k + (1 - \chi_k)\beta_k'$. 

On the other hand, the form $(\Phi_k \circ u_k)^*\Omega$ is equal to 
$$u_k^*(\omega + \beta_k(a)d\lambda + \beta_k'(a)(da \wedge \lambda).$$

We conclude that
\begin{equation} \label{eq:transitionAreaBounds1}
\begin{split}
    \text{Area}_{u_k^*g_k}(C_{k,\text{Between}}) &= \int_{C_{k,\text{Between}}} u_k^*(\omega + \beta_k(a)d\lambda + \theta_k(a)(da \wedge \lambda) \\
    &= \int_{C_{k,\text{Between}}} (\Phi_k \circ u_k)^*\Omega + \int_{C_{k,\text{Between}}} u_k^*(\chi_k(a)(1 - \beta_k'(a))(da \wedge \lambda)) \\
    &\leq \int_{C_k} (\Phi_k \circ u_k)^*\Omega + \int_{C_{k,\text{Between}}} u_k^*(\chi_k(a)(1 - \beta_k'(a))(da \wedge \lambda)).
\end{split}
\end{equation}

The last inequality follows from the fact $\bar J_k$ is tamed by $\Phi_k^*\Omega$, which implies that $(\Phi_k \circ u_k)^*\Omega$ is nonnegative on every tangent plane of $C_k$.  

The integral bound given in Proposition \ref{prop:sequenceOfCurves} gives a bound $\rho$ on the integral of $(\Phi_k \circ u_k)^*\Omega$ over $C_k$ which is independent of $k$. 

It remains to bound 
$$\int_{C_{k,\text{Between}}} u_k^*(\chi_k(a)(1 - \beta_k'(a))(da \wedge \lambda)).$$

Recall that $\beta_k'$ decreases monotonically from $1$ to $0$ on the interval $[-k-8\epsilon_0, -k-2\epsilon_0]$ and increases monotonically from $0$ to $1$ on the interval $[k+2\epsilon_0, k+8\epsilon_0]$. 

It follows that $\beta_k' \leq 1$ and that there is some constant $\delta \in (0, 1)$ such that
$$\beta_k'(t) \geq \delta$$
for $t$ in either $[-k-8\epsilon_0, -k-3\epsilon_0]$ or $[k+3\epsilon_0, k+8\epsilon_0]$. 

On this same interval, it then follows that
$$\theta_k = \chi(1 - \beta_k') + \beta_k' \leq 2$$
and
\begin{equation}\label{eq:transitionAreaBounds2}
\begin{split}
    \chi_k(1 - \beta_k') &= \theta_k - \beta_k' \\
    &\leq \theta_k - \delta \\
    &= \theta_k(1 - \delta\theta_k^{-1}) \\
    &\leq \theta_k(1 - \frac{\delta}{2}).
\end{split}
\end{equation}

We also observe that, by definition, the pullbacks of the two-forms $da \wedge \lambda$ and $\omega + \beta_k(a)$ are nonnegative on the tangent planes of $C_k$. 

It follows using (\ref{eq:transitionAreaBounds1}), (\ref{eq:transitionAreaBounds2}), the integral bound in Proposition \ref{prop:sequenceOfCurves} and rearranging that there is a $k$-independent constant $\rho > 0$ such that
\begin{align*}
    \frac{\delta}{2}\text{Area}_{u_k^*g_k}(C_{k,\text{Between}}) &\leq \rho.
\end{align*}

We conclude that
$$\text{Area}_{u_k^*g_k}(C_{k,\text{Between}}) \leq 2\delta^{-1}\rho$$
as desired.
\end{proof}

Recall that $W_k$ contains a cylindrical region of the form 
$$[-k - 3\epsilon_0, k+3\epsilon_0] \times M$$
where $\epsilon_0 > 0$ is some very small constant. Moreover, the almost-complex structure $\bar J_k$ is adapted with respect to the ``realized Hamiltonian homotopy'' structure (see Definitions \ref{defn:realizedHamiltonianHomotopy} and \ref{defn:homotopyadaptedJ}, as well as Lemma \ref{lem:transitionIsHomotopy})).

We define the $\lambda$-energy and $\omega$-energy for our curves $\mathbf{u}_k$ in the expected manner. Write $a: \mathbb{R} \times M \to \mathbb{R}$ for the coordinate projection. Then the composition $a \circ u_k$ is well-defined as a smooth map
$$u_k^{-1}([-k-8\epsilon_0, k+8\epsilon_0] \times M) \to [-k-8\epsilon_0, k+8\epsilon_0].$$

\begin{defn}
For any $k$, define
$$E_{\hat\omega}(u_k) = \int_{u_k^{-1}([-k-3\epsilon_0,k+3\epsilon_0] \times M)} u_k^*(\omega + d(\beta_k(a)\lambda))$$
and
$$E_\omega(u_k) = 
\int_{u_k^{-1}([-k-2\epsilon_0, k+2\epsilon_0] \times M)} u_k^*\omega.$$

For any $t$ that is a regular value of $a \circ u_k$, define
$$E_\lambda(u_k, t) = \int_{u_k^{-1}(\{t\} \times M)} u_k^*\lambda.$$
\end{defn}

The following lemma is essentially given in \cite[Lemma $3.7$]{FishHoferFeral}, but we write down a proof for the sake of convenience. 

\begin{lem} \label{lem:lambdaTrimBounds}
There is a constant $\rho > 0$ independent of $k$ such that for every $k$, there are real numbers
$$t_{k,-}, t_{k,+} \in (2\epsilon_0, 3\epsilon_0)$$
such that $-k - t_{k,-}$ and $k+t_{k,+}$ are regular values of $a \circ u_k$ and one has the energy bounds
$$E_\lambda(u_k, -k-t_{k,-})\leq \rho$$
and
$$E_\lambda(u_k, k+t_{k,+})\leq \rho$$
\end{lem}

\begin{proof}
Define the interval 
$$I_k = [-k - 3\epsilon_0, -k - 11\epsilon_0/4].$$

Recall that the function $h_k(t)$ used to define the diffeomorphism $\Phi_k$ from Proposition \ref{prop:tameJ} is equal to $\beta_k(t)$ on $I_k$. 

Set $$\delta = \inf_{t \in I_k} \beta_k(a) > 0.$$ 

By construction, the functions 
$$a \mapsto \beta_k(k + a)$$
are positive and agree on the interval $[-8\epsilon_0, -2\epsilon_0]$ for every $k$, so $\delta$ can be taken to be independent of $k$. 

Recall that on $I_k \times Y$, one can write
$$\Phi_k^*\Omega = \omega + d(\beta_k(a))\lambda = \omega + \beta_k(a)d\lambda + \beta_k'(a)da \wedge \lambda.$$

This follows from the fact that
$$\Phi_k^*\Omega = \omega + d(h_k(a))\lambda = \omega + h_k(a)d\lambda + h_k'(a)da \wedge \lambda$$
for a smooth function $h_k$ which is equal to $\beta_k$ on $I_k$.

By our construction, the two-forms
$$da \wedge \lambda$$
and
$$\omega + \beta_k(a)d\lambda$$
are non-negative on $\bar J_k$-complex lines, so the pullbacks
$$u_k^*(da \wedge \lambda)$$
and
$$u_k^*(\omega + \beta_k(a)d\lambda)$$
are non-negative on the tangent planes of $C_k$. 

It follows that
\begin{align*}
    \int_{u_k^{-1}(I_k \times M)} u_k^*(da \wedge \lambda) &\leq \delta^{-1}\int_{u_k^{-1}(I_k \times M)} u_k^*(\beta_k'(a)da \wedge \lambda) \\
    &\leq \delta^{-1}\int_{u_k^{-1}(I_k \times M)} u_k^*(\omega + \beta_k(a)d\lambda + \beta_k'(a)da \wedge \lambda + \omega) \\
    &= \delta^{-1}\int_{u_k^{-1}(I_k \times M)} (\Phi_k \circ u_k)^*\Omega \\
    &\leq \delta^{-1}\int_{C_k} (\Phi_k \circ u_k)^*\Omega \\
    &\leq \delta^{-1}\rho.
\end{align*}

The last inequality follows from the integral bound given by Proposition \ref{prop:sequenceOfCurves}. 

Now by the co-area formula (see \cite[Lemma $4.13$]{FishHoferFeral} for a precise version for this scenario), one has
$$\int_{u_k^{-1}(I_k \times Y)} da \wedge \lambda = \int_{I_k} E_\lambda(u_k, a) da.$$

It follows that there is \emph{some} $t_{k,-} \in (2\epsilon_0, 3\epsilon_0)$ such that
$$E_\lambda(u_k, k - t_{k,-}) \leq 100\delta^{-1}\epsilon_0^{-1}\rho.$$

The derivation of a number $t_{k,+}$ is identical, and so the lemma is proved.
\end{proof}

We can also bound the $\hat\omega$-energy of $u_k$ uniformly in $k$ using the bound on the integrals of $(\Phi_k \circ u_k)^*\Omega$ from Proposition \ref{prop:sequenceOfCurves}. 

\begin{lem} \label{lem:sameSmallOmega}
There is a constant $\rho > 0$ independent of $k$ such that for every $k$,
$$E_{\hat\omega}(u_k) \leq \rho.$$
\end{lem}

\begin{proof}
On $u_k^{-1}([-k-3\epsilon_0,k+3\epsilon_0] \times M)$, one has
$$(\Phi_k \circ u_k)^*\Omega = u_k^*(\omega + d(h_k(a)\lambda)$$
where we recall that $h_k$ is the function from Lemma \ref{lem:tameJ1} used to construct the map $\Phi_k$. 

The integral of this two-form over $u_k^{-1}([-k-3\epsilon_0,k+3\epsilon_0] \times M)$ is bounded by a $k$-independent constant $\rho > 0$ by combining the fact that $(\Phi_k \circ u_k)^*\Omega$ is nonnegative on the tangent planes of $C_k$, and the integral bound provided by Proposition \ref{prop:sequenceOfCurves}. 

To bound $E_{\hat\omega}(u_k)$, which is the integral of $u_k^*(\omega + d(\beta_k(a)\lambda))$ over the same domain, it therefore suffices to bound
$$\int_{u_k^{-1}([-k-3\epsilon_0,k+3\epsilon_0] \times M)} d( (\beta_k(a) - h_k(a))\lambda).$$

Note that $\beta_k(t) = h_k(t)$ for $t$ in a small neighborhood of $\{-k-3\epsilon_0,k+3\epsilon_0\}$. Therefore, by Stokes' theorem, it is immediate that this integral is equal to zero. 
\end{proof}

Using Lemma \ref{lem:outsideAreaBounds}, Lemma \ref{lem:transitionAreaBounds}, Lemma \ref{lem:lambdaTrimBounds}, Lemma \ref{lem:sameSmallOmega}, and the exponential area bound for realized Hamiltonian homotopies in Theorem \ref{thm:fishHoferHomotopyAreaBound}, we obtain the following area bound for portions of the curves $C_k$ mapping into $W \setminus (-k+\ell, k-\ell) \times M$.

Notice that it is reminiscent of the type of area bound required for exhaustive Gromov compactness (Theorem \ref{thm:exhaustiveGromov}), which is precisely what we will use it for. 

\begin{cor} \label{cor:transitionAreaBounds}
There is a constant $\rho_\ell > 0$ independent of $k$ such that the following holds. For any constant $\ell \in (-8\epsilon_0, 2k+8\epsilon_0)$, write 
$$C_k^\ell = u_k^{-1}(W_k^\ell).$$

Then there is a constant $\rho_\ell > 0$ independent of $k$, but possibly depending on $\ell$ such that 
$$\text{Area}_{u_k^*g_k}(C_k^\ell) \leq \rho_\ell.$$
\end{cor}

Next, we will prove a uniform bound in $k$ on the number of connected components of the domain $C_k$ of the pseudoholomorphic curve $\mathbf{u}_k$ on which $u_k$ is non-constant.

\begin{prop}
\label{prop:ncConnectedComponentsBound} There is a constant $\rho > 0$ independent of $k$ such that, for any $k$, $C_k$ has at most $\rho$ connected components on which $u_k$ is non-constant. 
\end{prop}

\begin{proof}

We write $\mathcal{C}_{k,\text{nc,in}}$ for the set of connected components $\Sigma$ on which $u_k$ is non-constant and 
$$u_k(\Sigma) \subset (-k+50, k-50) \times M \subset W_k.$$

Write $\mathcal{C}_{k,\text{nc,out}}$ for the set of connected components of $\Sigma$ on which $u_k$ is non-constant and are not in $\mathcal{C}_{k,\text{nc,in}}$. In other words, any component $\Sigma \in \mathcal{C}_{k,\text{nc,out}}$ has $u_k(\Sigma)$ intersecting 
$$W_k^{50} \subset W_k.$$

The proof proceeds in two steps. The first step is to bound $\#\mathcal{C}_{k,\text{nc,in}}$, and the second step is to bound $\#\mathcal{C}_{k,\text{nc,out}}$. The proposition follows from these two bounds. 

We begin with the first step. Pick $\Sigma \in \mathcal{C}_{k,\text{nc,in}}$. 

Since $\Sigma$ has no boundary, we can appeal to Theorem \ref{thm:omegaEnergyQuantization} with the parameters $a_0 = \text{inf}(a \circ u_k)(\Sigma)$, which by assumption lies in $[-k+50, k-50]$, and $h = 25$ to conclude that there is a constant $\hbar$ independent of $k$ such that
$$\int_{\Sigma} u_k^*\omega \geq \hbar.$$

It follows that
$$\int_{u_k^{-1}( (-k+50, k-50) \times M)} u_k^*\omega \geq \hbar\#\mathcal{C}_{k,\text{nc,in}}.$$

On the other hand, Lemma \ref{lem:sameSmallOmega} implies that 
$$\int_{u_k^{-1}( (-k+50,k-50) \times M)} u_k^*\omega \leq \rho_1$$
for some constant $\rho_1$ independent of $k$. 

We combine these two inequalities to conclude that
$$\#\mathcal{C}_{k,\text{nc,in}} \leq \hbar^{-1}\rho_1.$$

We now proceed with the second step. 

Pick $\Sigma \in \mathcal{C}_{k,\text{nc,out}}$. Then $u_k(\Sigma)$ passes through
$$W_k^{50} \subset W_k.$$

Let $\widetilde{\Sigma} \subset \Sigma$ be equal to the intersection of $\Sigma$ with
$$\widetilde{C}_k = u_k^{-1}(W_k^{100}).$$

Observe that, by our assumptions, $u_k(\widetilde{\Sigma})$ contains a point with $g_k$-distance at least $50$ from the boundary of $W_k^{100}$. 

We now observe that, for any $k$, the almost-Hermitian manifold $$(W_k^{100}, \bar J_k, g_k)$$
is isomorphic to the disjoint union of
$$(\widetilde{W}_-^{100}, \bar J_-, g_-)$$
and
$$(\widetilde{W}_+^{100}, \bar J_+, g_+).$$

It follows by Theorem \ref{thm:monotonicity} that there is a constant $\rho_2 \geq 1$ independent of $k$ such that
$$\text{Area}_{u_k^*g_k}(\widetilde{\Sigma}) \geq \rho_2^{-1}.$$

This implies
$$\text{Area}_{u_k^*g_k}(\widetilde{C}_k) \geq \rho_2^{-1}\#\mathcal{C}_{k,\text{nc,out}}.$$

On the other hand, Corollary \ref{cor:transitionAreaBounds} implies that there is a constant $\rho_3 > 0$ independent of $k$ such that
$$\text{Area}_{u_k^*g_k}(\widetilde{C}_k) \leq \rho_3.$$

Combine the two inequalities above to deduce
$$\#\mathcal{C}_{k,\text{nc,out}} \leq \rho_2\rho_3$$
and as a result complete the proof of the proposition.
\end{proof}

Proposition \ref{prop:ncConnectedComponentsBound} allows us to bound the number of nodal points $\#D_k$ uniformly in $k$. This can be done using the following explicit formula for the arithmetic genus.

\begin{lem} \label{lem:arithGenusFormula}
\cite[Lemma $A.1$]{FishHoferExhaustive} Let $(\Sigma, j, D)$ be a compact nodal Riemann surface, possibly with boundary. Let $\#\pi_0(|\Sigma|)$ the number of connected components of the surface obtained from $\Sigma$ by performing connect sums at all of the pairs of nodal points. Then
$$\text{Genus}_{\text{arith}}(\Sigma, D) = \#\pi_0(|\Sigma|) - \#\pi_0(\Sigma) + \sum_{S \in \pi_0(\Sigma)} \text{Genus}(S) + \frac{1}{2}\#D.$$
\end{lem}

\begin{rem}
It will also be convenient later to fix the following notation. Let $(\Sigma, j, D)$ be a nodal Riemann surface. 

We will say a surface $\Sigma' \subset \Sigma$ is \textbf{$|\Sigma|$-connected} if the surface given by performing connect sums at the nodal points $D \cap \Sigma'$ is connected. We will always assume for the purposes of this definition that, if one point in a pair $\{\underline{z}, \overline{z}\}$ lies in $\Sigma'$, then the other lies in $\Sigma'$ as well. 

We will say a surface $\Sigma' \subset \Sigma$ is a \textbf{connected component of $|\Sigma|$} if it is a maximal element in the set of $|\Sigma|$-connected sub-surfaces of $\Sigma$, partially ordered by inclusion. 

We see given this that a pseudoholomorphic curve 
$$\mathbf{u} = (u, C, j, W, J, D, \mu)$$
is connected if and only if $C$ is the only connected component of $|C|$. 
\end{rem}

We now prove an a priori uniform bound on the number of nodal points. 

\begin{prop}
\label{prop:nodalPointBound} There is a constant $\rho > 0$ independent of $k$ such that, for any $k$,
$$\#D_k \leq 12G + 6m + \rho.$$
\end{prop}

\begin{proof}
Use Lemma \ref{lem:arithGenusFormula} to deduce the identity
\begin{equation} \label{eq:nodalPointBound1} \text{Genus}_{\text{arith}}(C_k, D_k) = \#\pi_0(|C_k|) - \#\pi_0(C_k) + \sum_{\Sigma \in \pi_0(C_k)} \text{Genus}(\Sigma) + \frac{1}{2}\#D_k.\end{equation}

We now introduce the following notation. Let $\mathcal{C}_{k,\text{nc}}$ denote the set of all connected components of $C_k$ on which the map $u_k$ is non-constant. 

For any natural number $h \geq 0$, let 
$$\mathcal{C}_{k,\text{const,good}}$$
denote the set of connected components $\Sigma$ of $C_k$ such that $u_k$ is constant on $\Sigma$ and either
\begin{enumerate}
    \item $\Sigma$ has genus $\geq 1$.
    \item $\Sigma$ contains at least one point from the set $\mu_k$ of marked points. 
\end{enumerate}

Let
$$\mathcal{C}_{k,\text{const,bad}}$$
denote the set of connected components $\Sigma$ of $C_k$ that is not in either $\mathcal{C}_{k,\text{nc}}$ or $\mathcal{C}_{k,\text{nc,good}}$. In other words, $\Sigma$ lies in $\mathcal{C}_{k,\text{const,bad}}$ if and only if $u_k$ is constant on $\Sigma$, $\Sigma$ has genus zero, and $\Sigma$ contains no marked points. 

Expanding (\ref{eq:nodalPointBound1}) and removing manifestly non-negative terms, we conclude that
\begin{equation} \label{eq:nodalPointBound2} \text{Genus}_{\text{arith}}(C_k, D_k) + \#\mathcal{C}_{k,\text{nc}} + \#\mathcal{C}_{k,\text{const,good}} \geq \frac{1}{2}\#D_k - \#\mathcal{C}_{k,\text{const,bad}}. \end{equation}

Recall from Proposition \ref{prop:sequenceOfCurves} that
\begin{equation} \label{eq:nodalPointBound3} \text{Genus}_{\text{arith}}(C_k, D_k) \leq G.\end{equation}

Recall from Proposition \ref{prop:ncConnectedComponentsBound} that we have a bound of the form
\begin{equation}
    \label{eq:nodalPointBound4} \#\mathcal{C}_{k,\text{nc}} \leq \rho
\end{equation}
where $\rho$ is independent of $k$. 

By definition, 
\begin{equation} \label{eq:nodalPointBound5} \#\mathcal{C}_{k,\text{const,good}} \leq G + m. \end{equation}

This is because any member of $\mathcal{C}_{k,\text{const,good}}$ has genus greater than zero or at least one marked point. There can be at most $G$ members with genus greater than zero due to (\ref{eq:nodalPointBound3}) and at most $m$ members with at least one marked point because $\#\mu_k \leq m$ by Proposition \ref{prop:sequenceOfCurves}. 

Finally, since the curves $\mathbf{u}_k$ are stable, every member of $\mathcal{C}_{k,\text{const,bad}}$ has at least three nodal points. This shows
\begin{equation} \label{eq:nodalPointBound6} \#\mathcal{C}_{k,\text{const,bad}} \leq \frac{1}{3}\#D_k. \end{equation}

Plug (\ref{eq:nodalPointBound3}), (\ref{eq:nodalPointBound4}), (\ref{eq:nodalPointBound5}) and (\ref{eq:nodalPointBound6}) into (\ref{eq:nodalPointBound2}) to deduce that
$$2G + m + \rho \geq \frac{1}{6}\#D_k.$$

This implies the desired bound. 
\end{proof}

We conclude by bounding the number of connected components of $C_k$ on which $u_k$ is constant. This gives a uniform bound on the overall number of connected components of $C_k$ for every $k$.

\begin{prop}
\label{prop:constantConnectedComponentsBound} There is a constant $\rho > 0$ independent of $k$ such that, for any $k$, $C_k$ has at most $5G + 3m + \rho$ connected components on which $u_k$ is constant. 
\end{prop}

\begin{proof}
For every $k$, write $\mathcal{C}_{k,\text{const}}$ for the set of components of $C_k$ on which $u_k$ is constant. 

Observe in the course of the proof of Proposition \ref{prop:nodalPointBound} that we proved the bound
$$\#\mathcal{C}_{k,\text{const}} \leq G + m + \frac{1}{3}\#D_k.$$

This follows from adding together the inequalities in (\ref{eq:nodalPointBound5}) and (\ref{eq:nodalPointBound6}). 

Plugging in the bound on $\#D_k$ from Proposition \ref{prop:nodalPointBound} gives a uniform bound on the size of $\mathcal{C}_{k,\text{const}}$. 
\end{proof}

\subsection{Trimming the curves}

We will now begin to define the trimmed pseudoholomorphic curves. 

\begin{prop} \label{prop:boundaryComponentBound}
There is a strictly monotonic sequence of natural numbers $\sigma_i \to \infty$ and a constant $\rho > 0$ independent of $i$ such that the following holds. 

For every $i$, there is a compact surface
$$\widetilde{C}_i \subset C_{\sigma_i}$$
with at most $\rho$ boundary components and $\rho$ connected components such that, for every $i$, the following two conditions hold:
\begin{itemize} 
\item There is a connected component $C'$ of $|\widetilde{C}_i|$ such that $u_{\sigma_i}(C')$ that intersects
$$\text{Core}(W_-) \subset W_{\sigma_i}$$
and
$$W_{\sigma_i,-}^{2\sigma_i-1/4} \setminus W_{\sigma_i,-}^{2\sigma_i-3/8}.$$
\item The boundary $\partial\widetilde{C}_i$ satisfies
$$u_{\sigma_i}(\partial\widetilde{C}_i) \subset W_{\sigma_i,-}^{2\sigma_i-1/4} \setminus W_{\sigma_i,-}^{2\sigma_i-3/8}.$$
\end{itemize}
\end{prop}

\begin{proof}
Fix for any $k$ a real number $t_{k,\text{trim}} \in (1/4, 3/8)$ such that $k - t_{k,\text{trim}}$ is a regular value of $a \circ u_k$ and does not lie in $(a \circ u_k)(D_k \cup \mu_k)$. Define the compact surface
$$\Sigma_k = u_k^{-1}(W_{k,+}^{t_{k,\text{trim}}}.$$

Recall from Proposition \ref{prop:sequenceOfCurves} that, for every $k$, $u_k(C_k)$ is connected and intersects both connected components of $\text{Core}(W) \subset W_k$. It follows that $\Sigma_k$ is not empty, has nonempty boundary, and 
$$u_k(\partial\Sigma_k) \subset \{k - t_{k,\text{trim}}\} \times M.$$

Recall from Lemma \ref{lem:obviousInclusions} that the almost-Hermitian manifolds $$(W_{k,+}^{t_{k,\text{trim}}}, \bar J_k, g_k)$$
are isomorphic to the almost-Hermitian manifolds
$$(\widetilde{W}_+^{t_{k,\text{trim}}}, \bar J_+, g_+)$$ 
which themselves embed into 
$$(\widetilde{W}_+^{1/2}, \bar J_+, g_+)$$
for every $k$. 

Using Corollary \ref{cor:transitionAreaBounds} and the fact that $t_{k,\text{trim}} < 1/2$ for every $k$, we conclude that there is a constant $\rho > 0$ independent of $k$ such that
$$\text{Area}_{u_k^*g_k}(\Sigma_k) \leq \rho.$$

Define stable pseudoholomorphic curves
$$\mathbf{v}_k = (v_k, \Sigma_k, j_k, \widetilde{W}_+^{1/2}, \bar J_-, D_k \cap \Sigma_k, \mu_k \cap \Sigma_k)$$
by setting $v_k$ to be the composition of the restriction of $u_k$ to $\Sigma_k$ with the embedding
$$W_{k,+}^{t_{k,\text{trim}}} \hookrightarrow \widetilde{W}_+^{1/2}.$$

Here $D(\Sigma_k) = \Sigma_k \cap D_k$ and $\mu(\Sigma_k) = \mu_k \cap \Sigma_k$. 

It follows that
$$\text{Area}_{v_k^*g_+}(\Sigma_k) = \text{Area}_{u_k^*g_k}(\Sigma_k) \leq \rho.$$

By Proposition \ref{prop:sequenceOfCurves}, we have
$$\text{Genus}_{\text{arith}}(\Sigma_k, D(\Sigma_k)) \leq G$$
and
$$\#\mu(\Sigma_k) \leq m.$$

By Proposition \ref{prop:nodalPointBound}, we have a $k$-independent bound 
$$\#D(\Sigma_k) \leq \rho.$$

Using the target-local Gromov compactness theorem (Theorem \ref{thm:targetLocal}), there is a strictly monotonic sequence of natural numbers
$$\sigma_i \to \infty$$
and, for every $i$, a compact surface
$$\widetilde{\Sigma}_i \subset \Sigma_{\sigma_i}$$
such that
$$v_{\sigma_i}(\widetilde{\Sigma}_i) \subset \widetilde{W}_+^{1/2},$$
$$v_{\sigma_i}(\partial\widetilde{\Sigma}_i) \subset \widetilde{W}_+^{1/2} \setminus \widetilde{W}_+^{3/8},$$
and the pseudoholomorphic curves
$$\widetilde{\mathbf{v}}_i = (v_{\sigma_i}, \widetilde{\Sigma}_i, j_{\sigma_i}, \widetilde{W}_+^{1/2}, \bar J_+, D_{\sigma_i} \cap \widetilde{\Sigma}_i, \mu_{\sigma_i} \cap \widetilde{\Sigma}_i)$$
converge in the Gromov sense.

A consequence of the Gromov convergence is that there are decorations $r_i$ such that the compact, connected surfaces $$\widetilde{\Sigma}_i^{D_{\sigma_i} \cap \widetilde{\Sigma}_i, r_i}$$
are diffeomorphic for sufficiently large $i$, and therefore have the same number of boundary components for sufficiently large $i$. 
By definition, $\widetilde{\Sigma}_i$ and $\widetilde{\Sigma}_i^{D_{\sigma_i} \cap \widetilde{\Sigma}_i, r_k}$ have the same number of boundary components. 

We conclude that there is some $i$-independent constant $\rho > 0$ such that $\widetilde{\Sigma}_i$ has at most $\rho$ boundary components. 

Set
$$\widetilde{C}_i = C_{\sigma_i} \setminus \widetilde{\Sigma}_i.$$

Note that, since there is only one connected component of $|C_{\sigma_i}|$ and $u(C_{\sigma_i})$ intersects both components of $\text{Core}(W)$, we find that there must be a connected component $C'$ of $\widetilde{C}_i$ such that $u_{\sigma_i}(C')$ intersects $\text{Core}(W_-)$ and $W_{\sigma_i,-}^{2\sigma_i - 1/4}\setminus W_{\sigma_i,-}^{2\sigma_i-3/8}$.

Moreover,
$$u_{\sigma_i}(\partial\widetilde{C}_i) \subset W_{\sigma_i,-}^{2\sigma_i-1/4} \setminus W_{\sigma_i,-}^{2\sigma_i-3/8}.$$

Finally, we have 
$$\#\pi_0(\partial \widetilde{C}_i) = \#\pi_0(\partial\widetilde{\Sigma}_i)$$
and, by the Mayer-Vietoris sequence,
$$\#\pi_0(\widetilde{C}_i) \leq \#\pi_0(\partial\widetilde{C}_k) + \#\pi_0(C_{\sigma_i}).$$

Using the fact that $\#\pi_0(\partial\widetilde{\Sigma}_i)$ is uniformly bounded in $i$, along with the uniform bound on $\#\pi_0(C_k)$ coming from Propositions \ref{prop:ncConnectedComponentsBound} and \ref{prop:constantConnectedComponentsBound}, we conclude that the surfaces $\widetilde{C}_i$ satisfy the conditions of the proposition.
\end{proof}

Let $\sigma_i \to \infty$ be the strictly monotonic sequence of natural numbers and $\widetilde{C}_i$ be the sequence of surfaces in $C_{\sigma_i}$ guaranteed by Proposition \ref{prop:boundaryComponentBound}. 

Set
$$\widetilde{\mathbf{u}}_i = (\widetilde{u}_i, \widetilde{C}_i, W_{\sigma_i,-}^{2\sigma_i}, \bar J_{\sigma_i}, \widetilde{D}_i, \widetilde{\mu}_i)$$
to be the restriction of $\mathbf{u}_{\sigma_i}$ to $\widetilde{C}_i$. 

By restricting to $\widetilde{C}_k$, we are trimming away the part of $C_k$ that maps into a neighborhood of $W_+ \subset W_k$.

Then set $w_i$ to be the composition of $\widetilde{u}_i$ with the embedding
$$W_{\sigma_i,-}^{2\sigma_i} \subset \widetilde{W}_-$$
and define the sequence of pseudoholomorphic curves
$$\mathbf{w}_i = (w_i, \widetilde{C}_i, \widetilde{W}_-, \bar J_-, \widetilde{D}_i, \widetilde{\mu}_i).$$

We summarize the properties of the sequence $\mathbf{w}_k$ in the following proposition.

\begin{prop} 
\label{prop:sequenceOfCurves2} 
Assume that $(W, \Omega)$ and the energy level $M = H^{-1}(0)$ satisfy the assumptions of either Theorem \ref{thm:mainExample} or of Theorem \ref{thm:swExample}.

Then there is a sequence of natural numbers $\sigma_i \to \infty$, a sequence of stable pseudoholomorphic curves
$$\mathbf{w}_i = (w_i, \widetilde{C}_i, \widetilde{W}_-, \bar J_-, \widetilde{D}_i, \widetilde{\mu}_i)$$
such that the following conditions are satisfied. 

First, for every $i$ there is a connected component $C'$ of $|\widetilde{C}_i|$ such that $w_i(C')$ intersects both $\text{Core}(W_-)$ and $\widetilde{W}_-^{2\sigma_i - 1/4} \setminus \widetilde{W}_-^{2\sigma_i - 3/8}.$

Second,
$$w_i(\partial\widetilde{C}_i) \subset \widetilde{W}_-^{2\sigma_i - 1/4} \setminus \widetilde{W}_-^{2\sigma_i - 3/8}.$$

Third, there is a constant $\rho > 0$ independent of $i$ such that the following bounds are satisfied:
\begin{itemize}
    \item $\text{Genus}_{\text{arith}}(\widetilde{C}_i, \widetilde{D}_i) \leq \rho$.
    \item $\#\widetilde{\mu}_i \leq \rho$.
    \item $\#\widetilde{D}_i \leq \rho$.
    \item $\#\pi_0(\widetilde{C}_i) \leq \rho$.
    \item $\#\pi_0(\partial\widetilde{C}_i) \leq \rho$. 
    \item $E_\omega(\mathbf{w}_i) \leq \rho$.
\end{itemize}

Fourth, the integral of $\omega$ over $w_i^{-1}([0, \infty) \times M)$ is bounded above by $\rho$ for every $i$. 
\end{prop}

The first and second assertions in Proposition \ref{prop:sequenceOfCurves2} are from Proposition \ref{prop:boundaryComponentBound}. The bounds in the third assertion follow from Propositions \ref{prop:boundaryComponentBound}, \ref{prop:nodalPointBound}, and \ref{prop:sequenceOfCurves}, along with Lemma \ref{lem:sameSmallOmega}.

The almost-Hermitian manifolds
$$(\widetilde{W}_-^\ell, \bar J_-, g_-)$$
form an exhausting sequence of 
$$(\widetilde{W}_-, \bar J_-, g_-).$$

The area bound from Corollary \ref{cor:transitionAreaBounds} implies the following area bound for the sequence $\mathbf{w}_i$. 

\begin{prop}
\label{prop:exhaustiveAreaBounds} For any $\ell$ and any $i$ such that $2\sigma_i - 1 \geq \ell$, there is a constant $\rho_\ell > 0$ independent of $i$ such that if we set
$$\widetilde{C}_i^\ell = w_i^{-1}(\widetilde{W}_-^\ell)$$
we have
$$\text{Area}_{w_i^*g_-}(\widetilde{C}_i^\ell) \leq \rho_\ell.$$
\end{prop}

Proposition \ref{prop:exhaustiveAreaBounds} and the bounds from Proposition \ref{prop:sequenceOfCurves2} imply that the pseudoholomorphic curves $\mathbf{w}_i$ satisfy the conditions of the exhaustive Gromov compactness theorem (Theorem \ref{thm:exhaustiveGromov}) for the exhausting sequence $\widetilde{W}_-^\ell$. 

The theorem itself applies to connected pseudoholomorphic curves, but since we have a uniform bound on $\#\pi_0(\widetilde{C}_i)$, it follows that we can apply it in this setting as well, by applying it successively to restrictions of $\mathbf{w}_i$ to connected components of $|\widetilde{C}_i|$.  

We conclude the following lemma.

\begin{lem} \label{lem:exhaustiveGromovLimit}
There is a stable pseudoholomorphic curve
$$\bar{\mathbf{w}} = (w, \bar C, \bar j, \widetilde{W}_-, \bar J_-, \bar D, \bar \mu)$$
such that the following two conditions hold:
\begin{itemize}
\item There are finitely many connected components of $|\bar C|$.
\item After passing to a subsequence, the sequence of pseudoholomorphic curves $\mathbf{w}_i$ from Proposition \ref{prop:sequenceOfCurves2} converges to $\bar{\mathbf{w}}$
in the exhaustive Gromov sense.
\end{itemize}
\end{lem}

\subsection{The limit curve is $\omega$-finite}

Choose some number $\bar{t}_{\text{trim}} \in (0, 1)$ such that
$\bar{t}_{\text{trim}}$ is a regular value of $a \circ \bar{\mathbf{w}}$ and 
$$\bar{t}_{\text{trim}} \not\in (a \circ \bar{\mathbf{w}})(\bar D \cap \bar \mu).$$

Let
$$\hat{\mathbf{w}} = (\hat w, \hat C, \hat j, [0, \infty) \times M, \hat J_-, \hat D, \hat \mu)$$
be the restriction of $\bar{\mathbf{w}}$ to 
$$\hat C = \bar{\mathbf{w}}^{-1}([\bar{t}_{\text{trim}}, \infty) \times M).$$

We will now prove the following. 

\begin{prop} \label{prop:limitCurveOmegaFinite}
The pseudoholomorphic curve $\hat{\mathbf{w}}$ is $\omega$-finite.
\end{prop}

The proof of the proposition will be proved over the course of several lemmas.

\begin{lem} \label{lem:limitCurveProper} The pseudoholomorphic map $\hat w$ is proper. \end{lem}

\begin{proof}
This follows from showing that the pseudoholomorphic map 
$$\bar w: \bar C \to \widetilde{W}_-$$
is proper. 

This in turn follows from the fact that
$$\mathbf{w}_i \to \bar{\mathbf{w}}$$
in the exhaustive Gromov sense. 

The definition of exhaustive Gromov compactness implies that the following holds.

Fix any compact subset $K \subset \widetilde{W}_i$ and sufficiently large $\ell$ such that $K \subset \widetilde{W}_i^{\ell - 10}$. Then for sufficiently large $i$, there is a decoration $r_i^\ell$ on a non-empty, compact surface
$$\widetilde{C}_i^\ell \subset \widetilde{C}_i$$
containing
$$w_i^{-1}(\widetilde{W}_-^{\ell - 5})$$
so that the closed set $\bar{w}^{-1}(K)$ embeds continuously into 
$$(\widetilde{C}_i^\ell)^{\widetilde{D}_i \cap \widetilde{C}_i^\ell, r_i^\ell}.$$

It follows that $\bar{w}^{-1}(K)$ is compact. 
\end{proof}

Next, we show that the arithmetic genus of $\bar C$ is finite.

\begin{lem}
\label{lem:limitCurveGenusBound} Let $\rho$ be the constant in Proposition \ref{prop:sequenceOfCurves2}. Then the arithmetic genus of $\bar C$ is at most $\rho$.
\end{lem}

\begin{proof}
Recall from Proposition \ref{prop:sequenceOfCurves2} that
$$\text{Genus}_{\text{arith}}(\widetilde{C}_i, \widetilde{D}_i) \leq \rho$$
for some constant $\rho > 0$ independent of $i$. 

We now use the fact that
$$\mathbf{w}_i \to \bar{\mathbf{w}}$$
in the exhaustive Gromov sense. 

We write out the implication of this exhaustive Gromov convergence. 

For any $i$ and $\ell \leq 2\sigma_i - 1$ there is a compact surfaces $\hat{C}_i^\ell$ in $\widetilde{C}_i$. 

There is also a sequence of compact surfaces
$\bar{C}^\ell$ exhausting $\bar C$.

Now for any fixed $\ell$ and any $i$ sufficiently large, there are decorations $r_i^\ell$ on the nodal surfaces $(\hat{C}_i^\ell, \widetilde{j}_i, \widetilde{D}_i \cap \hat{C}_i^\ell)$ and a decoration $r^\ell$ on the nodal surface $(\bar{C}^\ell, \bar j, \bar D \cap \bar{C}^\ell)$ such that the normalizations 
$$(\hat{C}_i^\ell)^{\widetilde{D}_i \cap \hat{C}_i^\ell, r_i^\ell}$$
are all diffeomorphic to the normalization
$$(\bar{C}^\ell)^{\bar{D} \cap \bar{C}^\ell, r^\ell}.$$

It follows that, for any fixed $\ell$ and any $i$ sufficiently large that
\begin{equation}
    \label{eq:limitCurveGenusBound1}
    \begin{split} \text{Genus}_{\text{arith}}(\hat{C}_i^\ell, \widetilde{D}_i \cap \hat{C}_i^\ell) &= \text{Genus}((\hat{C}_i^\ell)^{\widetilde{D}_i \cap \hat{C}_i^\ell, r_i^\ell}) \\
    &= \text{Genus}((\bar{C}^\ell)^{\bar{D} \cap \bar{C}^\ell, r^\ell}) \\
    &= \text{Genus}_{\text{arith}}(\bar{C}^\ell, \bar{D} \cap \bar{C}^\ell).\end{split}
\end{equation}

By definition, we find 
\begin{equation} 
\label{eq:limitCurveGenusBound2}
\begin{split} \text{Genus}_{\text{arith}}(\hat{C}_i^\ell, \widetilde{D}_i \cap \hat{C}_i^\ell) &\leq \text{Genus}_{\text{arith}}(\widetilde{C}_i, \widetilde{D}_i) \\
&\leq \rho.
\end{split}
\end{equation}

Combining (\ref{eq:limitCurveGenusBound1}) and (\ref{eq:limitCurveGenusBound2}) we deduce for every $\ell$ the bound
$$\text{Genus}_{\text{arith}}(\bar{C}^\ell, \bar{D} \cap \bar{C}^\ell) \leq \rho$$
where $\rho$ is independent of $\ell$. 

On the other hand, since the surfaces $\bar{C}^\ell$ exhaust $\bar{C}$, by definition
$$\text{Genus}_{\text{arith}}(\bar{C}, \bar{D}) \leq  \limsup_{\ell \to \infty}\text{Genus}_{\text{arith}}(\bar{C}^\ell, \bar{D} \cap \bar{C}^\ell) \leq \rho.$$

This proves the lemma. 
\end{proof}

It is also useful to note the following.

\begin{lem}
\label{lem:limitCurveNoBoundary} The surface $\bar C$ has empty boundary. The surface $\hat C$ from Proposition \ref{prop:limitCurveOmegaFinite} satisfies
$$\hat{w}(\hat{C}) \subset \{\bar{t}_{\text{trim}}\} \times M.$$
\end{lem}

\begin{proof}
Suppose for the sake of contradiction that $\bar C$ had nonempty boundary. Let $K \subset \widetilde{\mathbf{W}}_-$ be some compact subset with nonempty interior intersecting $\bar{w}(\partial \bar C)$. 

It follows by the fact that 
$$\mathbf{w}_i \to \bar{\mathbf{w}}$$
in the exhaustive Gromov sense that, for sufficiently large $i$, $\widetilde{C}_i$ must have boundary intersecting $K$ as well. 

However, by the second assertion of Proposition \ref{prop:sequenceOfCurves2}, the sequence of sets
$$w_i(\partial\widetilde{C}_i)$$
escape any compact subset of $\widetilde{W}_-$. Therefore, we arrive at a contradiction and $\bar C$ has empty boundary. The second assertion is an immediate consequence of this. 
\end{proof}

We continue by showing that the $\omega$-energy of the pseudoholomorphic curve $\bar{\mathbf{w}}$ is finite.

\begin{lem} \label{lem:limitCurveOmegaEnergy}
Let $\rho$ be the constant from Proposition \ref{prop:sequenceOfCurves2}. The $\omega$-energy
$$E_\omega(\bar{\mathbf{w}}) = \int_{\bar{\mathbf{w}}^{-1}([-2\epsilon_0, \infty) \times M)} \bar{\mathbf{w}}^*\omega$$
is bounded above by $\rho$. 
\end{lem}

\begin{proof}
The lemma is immediate from the bound in the final bullet of the third assertion of Proposition \ref{prop:sequenceOfCurves2} and the fact that 
$$\mathbf{w}_i \to \bar{\mathbf{w}}$$
in the exhaustive Gromov sense.
\end{proof}

Another immediate consequence of the exhaustive Gromov convergence
$$\mathbf{w}_i \to \bar{\mathbf{w}}$$
is a bound on the number of marked points $\#\bar{\mu}$. 

\begin{lem}
\label{lem:limitCurveMarkedPoints} Let $\rho$ be the constant from Proposition \ref{prop:sequenceOfCurves2}. Then
$$\#\mu \leq \rho.$$
\end{lem}

Lemma \ref{lem:limitCurveOmegaEnergy} implies that $\bar C$ has finitely many connected components on which $\bar{\mathbf{w}}$ is not constant. This is proved in the same way as Proposition \ref{prop:ncConnectedComponentsBound}. 

\begin{lem}
\label{lem:limitCurveBoundedNcComponents} There is some constant $\rho > 0$ such that $\bar C$ has at most $\rho$ connected components on which $\bar{\mathbf{w}}$ is not constant. 
\end{lem}

\begin{proof}
Write $\mathcal{C}_{\text{nc}}$ for the set of connected components of $\bar C$ on which $\bar{\mathbf{w}}$ is not constant. 

Write $\mathcal{C}_{\text{nc,in}}$ for the set of components in $\mathcal{C}_{\text{nc}}$ whose images are contained inside 
$$(50, \infty) \times M \subset \widetilde{W}_-.$$

Write $\mathcal{C}_{\text{nc,out}}$ for the set of components in $\mathcal{C}_{\text{nc}}$ whose images intersect 
$$\widetilde{W}_-^{50} \subset \widetilde{W}_-.$$

We begin by bounding $\#\mathcal{C}_{\text{nc,in}}$. Pick any component $\Sigma \in \mathcal{C}_{\text{nc,in}}$. By Lemma \ref{lem:limitCurveNoBoundary}, $\Sigma$ has no boundary. We then appeal to Theorem \ref{thm:omegaEnergyQuantization} with the parameters $a_0 = \inf(a \circ u_k)(\Sigma)$ and $h = 25$ to conclude that there is a constant $\hbar > 0$ such that
$$\int_\Sigma \bar{w}^*\omega \geq \hbar.$$

It follows that
$$\int_{\bar{w}^{-1}([50, \infty) \times M)} \bar{w}^*\omega \geq \hbar\#\mathcal{C}_{\text{nc,in}}.$$

On the other hand, Lemma \ref{lem:limitCurveOmegaEnergy} implies that there is a constant $\rho_1 > 0$ such that
$$\int_{\bar{w}^{-1}([50, \infty) \times M)} \bar{w}^*\omega \leq \rho_1.$$

We combine these two inequalities to conclude that
$$\#\mathcal{C}_{\text{nc,in}} \leq \hbar^{-1}\rho_2.$$

Next, we bound $\#\mathcal{C}_{\text{nc,out}}$. Pick any component $\Sigma \in \mathcal{C}_{\text{nc,out}}$. Let $\widetilde{\Sigma} \subset \Sigma$ be equal to the intersection of $\Sigma$ with $\bar{w}^{-1}(\widetilde{W}_-^{100}).$

Observe that, by our assumptions, $\bar{w}(\widetilde{\Sigma})$ contains a point with $g_-$-distance at least $50$ from the boundary of $\widetilde{W}_-^{100}$. 

It follows by Theorem \ref{thm:monotonicity} that there is a constant $\rho_2 \geq 1$ such that
$$\text{Area}_{\bar{w}^*g_-}(\widetilde{\Sigma}) \geq \rho_2^{-1}.$$

This implies
$$\text{Area}_{\bar{w}^*g_-}(\bar{w}^{-1}(\widetilde{W}_-^{100})) \geq \rho_2^{-1}\#\mathcal{C}_{k,\text{nc,out}}.$$

On the other hand, it is a consequence of Proposition \ref{prop:exhaustiveAreaBounds} and the fact that $\mathbf{w}_i$ converges to $\mathbf{w}$ in the exhaustive Gromov sense that
$$\text{Area}_{\bar{w}^*\bar{g}_-}(\bar{w}^{-1}(\widetilde{W}_-^{100}) \leq \rho_3$$
for some constant $\rho_3 > 0$. 

Combine the two inequalities above to deduce
$$\#\mathcal{C}_{\text{nc,out}} \leq \rho_2\rho_3$$
and as a result complete the proof of the proposition.
\end{proof}

We can now show that $\bar C$ has at least one non-compact connected component. 

\begin{lem} \label{lem:limitCurveNonCompactComponent}
The domain $\bar C$ of $\bar{\mathbf{w}}$ has at least one non-compact connected component.
\end{lem}

\begin{proof}
First, it is a consequence of Proposition \ref{prop:sequenceOfCurves} that, for any compact subset $K$ of $\widetilde{W}_-$, the image $w_i(\widetilde{C}_i)$ is not contained in $K$ for sufficiently large $i$. 

it is immediate from the exhaustive Gromov convergence
$$\mathbf{w}_i \to \bar{\mathbf{w}}$$
that the image $\bar{w}(\bar C)$ is not contained in any compact subset of $\widetilde{W}_-$. 

Let $\mathcal{C}_{\text{nc,compact}}$ denote the set of all compact components of $\bar C$ on which $\bar{w}$ is not constant. By Lemma \ref{lem:limitCurveBoundedNcComponents}, $\#\mathcal{C}_{\text{nc,compact}}$ is finite. Therefore, there is a compact subset $K$ of $\widetilde{W}_-$ which contains the images of all connected components in $\#\mathcal{C}_{\text{nc,compact}}$. 

The image $\bar{w}(\bar C)$ contains a countable set $\{p_k\}_{k \in \mathbb{N}}$ of distinct points that do not lie in $K$.

Recall from Lemma \ref{lem:exhaustiveGromovLimit} that there are finitely many connected components of $|\bar C|$. 

It follows that, after passing to a subsequence, we can ensure that there is a single connected component $\Sigma$ of $|\bar C|$ such that $p_k \in \bar{w}(\Sigma)$ for every $k$. 

It follows that the restriction of $\bar{w}$ to $\Sigma$ is not constant. Write $\mathcal{S}_{\text{nc}}$ for the set of connected components of $\Sigma$ on which $\bar{w}$ is not constant. 

Then
$$\bar{w}(\Sigma) = \cup_{S \in \mathcal{S}_{\text{nc}}} \bar{w}(S).$$

In particular, since $\bar{w}(\Sigma)$ is not contained in $K$, there is $S \in \mathcal{S}_{\text{nc}}$ that is not compact. This is the desired non-compact connected component of $\bar C$. 
\end{proof}

We put together the above results to prove Proposition \ref{prop:limitCurveOmegaFinite}. 

\begin{proof}[Proof of Proposition \ref{prop:limitCurveOmegaFinite}]

The proposition follows from the combination of Lemmas \ref{lem:exhaustiveGromovLimit}, \ref{lem:limitCurveProper}, \ref{lem:limitCurveNoBoundary}, \ref{lem:limitCurveOmegaEnergy}, and \ref{lem:limitCurveNonCompactComponent}. 

Note that the last lemma produces a non-compact connected component $S$ of $\bar C$ on which $\bar{w}$ is not constant. 

To get a non-compact connected component $\hat S$ of $\hat C$ on which $\hat{w}$ is not constant, it suffices to take any non-compact connected component of $S \cap \hat C$. 
\end{proof}

\subsection{The limit curve is feral}

Let $\hat{\mathbf{w}}$ be the pseudoholomorphic curve mapping into $[0, \infty) \times M$ from Proposition \ref{prop:limitCurveOmegaFinite}. 

We will now prove that the pseudoholomorphic curve $\hat{\mathbf{w}}$ is feral. Given Propositions \ref{prop:sequenceOfCurves} and \ref{prop:sequenceOfCurves2}, and Lemma \ref{lem:exhaustiveGromovLimit}, this proves both Theorems \ref{thm:mainExample} and \ref{thm:swExample}. 

The proof follows from essentially the same approach as the proof of \cite[Proposition $4.49$]{FishHoferFeral}. 

Observe from Lemmas \ref{lem:limitCurveGenusBound} and \ref{lem:limitCurveMarkedPoints} that we have already shown that $\hat{\mathbf{w}}$ has finite arithmetic genus and a finite number of marked points. 

We will build upon this by showing that the number of nodal points $\#\bar D$ and the number of connected components of $\bar C$ are finite.

The approach to bounding $\#\bar D$ is similar to the proof of Proposition \ref{prop:nodalPointBound}. We want to use the fact that $(\bar C, \bar D)$ has finite arithmetic genus along with the fact that $\#\pi_0(|\bar C|)$ is finite, as well as the stability condition, to bound $\#\bar D$. However, the formula from Lemma \ref{lem:arithGenusFormula} for the arithmetic genus only applies to \emph{compact} surfaces. 

The correct approach is to find some compact exhaustion
$$\{\bar C_k\}_{k \in \mathbb{N}}$$
of $\bar C$ and use the same idea as in Proposition \ref{prop:nodalPointBound} to obtain a $k$-independent bound on $\#(\bar D \cap \bar C_k)$. This would imply that $\#\bar D$ must be finite.

However, this approach requires a $k$-independent bound on the number of connected components of $\#\bar C_k$ on which $\bar{w}$ is non-constant. Such a bound should be possible using the approach of Proposition \ref{prop:ncConnectedComponentsBound}.

However, we cannot expect such a bound for any choice of compact exhaustion using the same approach as Proposition \ref{prop:ncConnectedComponentsBound}. The surface $\bar C_k$ may have components which are too small, that is they do not satisfy the conditions of Theorem \ref{thm:omegaEnergyQuantization}.

With some care, however, we can pick a compact exhaustion for which the methods of Propositions \ref{prop:ncConnectedComponentsBound} and \ref{prop:nodalPointBound} apply. The exhaustion we use is simpler than the one introduced in the proof of Lemma $4.52$ in \cite{FishHoferFeral}. One can attribute the additional complexity in \cite{FishHoferFeral} to the fact that they attempt to bound the number of connected components of $\bar{C}$ first rather than the number of nodal points $\#\bar{D}$. As a result, they select their exhaustion $\bar C_k$ so that $\#\pi_0(\bar C_k) \leq \#\pi_0(\bar{C})$ for every $k$, which requires more effort than we expend in the proof of Proposition \ref{prop:limitCurveNodalPoints} below. 

\begin{prop} \label{prop:limitCurveNodalPoints}
There is some constant $\rho > 0$ such that 
$$\#\bar{D} \leq \rho.$$
\end{prop}

\begin{proof}
Let $\mathcal{R}$ be the complement in the set of regular values of $a \circ \bar{w}$ of the discrete set $(a \circ \bar{w})(\bar{D} \cup \bar{\mu})$. 

For any $x \in \mathcal{R}$, write
$$\bar C_x = \bar{w}^{-1}(\widetilde{W}_-^x),$$
$$\bar D_x = \bar D \cap \bar C_x,$$
and
$$\bar \mu_x = \bar\mu \cap \bar C_x.$$

Write $\mathcal{C}_{x}$ for the set of connected components of $\bar C_x$. 

Now write $\mathcal{C}_{x,\text{good}}$ for the following set of connected components of $\bar C_x$. A component $\Sigma \in \mathcal{C}_x$ lies in $\mathcal{C}_{x,\text{good}}$ if and only if it satisfies one of the following conditions:
\begin{enumerate}
    \item $\Sigma$ has genus at least one.
    \item $\Sigma$ contains at least one marked point. 
    \item The infimum of $(a \circ \bar{w})(\Sigma)$ is less than $x - 1$.
\end{enumerate}

We will call elements of $\mathcal{C}_{x,\text{good}}$ \emph{good}. 

Write $\mathcal{C}_{x,\text{bad}}$ for the complement of $\mathcal{C}_{x,\text{good}}$, the set of \emph{bad} components. A component $\Sigma \in \mathcal{C}_x$ lies in $\mathcal{C}_{x,\text{bad}}$ if and only if it satisfies exactly one of the following conditions:
\begin{enumerate}
    \item $\Sigma$ is a sphere containing no marked points on which $\bar{w}$ is constant and
    $$(a \circ \bar{w})(\Sigma) \in (x - 1, x).$$
    \item $\Sigma$ is a compact, connected surface of genus zero which contains no marked points and is such that $\bar{w}$ is not constant on $\Sigma$ and 
    $$\inf (a \circ \bar{w})(\Sigma) \in (x-1, x).$$
\end{enumerate}

Write $\bar{C}_{x,\text{good}}$ for the surface constructed by taking the union of all the elements of $\mathcal{C}$

Write $\bar{D}_{x,\text{good}}$ for the set of pairs of nodal points in $\bar{D}_x$ which both lie in $\bar{C}_{x,\text{good}}$. 

Fix $x < y$ in $\mathcal{R}$. Then observe that
\begin{equation} \label{eq:limitCurveNodalPoints1} \bar{C}_{x,\text{good}} \subset \bar{C}_{y,\text{good}}. \end{equation}

If $\Sigma \in \mathcal{C}_{x,\text{good}}$, then write $\Sigma'$ for the component of $\bar{C}_y$ containing $\Sigma$. 

If $\Sigma$ has genus at least one or at least one marked point, then $\Sigma'$ does as well, so it is good. 

If $\inf(a \circ \bar{w})(\Sigma) < x - 1$, then observe by definition that
\begin{align*}
    \inf(a \circ \bar{w})(\Sigma') &\leq \inf(a \circ \bar{w})(\Sigma) \\
    &< x - 1 \\
    &< y - 1.
\end{align*}

It follows that $\Sigma'$ is good in this case as well, so we deduce the inclusion (\ref{eq:limitCurveNodalPoints1}). 

Next, we show
\begin{equation} \label{eq:limitCurveNodalPoints2} \bar{C} = \cup_{x \in \mathcal{R}} \bar{C}_{x,\text{good}}.\end{equation}

Fix any point $p \in \bar{C}$. Fix $x > (a \circ \bar{w})(p) + 2$. Then we can show $p \in \bar{C}_{x,\text{good}}$. Let $\Sigma_p$ be the connected component of $\bar{C}_x$ containing $p$. Then 
\begin{align*}
    \inf(a \circ \bar{w})(\Sigma_p) &\leq (a \circ \bar{w})(p) \\
    &< x - 1
\end{align*}
from which we conclude that $\Sigma_p$ is good, and so $p \in \bar{C}_{x,\text{good}}$. 

Another consequence of this argument is the identity
\begin{equation} \label{eq:limitCurveNodalPoints3} \bar{D} = \cup_{x \in \mathcal{R}} \bar{D}_{x,\text{good}}. \end{equation}

Let $(\underline{z}, \overline{z})$ be a pair of nodal points in $\bar{D}$. If we fix $x$ such that $x - 2$ is greater than $(a \circ \bar{w})(\underline{z}) = (a \circ \bar{w})(\overline{z})$, then the same argument is in the proof of (\ref{eq:limitCurveNodalPoints2}) shows that both $\underline{z}$ and $\overline{z}$ are contained in good components of $\bar{C}_x$. It follows that they lie in $\bar{D}_{x,\text{good}}$ by definition.

By definition we find
$$\text{Genus}_{\text{arith}}(\bar{C}_{x,\text{good}}, \bar{D}_{x,\text{good}}) \leq \text{Genus}_{\text{arith}}(\bar{C}_{y,\text{good}}, \bar{D}_{y,\text{good}})$$
and
$$\lim_{x \to \infty} \text{Genus}_{\text{arith}}(\bar{C}_{x,\text{good}}, \bar{D}_{x,\text{good}}) = \text{Genus}_{\text{arith}}(\bar{C}, \bar{D}).$$

Combining these two inequalities yields a constant $\rho > 0$ independent of $x$ so that
\begin{equation} \label{eq:limitCurveNodalPoints4} \text{Genus}_{\text{arith}}(\bar{C}_{x,\text{good}}, \bar{D}_{x,\text{good}}) \leq \rho. \end{equation}

We are now in a position to estimate $\#\bar{D}_{x,\text{good}}$ using the exact same method as Proposition \ref{prop:nodalPointBound}. 

The next step is to bound the number of connected components of $\bar{C}_{x,\text{good}}$ on which $\bar{w}$ is not constant. By Lemma \ref{lem:limitCurveMarkedPoints}, there is also an $x$-independent bound $\rho > 0$ for the number of marked points in $\bar{C}_{x,\text{good}}$. It follows that there is an $x$-independent bound $\rho > 0$ for the number of good components that have genus at least one or contain at least one marked point.

The remaining components are those components $\Sigma$ of $\bar C_x$ for which $\inf (a \circ \bar{w})(\Sigma)$ is less than $x - 1$. It is important to note that these components $\Sigma$ have the property that $(a \circ \bar{w})(\partial\Sigma) \subset \{x\}$. 

We can use the area bound from Proposition \ref{prop:exhaustiveAreaBounds} along with the monotonicity bound stated in Theorem \ref{thm:monotonicity} to deduce an $x$-independent bound on those components for which $\inf(a \circ \bar{w})$ is less than $1$. This, of course relies on the fact that these components have boundary in $\{x\} \times M$. 

For those components $\Sigma$ for which $\inf(a \circ \bar{w})$ lies in $[1, x - 1)$, the fact that $(a \circ \bar{w})(\partial\Sigma) \subset \{x\}$ allows us to use the $\omega$-energy quantization from Theorem \ref{thm:omegaEnergyQuantization} along with the $\omega$-energy bound from Lemma \ref{lem:limitCurveOmegaEnergy} to deduce an $x$-independent bound on the number of thse components as well. 

We conclude, if we set $\mathcal{C}_{x,\text{good,nc}}$ to be the components of $\bar{C}_{x,\text{good}}$ on which $\bar{w}$ is not constant, that there is a constant $\rho > 0$ independent of $x$ such that
\begin{equation} \label{eq:limitCurveNodalPoints5}
\#\mathcal{C}_{x,\text{good,nc}} \leq \rho.
\end{equation}

On the other hand, let $\mathcal{C}_{x, \text{good,const}}$ be the components of $\bar{C}_{x,\text{good}}$ on which $\bar{w}$ is constant. From (\ref{eq:limitCurveNodalPoints4}) and the marked points bound from Lemma \ref{lem:limitCurveMarkedPoints}, there is an $x$-independent bound $\rho > 0$ for the number of elements of $\mathcal{C}_{x,\text{good,const}}$ with genus at least one or at least one marked point. 

Any other element $\Sigma$ of $\mathcal{C}_{x,\text{good,const}}$ is a sphere with no marked points on which $\bar{w}$ is constant. Therefore, since $\bar{\mathbf{w}}$ is stable, it must contain at least three nodal points from $\bar{D}$. We caution that it is not immediate that these nodal points lie in $\bar{D}_{x,\text{good}}$, but it does follow from the following simple argument.

Let $\underline{z}$ be any nodal point in $\bar{D}$ lying on $\Sigma$. It is paired with another nodal point $\overline{z} \in \bar{D}$. Note that since $\Sigma$ is good, 
$$(a \circ \bar{w})(\overline{z}) = (a \circ \bar{w})(\underline{z}) < x - 1.$$

Write $\Sigma_{\overline{z}}$ for the unique component of $\bar{C}_x$ containing $\overline{z}$. By definition, 
\begin{align*}
    \inf (a \circ \bar{w})(\Sigma_{\overline{z}}) &\leq \inf(a \circ \bar{w})(\overline{z}) \\
    &< x - 1.
\end{align*}

It follows that $\Sigma_{\overline{z}}$ is good, so $\underline{z}$ and $\overline{z}$ both lie in $\bar{D}_{x,\text{good}}$. Anyways, we are able to conclude from this that there are at most $\frac{1}{3}\bar{D}_{x,\text{good}}$ components of $\mathcal{C}_{x,\text{good,const}}$ that are spheres with no marked points. 

This yields the bound
\begin{equation} \label{eq:limitCurveNodalPoints6}
    \#\mathcal{C}_{x,\text{good,const}} \leq \rho + \frac{1}{3}\bar{D}_{x,\text{good}}
\end{equation}
where $\rho > 0$ is independent of $x$. 

We now put this all together. Using the formula for the arithmetic genus from Lemma \ref{lem:arithGenusFormula}, we find for any $x \in \mathcal{R}$ that
$$\text{Genus}_{\text{arith}}(\bar{C}_{x,\text{good}}, \bar{D}_{x,\text{good}}) + \#\mathcal{C}_{x,\text{good,nc}} \geq \frac{1}{2}\#\bar{D}_{x,\text{good}} - \#\mathcal{C}_{x,\text{good,const}}.$$

Substituting in (\ref{eq:limitCurveNodalPoints4}), (\ref{eq:limitCurveNodalPoints5}) and (\ref{eq:limitCurveNodalPoints6}) shows that there is an $x$-independent constant $\rho > 0$ such that
$$\rho \geq \frac{1}{6}\#\bar{D}_{x,\text{good}}.$$

Using (\ref{eq:limitCurveNodalPoints3}) allows us to conclude
$$\#\bar{D} \leq \rho.$$
\end{proof}

We remark that Proposition \ref{prop:limitCurveNodalPoints} works for \emph{any} $\omega$-finite pseudoholomorphic curve with finite arithmetic genus and a finite number of marked points. 

\begin{cor}
\label{cor:omegaFiniteCurveNodalPoints} Let $(M, \eta = (\lambda, \omega))$ be a closed framed Hamiltonian manifold and let $(\mathbb{R} \times M, J, g)$ be an $\eta$-adapted cylinder over $M$. Let
$$\mathbf{u} = (u, C, j, \mathbb{R} \times M, J, D, \mu)$$
be an $\omega$-finite pseudoholomorphic curve such that
$$\text{Genus}_{\text{arith}}(C, D) < \infty$$
and
$$\#\mu < \infty.$$

Then $\#D < \infty$. 
\end{cor}

Proposition \ref{prop:limitCurveNodalPoints} allows us to immediately bound the number of components of $\bar C$ on which $\bar w$ is constant. 

\begin{prop}
\label{prop:limitCurveBoundedConstantComponents} There is a constant $\rho > 0$ such that $\bar C$ has at most $\rho$ connected components on which $\bar w$ is constant.
\end{prop}

\begin{proof}
A connected component $\Sigma$ of $\bar C$ on which $\bar w$ is constant satisfies at least one of the following conditions:
\begin{itemize}
    \item $\Sigma$ has positive genus.
    \item $\Sigma$ contains at least one marked point.
    \item $\Sigma$ contains at least three nodal points.
\end{itemize}

This is an immediate consequence of the fact that $\bar{\mathbf{w}}$ is stable. 

From Lemmas \ref{lem:limitCurveGenusBound} and \ref{lem:limitCurveMarkedPoints}, it follows that there is some constant $\rho_1 > 0$ such that at most $\rho_1$ components $\Sigma$ on which $\bar w$ is constant that satisfy the first or second items above. 

From Proposition \ref{prop:limitCurveNodalPoints}, it follows that there is some constant $\rho_2 > 0$ such that at most $\rho_2$ components $\Sigma$ on which $\bar w$ is constant that satisfy the third item above. 

We conclude that $\bar C$ has at most $\rho_1 + \rho_2$ components on which $\bar w$ is constant. 
\end{proof}

Given the bound from Lemma \ref{lem:limitCurveBoundedNcComponents}, we conclude that $\bar C$ has finitely many components. 

The final task is to prove that $\text{Punct}(\bar C)$ is finite. This relies on the fact that the domains $\widetilde{C}_i$ of the pseudoholomorphic curves $\mathbf{w}_i$ have a uniformly bounded number of boundary components, which we have not used so far. It seems implausible that one could prove $\text{Punct}(\bar C)$ is finite without such a bound. 

\begin{prop} \label{prop:limitCurveFinitePunctures}
$\text{Punct}(\bar C)$ is finite. 
\end{prop}

The method used in the proof of \cite[Proposition $4.49$]{FishHoferFeral} can readily be adapted to prove this proposition. 

The following lemma is an immediate consequence of \cite[Lemma $4.53$]{FishHoferFeral}.

\begin{lem} \cite[Lemma $4.53$]{FishHoferFeral}
\label{lem:impossibleSubmanifold} 
Let $(C, j, D, \mu)$ be a marked nodal Riemann surface with $$\text{Punct}(C) = \infty$$
and let
$$u: C \to \widetilde{W}_-$$
be a proper smooth map which maps pairs of nodal points to the same point in $\widetilde{W}_-$. 

Suppose further that $\text{Genus}(C)$, $\#D$, $\#\mu$, and $\#\pi_0(C)$ are finite. Then for any constant $\kappa_* \geq 1$ there is a compact surface $S \subset C$ satisfying the following properties:
\begin{itemize}
    \item $\#\pi_0(S) = \#\pi_0(C)$.
    \item $\#\pi_0(\partial S) \geq \kappa_*$.
    \item Every connected component of $C \setminus (S \setminus \partial S)$ is non-compact.
\end{itemize}
\end{lem}

\begin{proof}
Write $\mathcal{R}$ for the regular values of the map $a \circ u$. 

Choose for each connected component of $C$ a special point lying in that connected component. We can write this set of special points as a finite set
$$Z = \{z_1, \ldots, z_k\}$$
by our assumptions. 

For any $x \in \mathcal{R}$, write
$$C^x = (a \circ u)^{-1}(\widetilde{W}_-^x).$$

Our assumptions show that there exists $x_0 \in \mathcal{R}$ such that the following holds for any $x > x_0$ in $\mathcal{R}$:
\begin{itemize}
    \item $Z \cup D \subset C^x$.
    \item The genus of $C^x$ is equal to the genus of $C$.
    \item Each compact connected component of $C$ is contained in $C^x$. 
    \item The number of non-compact connected components of $C \setminus (C^x \setminus \partial C^x)$ is at least $\kappa_*$. 
\end{itemize}

Given this choice, the proof of Lemma $4.53$ in \cite{FishHoferFeral} follows essentially without modification. It does not require the smooth map $u$ to be pseudoholomorphic. 
\end{proof}

Given Lemma \ref{lem:impossibleSubmanifold}, Proposition \ref{prop:limitCurveFinitePunctures} follows by copying the proof of Proposition $4.49$ following the statement of Lemma $4.53$ in \cite{FishHoferFeral}. We include the argument below. 

This concludes the proof of Theorems \ref{thm:mainExample} and \ref{thm:swExample}. 

\begin{proof}[Proof of Proposition \ref{prop:limitCurveFinitePunctures}]

We recall that the pseudoholomorphic curve
$$\bar{\mathbf{w}} = (\bar w, \bar{C}, \bar{j}, \widetilde{W}_-, J_-, \bar{D}, \bar{\mu})$$
is an exhaustive Gromov limit of the curves
$$\mathbf{w}_i = (w_i, \widetilde{C}_i, \widetilde{j}_i, \widetilde{W}_-, J_-, \widetilde{D}_i, \widetilde{\mu}_i).$$

By Proposition \ref{prop:sequenceOfCurves2}, there is a constant $\rho \geq 1$ independent of $i$ such that 
\begin{itemize}
    \item $\text{Genus}_{\text{arith}}(\widetilde{C}_i, \widetilde{D}_i) \leq \rho$.
    \item $\#\widetilde{\mu}_i \leq \rho$.
    \item $\#\widetilde{D}_i \leq \rho$.
    \item $\#\pi_0(\widetilde{C}_i) \leq \rho$.
    \item $\#\pi_0(\partial\widetilde{C}_i) \leq \rho$. 
    \item $E_\omega(\mathbf{w}_i) \leq \rho.$
\end{itemize}

Suppose for the sake of contradiction that 
$$\text{Punct}(\bar C) = \infty.$$

Let $\hbar > 0$ be the constant from Theorem \ref{thm:omegaEnergyQuantization} with parameter $h = 1$.

Let $\kappa_*$ be a constant equal to $10^{100}(1 + \hbar^{-1})\rho$. 

By Lemmas \ref{lem:limitCurveGenusBound}, \ref{lem:limitCurveMarkedPoints}, \ref{lem:limitCurveBoundedNcComponents}, \ref{lem:impossibleSubmanifold} and Propositions \ref{prop:limitCurveOmegaFinite}, \ref{prop:limitCurveNodalPoints}, and \ref{prop:limitCurveBoundedConstantComponents}, there is a compact submanifold $\bar{S} \subset \bar{C}$ such that:
\begin{itemize}
    \item $\#\pi_0(\bar S) = \#\pi_0(\bar C)$.
    \item $\#\pi_0(\partial \bar S) \geq \kappa_*$.
    \item Every connected component of $\bar C \setminus (\bar S \setminus \partial \bar S)$ is non-compact.
\end{itemize}

In particular, the exhaustive Gromov compactness provides us with decorations $r$, $r_i$, and $\bar{r}$ on $\bar{S}$, $\widetilde{C}_i$, and $\bar{C}$, respectively along with embeddings
$$\phi_i: \bar{S}^{\bar{D} \cap \bar{S}, r} \hookrightarrow \bar{C}^{\bar{D}, \bar{r}} \hookrightarrow \widetilde{C}_i^{\widetilde{D}_i, r_i}$$
for sufficiently large $i$.

For any sufficiently large $i$, write
$$\Sigma_0 = \widetilde{C}_i^{\widetilde{D}_i, r_i},$$
$$\Sigma_1 = \phi_i(\bar{S}^{\bar{D} \cap \bar{S}, r}),$$
and
$$\Sigma_2 = \text{cl}(\Sigma_0 \setminus \Sigma_1) \subset \Sigma_0.$$

Recall that every connected component of $\bar{C} \setminus (\bar{S} \setminus \partial\bar{S})$ is non-compact. It follows by the definition of exhaustive Gromov convergence that, for sufficiently large $i$, any connected component $\Sigma'$ of $\Sigma_2$ with $\partial\Sigma' \subset \partial\Sigma_1$ satisfies
$$\sup (a \circ w_i)(\Sigma') - \sup(a \circ w_i)(\partial\Sigma') \geq 1.$$

Let $\hbar > 0$ be the constant from Theorem \ref{thm:omegaEnergyQuantization} with parameter $h = 1$. 

It then follows from Theorem \ref{thm:omegaEnergyQuantization} that, for any connected component $\Sigma'$ of $\Sigma_2$ with $\partial\Sigma' \subset \partial\Sigma_1$ on which $w_i$ is nonconstant,
\begin{equation} \label{eq:limitCurveFinitePunctures1}  \int_{\Sigma'} w_i^*\omega \geq \hbar.
\end{equation}

Partition $\Sigma_2$ into three disjoint sub-surfaces. Let $\Sigma_{2,\text{const}}$ denote the union of all connected components of $\Sigma_2$ on which $w_i$ is constant. Let $\Sigma_{2,\text{int}}$ denote the union of all connected components of $\Sigma_2$ with boundary lying inside $\partial\Sigma_1$. Let $\Sigma_{2, \text{bdry}}$ denote the union of all connected components of $\Sigma_2$ which have nonempty intersection with $\partial\Sigma_0$. 

Using (\ref{eq:limitCurveFinitePunctures1}) and the fact that
$$E_\omega(\mathbf{w}_i) \leq \rho$$
for every $i$, we find
\begin{equation}
    \label{eq:limitCurveFinitePunctures2} \#\pi_0(\Sigma_{2,\text{int}}) \leq \hbar^{-1}\rho.
\end{equation}

Any component in $\Sigma_{2, \text{bdry}}$ must contain at least one component of $\partial\Sigma_0$. Recall that $$\#\pi_0(\partial\Sigma_0) = \#\pi_0(\widetilde{C}_i) \leq \rho.$$

We conclude the bound
\begin{equation}
    \label{eq:limitCurveFinitePunctures3} \#\pi_0(\Sigma_{2,\text{bdry}}) \leq \rho.
\end{equation}

It remains to estimate $\#\pi_0(\Sigma_{2,\text{const}})$. Observe that there are at most $2\rho$ marked and nodal points total in $\Sigma_2$. 

There are at most $\rho$ components of $\Sigma_{2,\text{const}}$ which have genus greater than $1$. Any other component contains a marked or nodal point since the pseudoholomorphic curves $\mathbf{w}_i$ are stable. 

It follows from this that
\begin{equation}
    \label{eq:limitCurveFinitePunctures4}
    \#\pi_0(\Sigma_{2,\text{const}}) \leq 3\rho
\end{equation}
and therefore
\begin{equation}
    \label{eq:limitCurveFinitePunctures5}
    \#\pi_0(\Sigma_2) \leq (4 + \hbar^{-1})\rho.
\end{equation}

The Mayer-Vietoris sequence shows that
$$\chi(\Sigma_1) + \chi(\Sigma_2) = \chi(\Sigma).$$

We expand, multiply by $-1$, rearrange, and discard nonpositive terms to deduce the inequality
\begin{equation}
    \label{eq:limitCurveFinitePunctures6}
    \begin{split}
    \#\pi_0(\partial\Sigma_1) &\leq 2\pi_0(\Sigma_1) + 2\pi_0(\Sigma_2) + 2\text{Genus}(\Sigma) + \pi_0(\partial\Sigma) \\
    &\leq 2\pi_0(\bar{C}) + 2\pi_0(\Sigma_2) + 2\text{Genus}(\Sigma) + \pi_0(\partial\Sigma) \\
    &\leq 3\rho + 2\pi_0(\bar{C}) + 2\pi_0(\Sigma_2) \\
    &\leq 5\rho + 2\pi_0(\Sigma_2) \\
    &\leq (13 + \hbar^{-1})\rho.
    \end{split}
\end{equation}

The second inequality uses the fact that $\bar S$ has the same number of connected components as $\bar C$ and $\#\pi_0(\Sigma_1) \leq \#\pi_0(\bar S)$. The third inequality uses the fact that the genus of $\Sigma$ and the number of boundary components of $\Sigma$ are each at most $\rho$. The fourth inequality uses the bound on $\#\pi_0(\bar C)$ from Lemma \ref{lem:limitCurveBoundedNcComponents} and Proposition \ref{prop:limitCurveBoundedConstantComponents}. The final inequality uses (\ref{eq:limitCurveFinitePunctures5}).

With this we have shown
$$10^{100}(1 + \hbar^{-1})\rho \leq \#\pi_0(\partial\Sigma_1) \leq (13 + \hbar^{-1})\rho.$$

This cannot hold, so our original assumption was incorrect and $\text{Punct}(\bar{C})$ must be finite. 
\end{proof}

\bibliographystyle{halpha}
\bibliography{main}

\end{document}